\documentclass[11pt]{article}

\usepackage[letterpaper, hmargin=1.3in, top=1in, bottom=1.3in, footskip=1in]{geometry}
\usepackage{float}
\usepackage{titlesec}
\titleformat{\subsection}[runin]{\normalfont\bfseries}{\thesubsection.}{.5em}{}[.]\titlespacing{\subsection}{0pt}{2ex plus .1ex minus .2ex}{.8em}
\titleformat{\subsubsection}[runin]{\normalfont\itshape}{\thesubsubsection.}{.3em}{}[.]\titlespacing{\subsubsection}{0pt}{1ex plus .1ex minus .2ex}{.5em}

\usepackage[labelfont=sc,font=small,labelsep=period]{caption}
\usepackage{amsmath} 
\usepackage{amssymb}
\usepackage{amsfonts}
\usepackage{latexsym}
\usepackage{amsthm}
\usepackage{amsxtra}
\usepackage{amscd}
\usepackage{bbm}
\usepackage{booktabs} 
\usepackage{mathrsfs}
\usepackage{bm}
\usepackage{xcolor}
\usepackage{xfrac}
\usepackage{subcaption}
\usepackage{threeparttable}
\usepackage[toc]{appendix}
\usepackage[utf8]{inputenc}
\usepackage[shortlabels]{enumitem}

\usepackage{graphicx, color}

\definecolor{darkred}{rgb}{0.9,0,0.3}
\definecolor{darkblue}{rgb}{0,0.3,0.9}

\usepackage{ifthen}
\def\comment#1{\ifthenelse{\isodd{\value{page}}}{\marginpar{\raggedright\scriptsize{\textcolor{darkred}{#1}}}}{\marginpar{\raggedleft\scriptsize{\textcolor{darkred}{#1}}}}}  

\usepackage[nottoc,notlof,notlot]{tocbibind}
\usepackage{cite} 
\usepackage{yhmath}

\usepackage{hyperref}
\hypersetup{
  pdftitle={The shifted ODE method for underdamped Langevin MCMC},
  pdfauthor={James Foster},
  colorlinks=false,
}

\numberwithin{equation}{section}
\numberwithin{figure}{section}

\renewcommand{\labelitemi}{\raisebox{0.2ex}{$\scriptstyle{\blacktriangleright}$}}

\theoremstyle{plain} 
\newtheorem{theorem}{Theorem}[section]
\newtheorem*{theorem*}{Theorem}

\newtheorem*{lemma*}{Lemma}
\newtheorem{corollary}[theorem]{Corollary}
\newtheorem*{corollary*}{Corollary}

\newtheorem*{proposition*}{Proposition}
\newtheorem{definition}[theorem]{Definition}
\newtheorem*{definition*}{Definition}

\newtheorem*{conjecture*}{Conjecture}

\theoremstyle{definition} 

\newtheorem*{example*}{Example}
\newtheorem{remark}[theorem]{Remark}
\newtheorem*{remark*}{Remark}

\newcommand{\m}{\hspace{0.25mm}}
\newcommand{\mm}{\hspace{5mm}}
\newcommand{\mmm}{\hspace{10mm}}
\newcommand{\mmmm}{\hspace{15mm}}
\newcommand{\mmmmm}{\hspace{20mm}}
\newcommand{\nm}{\hspace{-1.5mm}}
\newcommand{\nnm}{\hspace{-1mm}}

\renewcommand{\P}{\mathbb{P}}
\newcommand{\E}{\mathbb{E}}
\newcommand{\F}{\mathcal{F}}
\newcommand{\R}{\mathbb{R}}

\renewcommand{\L}{\mathbb{L}}
\newcommand{\T}{\mathsf{T}}
\newcommand{\h}{\mathcal{H}}
\newcommand{\1}{\mathbbm{1}}
\newcommand{\cubed}{\mysqrt{0}{3}{3}{d\m}}
\newcommand{\wx}{\wideparen{x}^{\m n}}
\newcommand{\wv}{\wideparen{v}^{\m n}}

\newcommand{\var}{\operatorname{Var}}

\newcommand{\fourthree}{\sqrt[\leftroot{1}\uproot{3}4]{3}}
\newcommand*{\mysqrt}[4]{\sqrt[\leftroot{#1}\uproot{#2}#3]{#4}}
\newcommand{\sbt}{\,\begin{picture}(-1,1)(-1,-3)\circle*{2}\end{picture}\ }

\newcommand{\Addresses}{{
  J.~Foster, T.~Lyons and H.~Oberhauser\\[-6pt]\indent\footnotesize
  \textsc{Mathematical Institute, University of Oxford, Oxford, OX2 6GG, UK.}\\[-6pt]\indent
  \textsc{The Alan Turing Institute, British Library, London, NW1 2DB, UK.}\\[-6pt]\indent
  \textit{Email:} \texttt{\{james.foster, terry.lyons, harald.oberhauser\}@\hspace{0.25mm}maths.ox.ac.uk}.
}}

\title{The shifted ODE method for underdamped Langevin MCMC}
\author{James Foster, Terry Lyons and Harald Oberhauser}
\date{}

\begin{document}

\maketitle

\begin{abstract}
In this paper, we consider the underdamped Langevin diffusion (ULD) and propose
a numerical approximation  using its associated ordinary differential equation (ODE).
When used as a Markov Chain Monte Carlo (MCMC) algorithm, we show that the
ODE approximation achieves a $2\m$-Wasserstein error of $\varepsilon$ in $\mathcal{O}\big(\cubed/\varepsilon^{\frac{2}{3}}\big)$ steps under the standard smoothness and strong convexity assumptions on the target distribution. 
This matches the complexity of the randomized midpoint method proposed by
Shen and Lee [NeurIPS 2019] which was shown to be order optimal by Cao, Lu and Wang.
However, the main feature of the proposed numerical method is that it can utilize
additional smoothness of the target log-density $f$. More concretely, we show that the
ODE approximation achieves a $2\m$-Wasserstein error of $\varepsilon$ in $\mathcal{O}\big(d^{\frac{2}{5}}/\varepsilon^{\frac{2}{5}}\big)$ and $\mathcal{O}\big(\sqrt{d\m}/\varepsilon^{\frac{1}{3}}\big)$
steps when Lipschitz continuity is assumed for the Hessian and third derivative of $f$.
By discretizing this ODE using a third order Runge-Kutta method, we can obtain a practical MCMC method that uses just two additional gradient evaluations per step. In our experiment, where the target comes from a logistic regression, this method shows faster convergence compared to other unadjusted Langevin MCMC algorithms.
\end{abstract}

\section{Introduction}

Traditionally used to model coarse-grained particle systems \cite{CoarseGrain1, CoarseGrain2, BadLangevin}, the underdamped Langevin diffusion (ULD) is given by the following stochastic differential equation (SDE):
\begin{align}\label{eq:ULD}
dx_t & = v_t\,dt,\\[3pt]
dv_t & = -\gamma v_t\,dt - u\nabla f(x_t)\,dt + \sqrt{2\gamma u}\, dW_t\m,\nonumber
\end{align}
where $x,v\in\R^d$ represent the position and momentum of a particle,
$f : \R^d\rightarrow \R$ denotes a scalar potential, $\gamma > 0$ is a friction coefficient and $W$ is a $d$-dimensional Brownian motion.\medbreak

Under mild assumptions, it is known that (\ref{eq:ULD}) admits a unique strong solution which is ergodic with respect to the Gibbs measure with Hamiltonian $H(x, v) = f(x) + \frac{1}{2u}\|v\|^2$
(we refer the reader to Proposition 6.1 in \cite{InvariantExists} for further details). That is, the diffusion process has a stationary measure $\pi$ on $\R^{2d}$ with density proportional to $\exp(-H(x, v))$.
In particular, the (marginal) stationary distribution for $x$ is proportional to $\exp(-f(x))$
and thus solving (\ref{eq:ULD}) is a method of sampling from unnormalized log-concave densities.
When the SDE solution cannot be obtained exactly, it must be appropriately discretized
and so numerical methods are required to apply ULD to log-concave sampling problems.
In this paper, we develop an ordinary differential equation (ODE) based method for ULD.
\medbreak

Langevin-type equations are widely used in statistical physics as they are natural extensions of Hamiltonian systems that can model randomness \cite{BadLangevin, MilsteinBook}. More recently, they have been applied to sampling and optimization problems in machine learning \cite{BayesSGLD1, BayesSGLD2, StochasticHMC1}.
As a result, the study of (underdamped) Langevin dynamics is an active area of research
with an extensive literature. Below we give a brief overview for some of this related work:\medbreak

\textbf{Numerical methods for underdamped Langevin dynamics}\smallbreak

A variety of approximations have been proposed for ULD. Some noteworthy examples are
the unadjusted and Metropolis-adjusted OBABO schemes \cite{OBABOphysics, OBABOtheory, HamsMC}, the randomized midpoint method \cite{MidpointMCMC, MidpointMCMCFollowup} (which was shown to be order optimal in \cite{ULDComplexity}), the left-point\footnote{In \cite{ChengMCMC}, the proposed approximation is referred to as the ``discrete underdamped Langevin diffusion''. We shall refer to this approach as the ``left-point method'' (simply to distinguish it from other methods).} method \cite{ChengMCMC, KineticLangevinMCMC, ChengMCMCfollowup}, Strang splitting \cite{IreneStrang, BadStrang} and quasi-symplectic Runge-Kutta methods \cite{MilsteinTretyakov}.
We shall compare our approach to several of these methods in our numerical experiment.\medbreak

\textbf{Probabilistic couplings for underdamped Langevin dynamics}\smallbreak

The analysis of ULD and its discretizations often makes use of probabilistic couplings.
In \cite{LangevinContraction}, the authors used reflection and synchronous couplings to establish a (Wasserstein) contraction property for the diffusion. Similar contraction theorems were later obtained in \cite{ChengMCMC, KineticLangevinMCMC} where only synchronous couplings were used. The exponential contractivity of Langevin dynamics allows one to obtain non-asymptotic estimates between the diffusion and approximation processes directly from local mean and mean-square estimates \cite{RungeKuttaMCMC}.
Contraction results have also been established for the OBABO discretization of (\ref{eq:ULD}) using synchronous and Gaussian couplings \cite{OBABOtheory}. However, unlike in \cite{RungeKuttaMCMC}, the author of \cite{OBABOtheory} obtains non-asymptotic estimates without comparing to the SDE solution (instead the OBABO chain is compared to another Markov chain with the desired stationary measure).
In \cite{CouplingForSDEs}, the authors propose the use of couplings to empirically estimate the convergence of a numerical solution to the stationary measure of an ergodic SDE. For example, they use a combination of reflection, synchronous and maximal couplings to estimate the convergence of an Euler-Maruyama discretization of ULD with nonconvex potential $f(x) = \big(\|x\|_2^2 - 1\big)^2$.\medbreak

\textbf{Applications of Langevin dynamics to MCMC algorithms and optimization}\smallbreak

After the breakthrough papers of Neal \cite{NealMC} along with Girolami and Calderhead \cite{GirolamiMC}, there has significant interest in Hamiltonian Monte Carlo (HMC) methods for sampling and integral estimation. Since these MCMC methods are based on Hamiltonian systems, they are intimately connected to Langevin dynamics \cite{HybridMC, StochasticHMC1, StochasticHMC2, HamsMC, RandomizedHMC}. More recently, the Metropolis-adjusted OBABO sampler (an MCMC algorithm which converges to ULD \cite{OBABOtheory})
was shown to empirically outperform the standard HMC method \cite{HamsMC}. Alongside MCMC \cite{RandomULD, HamsMC, ULDVAE},
underdamped Langevin diffusions have also been considered for optimization due to their natural parallels with existing momentum-based algorithms \cite{SGHMCconvergence1, SGHMCconvergence2, NesterovMCMC, ULDforNeuralNets, FractionalULD}.

\subsection{Contributions}
Our key contribution is an ODE-based numerical method for ULD, which we refer to as the ``shifted ODE method''. To the best of our knowledge, this is the first derivative-free approximation of ULD that can achieve third order convergence
(a third order numerical method is given in \cite{MilsteinTretyakov} but it requires further derivatives of $\nabla f$).

\begin{definition}[\textbf{Shifted ODE method}]\label{def:shifted_ode} Let $\{t_n\}_{n\m\geq\m 0}$ be a sequence of times with $t_0 = 0$,
$t_{n+1} > t_{n}$ and step sizes $h_n = t_{n+1} - t_n\m$, for $n\geq 0$. We construct a numerical solution
$\{(\widetilde{x}_n\m, \widetilde{v}_n)\}_{n\m\geq\m 0}$ by setting $(\widetilde{x}_0\m, \widetilde{v}_0) := (x_0\m, v_0)$
and for each $n\geq 0$, defining $(\widetilde{x}_{n+1}\m, \widetilde{v}_{n+1})$ as
\begin{align*}
\Bigg(\begin{matrix} \,\widetilde{x}_{n+1}\\[-3pt] \,\widetilde{v}_{n+1}\end{matrix}\Bigg)
:=
\Bigg(\begin{matrix} \,\overline{x}_1^{\m n}\\[-3pt] \,\overline{v}_1^{\m n}\end{matrix}\Bigg)
- \big(H_n - 6 K_n\big)\Bigg(\begin{matrix} 0\\[-3pt] \sqrt{2\gamma u}\end{matrix}\,\Bigg),
\end{align*}
where $\{(\overline{x}_t^{\m n}, \overline{v}_t^{\m n})\}_{t\m\in\m [0,1]}$ solves the following (rescaled) Langevin-type ODE,
\begin{align}\label{eq:linULE}
\frac{d}{dt}\Bigg(\begin{matrix} \,\overline{x}^{\m n}\\[-4pt] \,\overline{v}^{\m n}\end{matrix}\Bigg)
=
\Bigg(\begin{matrix} \,\overline{v}^{\m n}\\[-4pt] \,-\gamma \m\overline{v}^{\m n} - u\nabla f\big(\,\overline{x}^{\m n}\big)\end{matrix}\,\Bigg)\m h_n
+ \big(W_n - 12 K_n\big)\Bigg(\begin{matrix} 0\\[-4pt] \sqrt{2\gamma u}\end{matrix}\,\Bigg),
\end{align}
with initial condition
\begin{align*}
\Bigg(\begin{matrix}\, \overline{x}_{0}^{\m n}\\[-3pt]\, \overline{v}_{0}^{\m n}\end{matrix}\Bigg)
:= \Bigg(\begin{matrix}\,\widetilde{x}_{n}\\[-3pt] \,\widetilde{v}_{n}\end{matrix}\Bigg) + \big(H_n + 6K_n\big) \Bigg(\begin{matrix} 0 \\[-3pt] \sqrt{2\gamma u}\end{matrix}\,\Bigg),
\end{align*}
and the random Gaussian vectors $\{W_n\}_{n\m\geq\m 0}\m,\,\{H_n\}_{n\m\geq\m 0}$ and $\{K_n\}_{n\m\geq\m 0}$ are independent with
$W_n\sim\mathcal{N}\big(0, h_n\m I_d\big)\m,\, H_n\sim\mathcal{N}\big(0,\frac{1}{12}h_n\m I_d\big)$ and $K_n\sim\mathcal{N}\big(0,\frac{1}{720}\m h_n\m I_d\big)$ for all $n\geq 0$. Moreover, 
we define $\{\m W_n\m,\m H_n\m,\m K_n\m\}$ from the same Brownian motion $W$ (see Theorem \ref{thm:hk_st} for details).
\begin{align*}
W_n & := W_{t_{n+1}} - W_{t_n}\m,\\[2pt]
H_n & := \frac{1}{h_n}\int_{t_n}^{t_{n+1}}\nnm\Big(\big(W_t - W_{t_n}\big) - \frac{t - t_n}{h_n}\m W_n\Big)\,dt,\\[1pt]
K_n & := \frac{1}{h_n^2}\int_{t_n}^{t_{n+1}}\nnm\bigg(\frac{1}{2}h_n - (t-t_n)\bigg)\Big(\big(W_{t} - W_{t_{n+1}}\big)  - \frac{t - t_n}{h_n}\m W_n\Big)\,dt.
\end{align*}
\end{definition}\bigbreak

We discuss the intuition behind this method in Section \ref{sect:main_results} and perform an error analysis in the appendices. In Table \ref{table:ode_convergence}, we compare this ODE approach to some related methods:

\begin{threeparttable}[h]
  \caption{Summary of complexities for ULD methods (with respect to $2$\hspace{0.125mm}-Wasserstein error).\vspace{-3mm}}\vspace{-3mm}
  \label{table:ode_convergence}
  \centering
  \begin{tabular}{llc}
    \toprule
    Numerical method & Smoothness assumptions & Number of steps $n$ to achieve\\[-3pt]
    & on the strongly convex $f$ & \hspace*{-2mm}an error of $W_2\big(\m\widetilde{x}_n\m , e^{-f}\m\big) \leq \varepsilon$\\
    \midrule
    Shifted ODE method & Lipschitz Gradient  & $\mathcal{O}\big(\m\cubed/\varepsilon^{\frac{2}{3}}\m\big)$\tnote{*} \\
     (this paper) & \, + Lipschitz Hessian & $\mathcal{O}\big(\m d^{\frac{2}{5}}/\varepsilon^{\frac{2}{5}}\m\big)$\tnote{*}\\
     & \, + Lipschitz third derivative & $\mathcal{O}\big(\m\sqrt{d\m}/\varepsilon^{\frac{1}{3}}\m\big)$\\
     \midrule
     Strang splitting \cite{IreneStrang} & Lipschitz Gradient  & $\mathcal{O}\big(\sqrt{d\m}/\varepsilon\m\big)$\tnote{$\dagger$} \\
     & \, + Lipschitz Hessian & $\mathcal{O}\big(\sqrt{d\m}/\hspace{-0.25mm}\sqrt{\varepsilon}\m\big)$\tnote{$\dagger$}\\
     \midrule
     OBABO scheme \cite{OBABOtheory} & Lipschitz Gradient  & $\mathcal{O}\big(\sqrt{d\m}/\varepsilon\m\big)$ \\
           & \, + Lipschitz Hessian & $\mathcal{O}\big(\sqrt{d\m}/\hspace{-0.25mm}\sqrt{\varepsilon}\m\big)$\\
     \midrule
     Randomized midpoint & Lipschitz Gradient  & $\mathcal{O}\big(\m\cubed/\varepsilon^{\frac{2}{3}}\m\big)$ \\[-3pt]
     method \cite{MidpointMCMC} &  & \\
     \midrule
     Left-point method \cite{ChengMCMC} & Lipschitz Gradient  & $\mathcal{O}\big(\m\sqrt{d\m}/\varepsilon\big)$ \\ 
    \bottomrule
  \end{tabular}
    \begin{tablenotes}
      \footnotesize
      \item[$\ast$]We expect these complexities to be reduced when the ODE is discretized using standard high order ODE solvers. In this case, we expect the complexities to be $\mathcal{O}\big(\sqrt{d\m}/\varepsilon\m\big)$ and $\mathcal{O}\big(\sqrt{d\m}/\hspace{-0.125mm}\sqrt{\varepsilon}\m\big)$ respectively.
      On the other hand, if a randomized ODE solver is applied (such as a randomized midpoint method),
      we conjecture that it is possible to achieve these complexities. This may be a topic for future work.
      \item[$\dagger$]These orders are conjectured, but seem quite likely given the method's stochastic Taylor expansion.
    \end{tablenotes}
  \end{threeparttable}

In practice, the ODE (\ref{eq:linULE}) cannot be solved exactly, so must be discretized using an appropriate ODE solver. For the resulting method to demonstrate third order convergence, we expect (but have not proven) that it suffices to use a third order numerical ODE solver.
Therefore we consider two approaches which are derived from the standard third order Runge-Kutta method \cite{Butcher} and the fourth order splitting method of Forest and Ruth \cite{FourthOrderSplitting}.
In order to distinguish them, we refer to these discretizations of ULD as the SORT\footnote{\textbf{S}hifted \textbf{O}DE with \textbf{R}unge-Kutta-\textbf{T}hree.} and SOFA\footnote{\textbf{S}hifted \textbf{O}DE with \textbf{F}ourth-order-splitting \textbf{A}pplied.} methods. Due to the structure of the ``shifted ODE'', the SORT and SOFA methods only require two and three additional evaluations of $\nabla f$ per step respectively.
This means the SORT and SOFA methods have similar cost to the methods in Table \ref{table:ode_convergence},
which use one extra evaluation per step (except the randomized midpoint which uses two).
However by using more evaluations of $\nabla f$ per step, the SORT and SOFA methods can exhibit faster
convergence when $\nabla f$ has more smoothness (see the numerical experiment).

\begin{definition}[\textbf{The SORT method}]\label{def:sort_method} Let $\{t_n\}_{n\m\geq\m 0}$ be a sequence of times with $t_0 = 0$,
$t_{n+1} > t_{n}$ and step sizes $h_n = t_{n+1} - t_n\m$. We construct a numerical solution
$\{(\m\overrightarrow{x\,}_{\nnm n}\m, \overrightarrow{v\,}_{\nnm n})\}_{n\m\geq\m 0}$ for (\ref{eq:ULD}) by setting $\m(\m\overrightarrow{x\,}_{\nnm 0}\m, \overrightarrow{v\,}_{\nnm 0})\m :=\m (x_0\m, v_0)\m$
and for $n\m\geq\m 0$, defining $\m(\m\overrightarrow{x\,}_{\nnm n+1}\m, \overrightarrow{v\,}_{\nnm n+1}\m)\m$ as follows:
\begin{align*}
\overrightarrow{v\,}_{\nnm n}^{\m(1)} & := \overrightarrow{v\,}_{\nnm n} + \sqrt{2\gamma u}\m\big(H_n + 6K_n\big),\\[3pt]
\overrightarrow{x\,}_{\nnm n}^{\m(1)} & := \overrightarrow{x\,}_{\nnm n} + \bigg(\frac{1 - e^{-\frac{1}{2}\gamma h_n}}{\gamma}\bigg)\overrightarrow{v\,}_{\nnm n}^{\m(1)} - \bigg(\frac{e^{-\frac{1}{2}\gamma h_n} + \frac{1}{2}\gamma h_n - 1}{\gamma^2}\bigg)u\nabla f\big(\m\overrightarrow{x\,}_{\nnm n}\big)\\
&\hspace{11mm} + \bigg(\frac{e^{-\frac{1}{2}\gamma h_n} + \frac{1}{2}\gamma h_n - 1}{\gamma^2 h_n}\bigg)\sqrt{2\gamma u}\m\big(W_n - 12K_n\big)\m,\\[3pt]
\overrightarrow{x\,}_{\nnm n+1} & := \overrightarrow{x\,}_{\nnm n} + \bigg(\frac{1 - e^{-\gamma h_n}}{\gamma}\bigg)\overrightarrow{v\,}_{\nnm n}^{\m(1)} - \bigg(\frac{e^{-\gamma h_n} + \gamma h_n - 1}{\gamma^2}\bigg)\bigg(\frac{1}{3} u\nabla f\big(\m\overrightarrow{x\,}_{\nnm n}\big) + \frac{2}{3} u\nabla f\big(\m\overrightarrow{x\,}_{\nnm n}^{\m(1)}\big)\bigg)\\
&\hspace{11mm} + \bigg(\frac{e^{-\gamma h_n} + \gamma h_n - 1}{\gamma^2 h_n}\bigg)\sqrt{2\gamma u}\m\big(W_n - 12K_n\big)\m,\\[3pt]
\overrightarrow{v\,}_{\nnm n}^{\m(2)} & := e^{-\gamma h_n}\overrightarrow{v\,}_{\nnm n}^{\m(1)} - \frac{1}{6}\m e^{-\gamma h_n} u\nabla f\big(\m\overrightarrow{x\,}_{\nnm n}\big)h_n - \frac{2}{3}\m e^{-\frac{1}{2}\gamma h_n} u\nabla f\big(\m\overrightarrow{x\,}_{\nnm n}^{\m(1)}\big)h_n - \frac{1}{6}u\nabla f\big(\m\overrightarrow{x\,}_{\nnm n + 1}\big)h_n\\
&\hspace{11mm} + \bigg(\frac{1 - e^{-\gamma h_n}}{\gamma h_n}\bigg)\sqrt{2\gamma u}\m\big(W_n - 12K_n\big),\\[3pt]
\overrightarrow{v\,}_{\nnm n + 1} & := \overrightarrow{v\,}_{\nnm n}^{\m(2)} - \sqrt{2\gamma u}\m\big(H_n - 6K_n\big),
\end{align*}
where $\{W_n, H_n, K_n\}_{n\m\geq\m 0}$ are the Gaussian random vectors given in definition \ref{def:shifted_ode}.
\end{definition}\medbreak

\begin{definition}[\textbf{The SOFA method}]\label{def:sofa_method} Let $\{t_n\}_{n\m\geq\m 0}$ be a sequence of times with $t_0 = 0$,
$t_{n+1} > t_{n}$ and step sizes $h_n = t_{n+1} - t_n\m$. We construct a numerical solution
$\{(\m\overrightarrow{x\,}_{\nnm n}\m, \overrightarrow{v\,}_{\nnm n})\}_{n\m\geq\m 0}$ for (\ref{eq:ULD}) by setting $\m(\m\overrightarrow{x\,}_{\nnm 0}\m, \overrightarrow{v\,}_{\nnm 0})\m :=\m (x_0\m, v_0)\m$
and for $n\m\geq\m 0$, defining $\m(\m\overrightarrow{x\,}_{\nnm n+1}\m, \overrightarrow{v\,}_{\nnm n+1}\m)\m$ as follows:
\begin{align*}
\overrightarrow{v\,}_{\nnm n}^{\m(0)} & := \overrightarrow{v\,}_{\nnm n} + \sqrt{2\gamma u}\m\big(H_n + 6K_n\big),\\
\overrightarrow{v\,}_{\nnm n}^{\m(1)} & := e^{-\gamma(\frac{1}{2}+\phi)h_n}\overrightarrow{v\,}_{\nnm n}^{\m(0)} + \Big(\hspace{-0.25mm}-u\nabla f\big(\m\overrightarrow{x\,}_{\nnm n}\big)\m h_n + \sqrt{2\gamma u}\m\big(W_n - 12K_n\big)\Big)\bigg(\frac{1 - e^{-\gamma\left(\frac{1}{2}+\phi\right)h_n}}{\gamma h_n}\bigg)\m,\\
\overrightarrow{x\,}_{\nnm n}^{\m(1)} & := \overrightarrow{x\,}_{\nnm n} + \overrightarrow{v\,}_{\nnm n}^{\m(1)}\big(1+2\phi\big)h_n\m,\\
\overrightarrow{v\,}_{\nnm n}^{\m(2)} & := e^{\gamma \phi h_n}\overrightarrow{v\,}_{\nnm n}^{\m(1)} + \Big(\hspace{-0.25mm}-u\nabla f\hspace{-0.5mm}\left(\overrightarrow{x\,}_{\nnm n}^{\m(1)}\right)\hspace{-0.125mm} h_n + \sqrt{2\gamma u}\m\big(W_n - 12K_n\big)\Big)\bigg(\frac{1 - e^{\gamma\phi h_n}}{\gamma h_n}\bigg)\m,\\
\overrightarrow{x\,}_{\nnm n}^{\m(2)} & := \overrightarrow{x\,}_{\nnm n}^{\m(1)} - \overrightarrow{v\,}_{\nnm n}^{\m(2)}\big(1+4\phi\big)h_n\m,\\
\overrightarrow{v\,}_{\nnm n}^{\m(3)} & := e^{\gamma \phi h_n}\overrightarrow{v\,}_{\nnm n}^{\m(2)} + \Big(\hspace{-0.25mm}-u\nabla f\hspace{-0.5mm}\left(\overrightarrow{x\,}_{\nnm n}^{\m(2)}\right)\hspace{-0.125mm} h_n + \sqrt{2\gamma u}\m\big(W_n - 12K_n\big)\Big)\bigg(\frac{1 - e^{\gamma\phi h_n}}{\gamma h_n}\bigg)\m,\\
\overrightarrow{x\,}_{\nnm n+1} & := \overrightarrow{x\,}_{\nnm n}^{\m(2)} + \overrightarrow{v\,}_{\nnm n}^{\m(3)}\big(1+2\phi\big)h_n\m,\\
\overrightarrow{v\,}_{\nnm n}^{\m(4)} & := e^{-\gamma(\frac{1}{2}+\phi)h_n}\overrightarrow{v\,}_{\nnm n}^{\m(3)} + \Big(\hspace{-0.25mm}-u\nabla f\big(\m\overrightarrow{x\,}_{\nnm n+1}\big)\m h_n + \sqrt{2\gamma u}\big(W_n - 12K_n\big)\Big)\bigg(\frac{1 - e^{-\gamma\left(\frac{1}{2}+\phi\right)h_n}}{\gamma h_n}\bigg)\m,\\
\overrightarrow{v\,}_{\nnm n + 1} & := \overrightarrow{v\,}_{\nnm n}^{\m(4)} - \sqrt{2\gamma u}\m\big(H_n - 6K_n\big),
\end{align*}
where\hspace{0.25mm} $\{W_n, H_n, K_n\}_{n\m\geq\m 0}$\hspace{0.25mm} are\hspace{0.25mm} the\hspace{0.25mm} Gaussian\hspace{0.25mm} vectors\hspace{0.25mm} from\hspace{0.25mm} definition\hspace{0.25mm} \ref{def:shifted_ode}\hspace{0.25mm} and\hspace{0.25mm} $\phi$\hspace{0.25mm} is\hspace{0.25mm} given\hspace{0.25mm} by
\begin{align*}
\phi := \frac{-1 + \mysqrt{0}{3}{3}{2\m}\,}{2\big(2-\mysqrt{0}{3}{3}{2\m}\,\big)}\m.
\end{align*}
\end{definition}\medbreak

Despite being straightforward to implement, the above methods are difficult to analyse.
At the very least, it should be much easier to study the shifted ODE method due to the similarities in structure between the SDE (\ref{eq:ULD}) and the ODE (\ref{eq:linULE}). For example, both differential equations have continuous-time solutions and admit similar Taylor expansions.
As our analysis is based on these continuous-time Taylor expansions, it does not apply if the ODE is discretized. That said, we expect these methods to achieve a $2$-Wasserstein error of $\varepsilon$ in $\mathcal{O}\big(\sqrt{d\m}/\varepsilon^{\frac{1}{k}}\m\big)$ steps when the first $k$ derivatives of $f$ are Lipschitz continuous.
In addition, we do not study the numerical stability of the SORT and SOFA methods.
Therefore it is likely that further improvements can be made to the proposed ODE solvers (particularly for large step sizes). For example, it may be possible to discretize (\ref{eq:linULE}) using ideas from \cite{ODEstability1, ODEstability2}, where the stability of splitting methods is considered in the HMC setting.\medbreak

Although this is not explored in the paper, it is also straightforward to incorporate adaptive step sizes into these methodologies. That is, given $(W_n\m, H_n\m, K_n)$, we can generate the same triple over the intervals $[t_n, t_{n+\frac{1}{2}}]$ and $[t_{n+\frac{1}{2}}, t_{n+1}]$ where $t_{n+\frac{1}{2}} := \frac{1}{2}(t_n + t_{n+1})$.
This procedure can be viewed as an extension of L\'{e}vy's construction of Brownian motion and is detailed in \cite{Fosterthesis}. We expect that adaptive step sizes are also possible for the other ULD numerical methods (given in section \ref{sect:other_methods}). However this may be a topic for future work.
\medbreak

\subsection{Organization of the paper}
Having introduced our proposed numerical methods, we now outline the paper.
In Section \ref{sect:notation} we establish our key notation, definitions and assumptions. In Section \ref{sect:main_results} we give an informal derivation of the shifted ODE method and present our main results (which imply the $2$\hspace{0.125mm}-Wasserstein complexities given in Table \ref{table:ode_convergence}).
In Section \ref{sect:other_methods}, we discuss some related methods for ULD in preparation of our numerical experiment. In Section \ref{sect:experiment}, we empirically compare the SORT and SOFA methods to these related methods in an example where the target density comes from a logistic regression.
The technical details for our analysis of the shifted ODE method are given in the appendix.

\section{Notation and definitions}\label{sect:notation}

In this section, we present some of the notation, definitions and assumptions for the paper.\medbreak

Throughout, we use $\|\cdot\|_2$ to denote the standard Euclidean norm on $\R^d$ and $\R^{2d}$.
Along with the Euclidean norm, we use the usual inner product $\langle\m\cdot\m, \m\cdot\m\rangle$ with $\langle v, v\rangle = \|v\|_2^2\m$.
For a function $g$, we use Big O notation $O(g)$ and let $\mathcal{O}(g)$ be the class $O(g)\cdot(\log(g))^{\m O(1)}$.

\subsection{Assumptions on $f$} The function $f :\R^d\rightarrow \R$ is assumed to be  $m$-strongly convex
\begin{align}\label{eq:conv_cond}
f(y) \geq f(x) + \big\langle\m \nabla f(x), y - x\big\rangle + \frac{1}{2}m\|x-y\|_2^2\m,
\end{align}
and at least twice continuously differentiable with an $M$\hspace{-0.125mm}-\hspace{-0.125mm}Lipschitz continuous gradient $\nabla f\hspace{-0.125mm}$,
\begin{align}\label{eq:lip_cond1}
\|\nabla f(x) - \nabla f(y)\|_2 \leq M\|x-y\|_2\m,
\end{align}
for all $x,y\in\R^d$. It is straightforward to show that the conditions (\ref{eq:conv_cond}) and (\ref{eq:lip_cond1}) are equivalent to the Hessian of $f$ being positive definite and satisfying $m I_d\preccurlyeq\nabla^2 f(x) \preccurlyeq M I_d\m$ (where $I_d$ is the $d\times d$ identity matrix and $A \preccurlyeq B$ means that $x^{\T}\nnm A\m x\leq x^{\T}\hspace{-0.25mm} B\m x$ for $x\in\R^d\m$).\medbreak

At certain points in our analysis, we make additional smoothness assumptions on $f$.
Before we give these assumptions, we first recall the notions of tensors and operator norms. 
For $k\geq 1$, a $k$-tensor on $\R^d$ is simply an element of the $d^{\m k}$-\hspace{0.125mm}dimensional space $\R^{d\times \cdots \times d}$.
We can interpret a $k$-tensor on $\R^d$ as a multilinear map into $\R$ with $k$ arguments from $\R^d$.
Hence for a $k$-tensor $T$ on $\R^d$ (with $k\geq 2$) and a vector $v\in\R^d$, we define $Tv := T(v, \cdots)$ as a $(k-1)$-tensor on $\R^d$. Using this, we can define the operator norm of $T$ recursively as
\begin{align*}
\|T\|_{\text{op}} & := \begin{cases}\,\,\,\m\|T\|_2\m, & \text{if}\,\,T\in\R^d,\\[3pt] \,\underset{\substack{v\m\in\m\R^d,\\[2pt] \|v\|_2\m\leq\m 1}}{\sup}\|Tv\|_{\text{op}}\m, & \text{if}\,\,T\,\,\text{is a }k\text{-tensor on }\R^d\text{ with }k\geq 2\m.\end{cases}
\end{align*}

With a slight abuse of notation, we will write $\|\cdot\|_2$ instead of $\|\cdot\|_{\text{op}}$ for all $k$-tensors.
In Section \ref{append:hessian}, we additionally assume that the Hessian of $f$ is $M_2$-Lipschitz continuous:
\begin{align}\label{eq:lip_cond2}
\left\|\nabla^2 f(x) - \nabla^2 f(y)\right\|_2 \leq M_2\|x-y\|_2\m, 
\end{align}
for all $x, y\in \R^{d\times d}$. In Section \ref{append:third}, we shall additionally assume that the function $f$ is three times continuously differentiable and its third derivative is $M_3$-Lipschitz continuous:
\begin{align}\label{eq:lip_cond3}
\left\|\nabla^3 f(x) - \nabla^3 f(y)\right\|_2 \leq M_2\|x-y\|_2\m, 
\end{align}
for all $x, y\in \R^{d\times d\times d}$.

\subsection{Probability notation} Suppose $\big(\Omega, \mathcal{F}, \P\, ; \{\mathcal{F}_t\}_{t\m\geq\m 0}\big)$ is a filtered probability space carrying a standard $d$-dimensional Brownian motion. The only SDE that we analyse in this paper is (\ref{eq:ULD}), which admits a unique strong solution that is ergodic and whose stationary measure $\pi$ has a density $\pi(x,v) \propto e^{-f(x) + \frac{1}{2u}\|v\|_2^2}$ under our assumptions on $f$ (see Proposition 6.1 in \cite{InvariantExists}). Numerical SDE solutions are obtained at times $\{t_n\}_{n\m\geq\m 0}$ with $t_0 = 0$, $t_{n+1} > t_{n}$ and step sizes $h_n = t_{n+1} - t_n\m$. The results in Table \ref{table:ode_convergence} assume fixed steps.
Given a random variable $X$, taking its values in $\R^d$ or $\R^{2d}$, we define the $\L_p$ norm of $X$ as
\begin{align*}
\|X\|_{\L_p} := \E\Big[\|X\|_2^p\Big]^{\frac{1}{p}},
\end{align*}
for $p\geq 1$. Similarly, for a stochastic process $\{X_t\}$, the $\mathcal{F}_{t_n}$-conditional $\L_p$ norm of $X_t$ is
\begin{align*}
\|X_t\|_{\L_p^n} := \E_n\Big[\|X_t\|_2^p\Big]^{\frac{1}{p}} = \E\Big[\|X_t\|_2^p \,\big|\, \mathcal{F}_{t_n}\Big]^{\frac{1}{p}},
\end{align*}
for $t\geq t_n\m$. For $0\leq s\leq t$, we use $X_{s,t} := X_t - X_s$ to denote the increment of $X$ over $[s,t]$.

\subsection{Coupling and Wasserstein distance}

Let $\mu$, $\nu$ be probability measures on $\mathbb{R}^{d}$.
A coupling between $\mu$ and $\nu$ is a random variable $Z = (X,Y)$ for which $X\sim\mu$ and $Y\sim\nu$.
For $p\geq 1$, we define the $p$\hspace{0.25mm}-Wasserstein distance between $\mu$ and $\nu$ as
\begin{align}\label{eq:wasserstein_def}
W_{p}\big(\mu, \nu\big) :=\inf_{(X, Y)\m\sim\m (\mu\times \nu)}\|X-Y\|_{\L_p}\m,
\end{align}
where the above infimum is taken over all couplings of
random variables $X$ and $Y$ with distributions $X\sim\mu$ and $Y\sim\nu$. In this paper, we obtain non-asymptotic bounds for the $2$\hspace{0.125mm}-Wasserstein distance between the target distribution and shifted ODE approximation.

\subsection{Diffusion and approximation processes}\label{subsect:define_approx}

We shall use $W$ to denote a standard $d$-dimensional Brownian motion and define the (piecewise) Brownian bridge process $B$ as
\begin{align}\label{eq:bridge_def}
B_t = W_t - \bigg(W_{t_n} + \bigg(\frac{t-t_n}{h_n}\bigg)\m W_{t_n, t_{n+1}}\bigg),
\end{align}
for $t\in[t_n, t_{n+1}]$ and $n\geq 0$. Using $W$ and $B$, we can define the following Gaussian vectors:
\begin{align}
W_n & := W_{t_n, t_{n+1}}\m,\label{eq:Wn_def}\\[2pt]
H_n & := \frac{1}{h_n}\int_{t_n}^{t_{n+1}}\nnm B_{t_n, t}\,dt,\label{eq:Hn_def}\\[1pt]
K_n & := \frac{1}{h_n^2}\int_{t_n}^{t_{n+1}}\nnm\bigg(\frac{1}{2}h_n - (t-t_n)\bigg)B_{t_n, t}\,dt.\label{eq:Kn_def}
\end{align}
Since ``$\m 12K_n$'' appears in the shifted ODE, we also consider the ``shifted'' Brownian bridge:
\begin{align*}
\widetilde{B}_{t} := B_{t} + 12K_n\bigg(\frac{t-t_n}{h_n}\bigg),
\end{align*}
for $t\in[t_n, t_{n+1})$ and $n\geq 0$. In the SDE (\ref{eq:ULD}), we use $\sigma := \sqrt{2\gamma u}$ to simply the notation. In our analysis of ULD, we shall use the following diffusion and approximation processes:\medbreak

1. The underdamped Langevin diffusion $\{(x_t, v_t)\}_{t\m\geq\m 0}$ is defined by the SDE,
\begin{align*}
dx_t & = v_t\,dt,\\[3pt]
dv_t & = -\gamma v_t\,dt - u\nabla f(x_t)\,dt + \sigma\, dW_t\m,\nonumber
\end{align*}

with initial condition $(x_0, v_0)\sim p_0\m$, where $p_0$ is a distribution on $\R^{2d}$. We usually set $p_0$ to the unique invariant measure $\pi$ of ULD, which has density $\pi(x,v) \propto e^{-f(x) + \frac{1}{2u}\|v\|_2^2}$.
In this case, we have $(x_t, v_t)\sim \pi$ for all $t\geq 0\m$, and so the $2$-Wasserstein distance between
(the law of) an approximation $\widetilde{x}_n$ at time $t_n$ and $\pi$ can be estimated using $\|\m\widetilde{x}_n - x_{t_n}\|_{\L_2}$.\medbreak

2. Instead of using definition \ref{def:shifted_ode}, it will be more convenient to define the shifted ODE approximation $\{(\widetilde{x}_n, \widetilde{v}_n)\}_{n\m\geq\m 0}$ by setting $(\widetilde{x}_0\m, \widetilde{v}_0)\sim \widetilde{p}_0$ (for some distribution $\widetilde{p}_0$ on $\R^{2d}\m$) and defining $(\widetilde{x}_{n+1}, \widetilde{v}_{n+1})$ as
\begin{align*}
\Bigg(\begin{matrix} \,\widetilde{x}_{n+1}\\[-3pt] \,\widetilde{v}_{n+1}\end{matrix}\Bigg)
:=
\Bigg(\begin{matrix} \,\widehat{x}_{t_{n+1}}^{\m n}\\[-3pt] \,\widehat{v}_{t_{n+1}}^{\m n}\end{matrix}\Bigg)
+ 12 K_n\Bigg(\begin{matrix} \,0\\[-3pt] \,\sigma\end{matrix}\,\Bigg),
\end{align*}
where $\big\{\big(\widehat{x}_t^{\m n}, \widehat{v}_t^{\m n}\big)\big\}_{t\m\in\m [t_n, t_{n+1}]}$ solves the following ODE,
\begin{align*}
\frac{d}{dt}\Bigg(\begin{matrix} \,\widehat{x}^{\m n}\\[-3pt] \,\widehat{v}^{\m n}\end{matrix}\Bigg)
=
\Bigg(\begin{matrix} \,\widehat{v}^{\m n} + \sigma\big(H_n + 6K_n\big)\\[-3pt] \,-\gamma\big(\m\widehat{v}^{\m n} + \sigma\big(H_n + 6K_n\big)\big) - u\nabla f\big(\,\widehat{x}^{\m n}\big)\end{matrix}\,\Bigg)
+ \frac{W_n - 12K_n}{h_n}\,\Bigg(\begin{matrix} \,0\\[-3pt] \,\sigma\end{matrix}\,\Bigg)\m,
\end{align*}
with initial condition $\big(\m\widehat{x}_{t_n}^{\m n}, \widehat{v}_{t_n}^{\m n}\big) := \big(\widetilde{x}_{n}\m, \widetilde{v}_n\big)$.\medbreak

3. We also consider a sequence of shifted ODE approximations $\{(x_n^\prime, v_n^\prime)\}_{n\m\geq\m 0}$ given by
\begin{align*}
\Bigg(\begin{matrix} \,x_{n+1}^\prime\\[-3pt] \,v_{n+1}^\prime\end{matrix}\Bigg)
:=
\Bigg(\begin{matrix} \,\wx_{t_{n+1}}\\[-3pt] \,\wv_{t_{n+1}}\end{matrix}\Bigg)
+ 12 K_n\Bigg(\begin{matrix} 0\\[-3pt] \sigma\end{matrix}\,\Bigg),
\end{align*}
where $\big\{\big(\wx_t, \wv_t\big)\big\}_{t\in [t_n, t_{n+1}]}$ solves the following ODE,
\begin{align*}
\frac{d}{dt}\Bigg(\begin{matrix} \,\wx\\[-3pt] \,\wv\end{matrix}\Bigg)
=
\Bigg(\begin{matrix} \,\wv + \sigma\big(H_n + 6K_n\big)\\[-3pt] \,-\gamma\big(\m\wv + \sigma\big(H_n + 6K_n\big)\big) - u\nabla f\big(\,\wx\big)\end{matrix}\,\Bigg)
+ \frac{W_n - 12K_n}{h_n}\,\Bigg(\begin{matrix} \,0\\[-3pt] \,\sigma\end{matrix}\,\Bigg)\m,
\end{align*}
with initial condition $\big(\m\wx_{t_n}\m, \wv_{t_n}\big)\m :=\m \big(x_n\m, v_n\big)$. In addition, we set $(\m x_{0}^\prime\m, v_{0}^\prime\m)\m :=\m (x_0\m, v_0)\m\sim\m p_0\m$.\medbreak

4. Since we use the same $2$-Wasserstein contractivity arguments as those given in \cite{KineticLangevinMCMC}, we define certain linear combinations of $\{(x_{t_n}, v_{t_n})\}_{n\m\geq\m 0}\m$, $\{(\widetilde{x}_n, \widetilde{v}_n)\}_{n\m\geq\m 0}$ and $\{(x_n^\prime, v_n^\prime)\}_{n\m\geq\m 0}\m$:
\begin{align*}
y_n & := \Bigg(\begin{matrix} \,\big(\lambda\m\widetilde{x}_n + \widetilde{v}_n\big) - \big(\lambda\m x_{t_n} + v_{t_n}\big)\\[-3pt] \,\big(\eta\m\widetilde{x}_n + \widetilde{v}_n\big) - \big(\eta\m x_{t_n} + v_{t_n}\big)\end{matrix}\Bigg),\\[3pt]
\widetilde{y}_n & := \Bigg(\begin{matrix} \,\big(\lambda\m\widetilde{x}_n + \widetilde{v}_n\big) - \big(\lambda\m x_n^\prime + v_n^\prime\big)\\[-3pt] \,\big(\eta\m\widetilde{x}_n + \widetilde{v}_n\big) - \big(\eta\m x_n^\prime + v_n^\prime\big)\end{matrix}\Bigg),\\
y_n^\prime & := \Bigg(\begin{matrix} \,\big(\lambda\m x_n^\prime + v_n^\prime\big) - \big(\lambda\m x_{t_n} + v_{t_n}\big)\\[-3pt] \,\big(\eta\m x_n^\prime + v_n^\prime\big) - \big(\eta\m x_{t_n} + v_{t_n}\big)\end{matrix}\Bigg),
\end{align*}
where $\lambda\in[0,\frac{1}{2}\gamma)$ and $\eta := \gamma - \lambda\m$.\medbreak

We use the notation $\big\{\big(\m\overrightarrow{x\,}_{\nnm n}\m, \overrightarrow{v\,}_{\nnm n}\big)\big\}_{n\m\geq\m 0}$ to denote the other numerical solutions of ULD (SORT, SOFA, OBABO, Strang splitting, randomized midpoint and left-point methods).
ULD and its various approximations will be defined from the same Brownian motion $W$.

\section{Derivation and error analysis of the shifted ODE method}\label{sect:main_results}

In this section, we shall discuss the theoretical aspects of the shifted ODE method and present our main results regarding the $2$-Wasserstein convergence of (the law of) $\widetilde{x}_n$ to $\pi$.
For general SDEs, it is well known that the best $\L_2$ approximations converge with rate $O\big(N^{-\frac{1}{2}}\m\big)$ where $N$ is the number of Gaussian random variables used by the method \cite{CameronClark, Dickinson}.
Hence we begin the section by explaining why high order methods are possible for ULD.

\subsection{The stochastic Taylor expansion of ULD} Perhaps the most important tool when developing and analysing numerical methods for SDEs is the stochastic Taylor expansion.
Just like for ODEs, this allows one to express the solution as a sum of ``polynomial'' and ``remainder'' terms. The challenge for SDEs is that these ``polynomial'' terms can now include certain iterated integrals of Brownian motion that are difficult to generate \cite{GainesLyonsInt, OptimalPoly}. 
By applying the It\^{o}-Taylor expansion (Theorem 5.5.1 in \cite{KloePlat}) to the Langevin SDE (\ref{eq:ULD}),
we see the significant terms are either monomials in $h = t - s$ or Gaussian integrals of $W$.

\begin{theorem}[\textbf{High order Taylor expansion of ULD}]\label{thm:ULD_Taylor} Consider the SDE (\ref{eq:ULD}) and suppose that the potential $f$ is three times continuously differentiable. Then for $0\leq s\leq t$,
\begin{align}
x_t = x_s & + v_s\,h - \frac{1}{2}\big(\gamma v_s + u \nabla f(x_s)\big)h^2  + \sigma \int_s^{t}W_{s, r}\,dr - \sigma\gamma\int_s^t\int_s^{r_1} W_{s, {r_2}}\,dr_2\,dr_1 \label{eq:complex_x_expand}\\
& + \frac{1}{6}\big(\gamma^2 v_s + \gamma u \nabla f(x_s) - u\nabla^2 f(x_s)v_{s}\big) h^3\nonumber\\
& + \sigma \big(\gamma^2 - u\nabla^2 f(x_s)\big)\int_s^t\int_s^{r_1}\nnm\int_s^{r_2}W_{s,r_3}\,dr_3\,dr_2\,dr_1\nonumber\\
& - \frac{1}{24}\big(\big(\gamma^2 - u\nabla^2 f(x_s)\big)\big(\gamma v_s + u\nabla f(x_s)\big) - u\gamma\nabla^2 f(x_s) v_s + u\nabla^3 f(x_s)(v_s, v_s)\big)h^4\nonumber\\[2pt]
& + R^{\m x}(h, x_{s}, v_{s})\m,\nonumber\\[6pt]
v_t = v_s & - \big(\gamma v_s + u \nabla f(x_s)\big) h + \sigma W_{s,t} - \sigma\gamma\int_s^t W_{s,r}\,dr \label{eq:complex_v_expand}\\
& + \frac{1}{2}\big(\gamma^2 v_s + u\gamma\nabla f(x_s) - u\nabla^2 f(x_s)v_s\big) h^2 + \sigma\big(\gamma^2 - u\nabla^2 f(x_s)\big)\int_{s}^t\int_{s}^{r_1} W_{s, r_2}\,dr_2 \,dr_1\nonumber\\[1pt]
&  - \frac{1}{6}\big(\big(\gamma^2 - u\nabla^2 f(x_s)\big)\big(\gamma v_s + u\nabla f(x_s)\big) - u\gamma \nabla^2 f(x_s)v_s + u\nabla^3 f(x_s)(v_s, v_s)\big)h^3\nonumber\\
& - \sigma\gamma\big(\gamma^2 - u\nabla^2 f(x_s)\big)\int_{s}^t\int_{s}^{r_1}\nm\int_{s}^{r_2} W_{s, r_3}\,dr_3\,dr_2\,dr_1\nonumber\\
& - \sigma u\int_s^t\int_s^{r_1} \nabla^3 f(x_s)\big(v_s, (r_2 - s)W_{s, r_2}\big)\,dr_2\,dr_1 + R^{\m v}(h, x_{s}, v_{s}).\nonumber
\end{align}
where $h = t - s$ and the remainder terms $R^{\m x}(h, x_{s}, v_{s})$, $R^{\m v}(h, x_{s}, v_{s})$ are given by
\begin{align}
R^{\m x}(h, x_{s}, v_{s}) & = \gamma\big(u\nabla^2 f(x_s) - \gamma^2\big)\int_s^t\int_s^{r_1}\nnm\int_s^{r_2}\nnm\int_s^{r_3}(v_{r_4} - v_{s})\,dr_4\,dr_3\,dr_2\,dr_1\label{eq:complex_x_remainder} \\
&\mm + u\big(u\nabla^2 f(x_s) - \gamma^2\big)\int_s^t\int_s^{r_1}\nnm\int_s^{r_2}\nnm\int_s^{r_3} \big(\m\nabla f(x_{r_4}) - \nabla f(x_s)\big)\,dr_4\,dr_3\,dr_2\,dr_1 \nonumber\\
&\mm + u\gamma \int_s^t\int_s^{r_1}\nnm\int_s^{r_2}\nnm\int_s^{r_3} \big(\m\nabla^2 f(x_{r_4})v_{r_4} - \nabla^2 f(x_s)v_s\big) \,dr_4\,dr_3\,dr_2\,dr_1\nonumber\\
&\mm - u \int_s^t\int_s^{r_1}\nnm\int_s^{r_2}\nnm\int_s^{r_3}\big(\m\nabla^3 f(x_{r_4})(v_{r_4}, v_{r_3})- \nabla^3 f(x_s)(v_s, v_s)\big)\,dr_4\,dr_3\,dr_2\,dr_1\m,\nonumber\\[6pt]
R^{\m v}(h, x_{s}, v_{s}) & = \gamma\big(u\nabla^2 f(x_s) - \gamma^2\big)\int_{s}^t\int_{s}^{r_1}\nm\int_{s}^{r_2} \big(v_{r_3} - v_s - \sigma W_{s, r_3}\big)\,dr_3\,dr_2\,dr_1\label{eq:complex_v_remainder}\\
&\mm + u\big(u\nabla^2 f(x_s) - \gamma^2\big)\int_{s}^t\int_{s}^{r_1}\nm\int_{s}^{r_2}\big(\nabla f(x_{r_3}) - \nabla f(x_s)\big)\,dr_3\,dr_2 \,dr_1\nonumber\\
&\mm + u\gamma \int_{s}^t\int_{s}^{r_1}\nm\int_{s}^{r_2}\big(\nabla^2 f(x_{r_3})v_{r_3} - \nabla^2 f(x_s)v_s\big)\,dr_3\,dr_2\,dr_1\nonumber\\
&\mm - u\int_{s}^t\int_{s}^{r_1}\nm\int_{s}^{r_2}\big(\m\nabla^3 f(x_{r_3})(v_{r_3}, v_{r_2}) - \nabla^3 f(x_s)(v_s, v_s)\big)\,dr_3\,dr_2\,dr_1\nonumber\\
&\mm - \sigma u\int_s^t\int_s^{r_1}\nm\int_{s}^{r_2}\big(\m\nabla^3 f(x_{r_3})\big(v_{r_3}, W_{s, r_2}\big) - \nabla^3 f(x_s)\big(v_s, W_{s, r_2}\big)\big)\,dr_3\,dr_2\,dr_1\m.\nonumber
\end{align}
\end{theorem}\medbreak\medbreak
\begin{proof}
For $\tau \in [s,t]$, the SDE (\ref{eq:ULD}) can be written in integral form as
\begin{align}
x_\tau & = x_s + \int_s^\tau v_r\,dr,\label{eq:taylor_proof1}\\[3pt]
v_\tau & = v_s - \gamma\int_s^\tau v_r\,dr - u\int_s^\tau \nabla f(x_r)\,dr + \sigma W_{s,\tau}\m.\label{eq:taylor_proof2}
\end{align}
Moreover, it then follows by It\^{o}'s lemma that
\begin{align}
\nabla f(x_{r_1}) & = \nabla f(x_s) + \int_s^{r_1}\nabla^2 f(x_{r_2})\,dx_{r_2}\nonumber\\
& = \nabla f(x_s) + \int_s^{r_1}\nabla^2 f(x_{r_2})v_{r_2}\,dr_2\m,\label{eq:taylor_proof3}\\[3pt]
\nabla^2 f(x_{r_2})v_\tau & = \nabla^2 f(x_s)v_\tau + \int_s^{r_2}\nabla^3 f(x_{r_3})\big(dx_{r_3}, v_\tau\big)\nonumber\\
& =  \nabla^2 f(x_s)v_\tau + \int_s^{r_2}\nabla^3 f(x_{r_3})(v_{r_3}, v_\tau)\,dr_3\m.\label{eq:taylor_proof4}
\end{align}
for $s\leq r_1\leq r_2\leq t$. The result now follows by repeatedly substituting the four identities (\ref{eq:taylor_proof1}) -- (\ref{eq:taylor_proof4}) into the integral equations (\ref{eq:taylor_proof1}) -- (\ref{eq:taylor_proof2}) and rearranging the various terms.
\end{proof}
\begin{remark}
As a quick sanity check, we see that by integrating the lower order terms in the expansion of $v$ over $[s,t]$, we obtain the corresponding terms in the expansion of $x$.
\end{remark}

By ignoring the remainder terms from the expansions (\ref{eq:complex_v_expand}) and (\ref{eq:complex_x_expand}), we can derive a stochastic Taylor approximation for the diffusion $(x,v)$ at time $t$ given its value at time $s$.
However this would require us to compute or estimate the derivatives $\nabla^2 f(x_s), \nabla^3 f(x_s)$
and in applications, this may become computationally expensive (especially if $d$ is high).
That said, this does indicate that ``third order'' methods can be implemented for ULD.
Hence we seek a numerical method that captures the Taylor expansions in Theorem \ref{thm:ULD_Taylor} but only requires evaluations of $\nabla f$ (and is therefore ``derivative-free'' in a certain sense).

\subsection{Derivation of the shifted ODE method} Our key insight (which is really an insight from rough path theory \cite{StFlourLectureNotes}) is that SDEs are a particular example of so-called ``controlled differential equations'' (CDEs). In our setting, we consider the following CDE:
\begin{align}\label{eq:CDE}
d\m\overline{x}_t & = \overline{v}_t\,dt,\\[3pt]
d\m\overline{v}_t & = -\gamma\m \overline{v}_t\,dt - u\nabla f(\m\overline{x}_t)\,dt + \sigma\, dX_t\m,\nonumber
\end{align}
where $X : [0,T]\rightarrow \R^d$ is continuous and has bounded variation so that (\ref{eq:CDE}) is well defined.
Moreover, by constructing the path $X$ using information about the Brownian motion $W$,
we can obtain a CDE solution that approximates the SDE solution. To understand how $X$ should be constructed, we first note that (\ref{eq:CDE}) admits a Taylor expansion analogues to (\ref{eq:complex_x_expand}) -- (\ref{eq:complex_v_expand}). The only difference is that the chain rule is employed instead of It\^{o}'s lemma.
\begin{theorem}[\textbf{Taylor expansion of underdamped Langevin CDE}]\label{thm:CDE_Taylor} 
Consider the CDE (\ref{eq:CDE}) and suppose $f$ is three times continuously differentiable. Then for $0\leq s\leq t$,
\begin{align}
\overline{x}_t = \overline{x}_s & + \overline{x}_s\,h - \frac{1}{2}\big(\cdots\big)h^2  + \sigma \int_s^{t}X_{s, r}\,dr - \sigma\gamma\int_s^t\int_s^{r_1} X_{s, {r_2}}\,dr_2\,dr_1 \label{eq:complex_cde_x_expand}\\
& + \frac{1}{6}\big(\cdots\big) h^3 + \sigma \big(\cdots\big)\int_s^t\int_s^{r_1}\nnm\int_s^{r_2}X_{s,r_3}\,dr_3\,dr_2\,dr_1 - \frac{1}{24}\big(\cdots\big)h^4 + R^{\m \overline{x}}(h, \overline{x}_{s}, \overline{v}_{s})\m,\nonumber\\[6pt]
\overline{v}_t = \overline{v}_s & - \big(\gamma\m \overline{v}_s + u \nabla f(\m\overline{x}_s)\big) h + \sigma X_{s,t} - \sigma\gamma\int_s^t X_{s,r}\,dr + \frac{1}{2}\big(\cdots\big) h^2\label{eq:complex_cde_v_expand}\\
& + \sigma\big(\cdots\big)\int_{s}^t\int_{s}^{r_1} X_{s, r_2}\,dr_2 \,dr_1 - \frac{1}{6}\big(\cdots\big)h^3 - \sigma\gamma\big(\cdots\big)\int_{s}^t\int_{s}^{r_1}\nm\int_{s}^{r_2} X_{s, r_3}\,dr_3\,dr_2\,dr_1\nonumber\\
& - \sigma u\int_s^t\int_s^{r_1} \nabla^3 f(\m\overline{x}_s)\big(\m\overline{v}_s, (r_2 - s)X_{s, r_2}\big)\,dr_2\,dr_1 + R^{\m \overline{v}}(h, \overline{x}_{s}, \overline{v}_{s}),\nonumber
\end{align}
where $h = t - s$, the coefficients $(\m\cdots)$ are given in Theorem \ref{thm:ULD_Taylor} and the remainder terms $R^{\m \overline{x}}(h, \overline{x}_{s}, \overline{v}_{s}), R^{\m \overline{v}}(h, \overline{x}_{s}, \overline{v}_{s})$ have the same form as their stochastic counterparts (\ref{eq:complex_x_remainder}), (\ref{eq:complex_v_remainder}).
\end{theorem}
\medbreak

Thus, to approximate the SDE solution, we construct $X$ over the interval $[s,t]$ so that
\begin{align*}
1. & \,\,\,\,X_{s,t} = W_{s,t}\m,\\[3pt]
2. & \,\int_s^t X_{s,r}\,dr = \int_s^t W_{s,r}\,dr\m,\\[3pt]
3. & \,\int_s^t\int_s^{r_1} X_{s,r_2}\,dr_2\,dr_1 = \int_s^t\int_s^{r_1} W_{s,r_2}\,dr_2\,dr_1\m.
\end{align*}

At this point, it is worth noting that the above three properties are not sufficient for accurately approximating general SDEs. This is because a general SDE will contain additional non-Gaussian iterated integrals of Brownian motion in its Taylor expansion.
The two most notable examples of such integrals (neither of which appears for ULD) are
\begin{align*}
\int_s^t W_{s, r}^{(i)}\,dW_r^{(j)},\hspace{5mm}\int_s^t\int_s^{r_1} W_{s, r_2}\,dW_{r_2}\,dr_1\m,
\end{align*}
for $i,j\in\{1,\cdots, d\}$ with $i\neq j$. Due to the specific structure of the Langevin SDE/CDE, we do not need to take these integrals into consideration when constructing the path $X$.\medbreak

To make the CDE (\ref{eq:CDE}) easier to discretize, we will assume that $X$ is piecewise linear.
Moreover, the path should have as few pieces as possible (to reduce computational cost).
The following theorem shows that, in general, it is not enough for $X$ to have two pieces.
\begin{theorem}
Let $a\in(s,t)$ and $b,c\in\R^d$. Suppose that the piecewise linear path $X$ is defined on $[s,t]$ by $X_r := b\big(\frac{r-s}{a-s}\big)$ for $r\in[s,a]$ and $X_r := b + (c-b)\big(\frac{r-a}{t-a}\big)$ for $r\in[a,t]$.
Then $X_{s,t} = c$, $\int_s^t X_{s,r}\,dr = \frac{1}{2}(b+c)(t-s)$ and $\int_s^t\int_s^{r_1} X_{s,r_2}\,dr_2\,dr_1$ depends linearly on $b,c$ and quadratically on $\m a$. In particular, for every $b,\m c\in\R^d$, it is bounded as a function of $ a$.
\end{theorem}
\begin{proof}
Using integration by parts, we can express the iterated integral of $X$ as
\begin{align*}
\int_s^t\int_s^{r_1} X_{s,r_2}\,dr_2\,dr_1 = (t-s)\int_s^t X_{s,r}\,dr - \int_s^t (r - s)\m X_{s,r}\,dr\m.
\end{align*}
By evaluating the above, we see that the integral is linear in $b,c$ and quadratic in $a$.
\end{proof}

Therefore, the path $X$ should be a piecewise linear path on $[s,t]$ with three pieces
and thus defined by five variables: the discretization values $b,c,d\in\R^d$ and times $a_1, a_2\in(s,t)$. 
Our key observation is that we can set $a_1 = s$, $a_2 = t$. That is, we give $X$  ``vertical pieces''.
Since the Langevin equation has additive noise, it is trivial to solve the CDE along these vertical pieces and so (\ref{eq:CDE}) reduces to a single ODE (corresponding to the middle piece).

\begin{figure}[H]\label{piecelindiagram}
\begin{center}
\includegraphics[width=\textwidth]{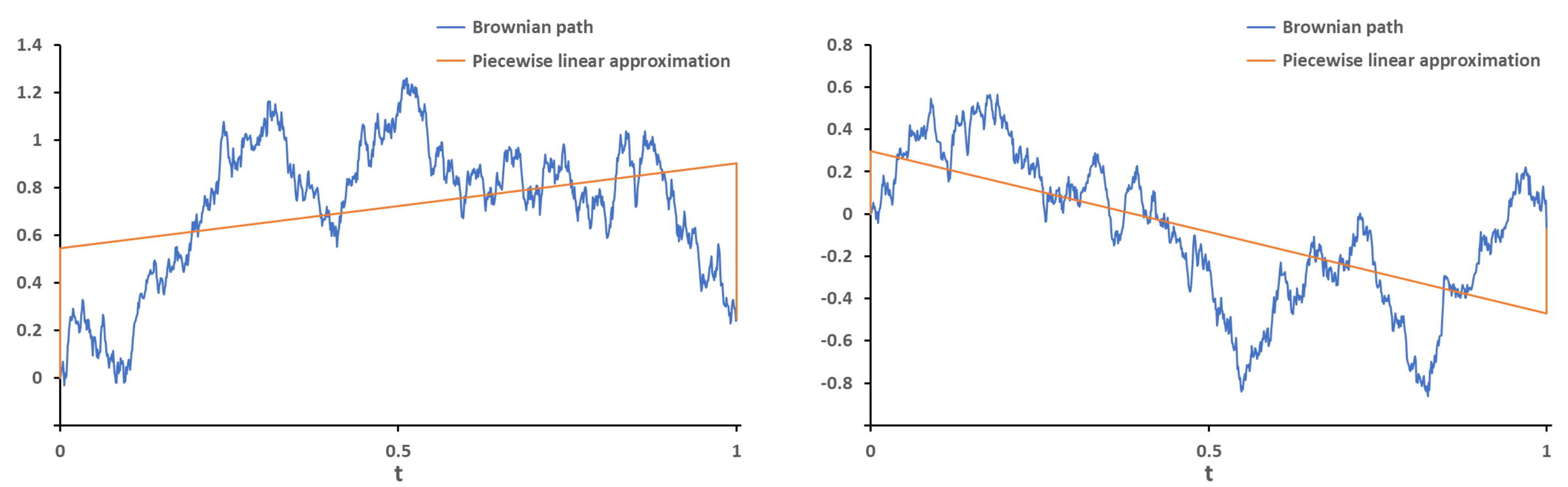}
\caption{Brownian motion approximated using piecewise linear paths with vertical pieces.}
\end{center}
\end{figure}\vspace*{-10mm}
Finally to derive the shifted ODE method, we determine the unknown variables $b,c,d$.
To simply calculations, we express the integrals of $W$ using independent random vectors.
\begin{theorem}\label{thm:hk_st}
Let $W$ denote a Brownian motion. With $h = t - s$, we define the integrals:
\begin{align}
H_{s,t} & := \frac{1}{h}\int_s^t \bigg(W_{s, r} - \frac{r-s}{h}W_{s,t}\bigg)\,dr,\label{eq:h_st}\\[3pt]
K_{s,t} & := \frac{1}{h^2}\int_s^t \bigg(\frac{1}{2}h - (r-s)\bigg)\bigg(W_{s, r} - \frac{r-s}{h}W_{s,t}\bigg)\,dr.\label{eq:k_st}
\end{align}
Then $W_{s,t}\sim\mathcal{N}\big(0, h I_d\big)$, $H_{s,t}\sim\mathcal{N}\big(0, \frac{1}{12}h I_d\big)$, $K_{s,t}\sim\mathcal{N}\big(0, \frac{1}{720}h I_d\big)$ are independent and
\begin{align}
\int_s^t W_{s,r}\,dr & = \frac{1}{2}hW_{s,t} + hH_{s,t}\m,\label{eq:k_st_ident_1}\\[3pt]
\int_s^t\int_s^{r_1} W_{s,r_2}\,dr_2\,dr_1 & = \frac{1}{6}h^2 W_{s,t} + \frac{1}{2}h^2 H_{s,t} + h^2 K_{s,t}\m.\label{eq:k_st_ident_2}
\end{align}
\end{theorem}
\begin{proof}
The first identity is trivial and the second identity follows by a direct calculation (involving integration by parts). For each $i\in\{1,\cdots, d\}$, we can identify $W_{s,t}^{(i)}, H_{s,t}^{(i)}, K_{s,t}^{(i)}$ as rescaled coefficients from the polynomial expansion of $W^{(i)}$ (see Theorem 2.3 of \cite{OptimalPoly}).
In particular, this means that $W_{s,t}, H_{s,t}, K_{s,t}$ are independent centered Gaussian vectors.
The variances of both $H_{0,1}$ and $K_{0,1}$ can be directly computed using Fubini's theorem and the Brownian bridge covariance function, $\E[B_{r_1} B_{r_2}] = \min(r_1\m ,r_2) - r_1 r_2$ for $r_1\m,r_2\in[0,1]$.
\begin{align*}
\var\big(H_{0,1}\big) & = \int_0^1\int_0^1 \E\big[B_{r_1}B_{r_2}\big] \,dr_2\,dr_1 = \frac{1}{12}\m,\\[3pt]
\var\big(K_{0,1}\big) & = \int_0^1\int_0^1\bigg(\frac{1}{2} - r_1\bigg)\bigg(\frac{1}{2} - r_2\bigg)\E\big[B_{r_1}B_{r_2}\big]\,dr_2\,dr_1 =  \frac{1}{720}\m.
\end{align*}
The result follows by the natural Brownian scaling $(H_{s,t}\sim \sqrt{h} H_{0,1}\m, K_{s,t}\sim \sqrt{h} K_{0,1})$.
\end{proof}
Using the identities from Theorem \ref{thm:hk_st}, we can derive a simple formula for the path $X$.
\begin{theorem}
Let $0\leq s < t$ and $h = t - s$. We define a path $X : [s,t] \rightarrow \R^d$ as follows:
\begin{align*}
X_r := \begin{cases} \big(H_{s,t} + 6K_{s,t}\big) + \big(W_{s,t} - 12K_{s,t}\big)\displaystyle\frac{r-s}{h}\m, & \text{for}\,\,\,r\in(s,t)\\
\,\,W_r\m, & \text{for}\,\,r\in\{s,t\}\end{cases}.
\end{align*}
Then $\displaystyle X_{s,t} = W_{s,t}\m,\int_s^t\nnm X_{s,r}\,dr  = \int_s^t\nnm W_{s,r}\,dr\,$ and $\displaystyle\int_s^t\int_s^{r_1}\nnm X_{s,r_2}\,dr_2\,dr_1 = \int_s^t\int_s^{r_1}\nnm W_{s,r_2}\,dr_2\,dr_1\m.$
\end{theorem}
\begin{proof} The result follows by a direct calculation with the identities (\ref{eq:k_st_ident_1}) and (\ref{eq:k_st_ident_2}).
\end{proof}\medbreak

By driving the CDE (\ref{eq:CDE}) with the above path, we obtain the shifted ODE method (definition \ref{def:shifted_ode}). The ODE (\ref{eq:linULE}) defined on $[0,1]$ is then obtained by a change of variables.

\subsection{Derivation of the SORT method} For simplicity, we consider the following ODE:
\begin{align}\label{eq:ham_system}
\frac{d}{dr}\Bigg(\begin{matrix} \m x_r^{n}\\[-4pt] \m v_r^{n}\end{matrix}\Bigg) = \Bigg(\begin{matrix} \,v_r^{n}\\[-4pt] \,-u\nabla f\big(x_r^{n}\big)\end{matrix}\Bigg),
\end{align}
on $[t_n, t_{n+1}]$. Then by applying the standard third order Runge-Kutta method \cite{Butcher}, we have
\begin{align*}
\overrightarrow{x\,}_{\nnm n}^{\m (1)} & := \overrightarrow{x\,}_{\nnm n} + \frac{1}{2}\overrightarrow{v\,}_{\nnm n}\m h_n\m,\\[3pt]
\overrightarrow{v\,}_{\nnm n}^{\m (1)} & := \overrightarrow{v\,}_{\nnm n} - \frac{1}{2}u\nabla f\big(\overrightarrow{x\,}_{\nnm n}\big)\m h_n\m,\\[3pt]
\overrightarrow{x\,}_{\nnm n}^{\m (2)} & := \overrightarrow{x\,}_{\nnm n} - \overrightarrow{v\,}_{\nnm n} \m h_n + 2 \overrightarrow{v\,}_{\nnm n}^{\m (1)}\m h_n\m,\\[3pt]
\overrightarrow{v\,}_{\nnm n}^{\m (2)} & := \overrightarrow{v\,}_{\nnm n} + u \nabla f\big(\m\overrightarrow{x\,}_{\nnm n}\big)h_n - 2u\nabla f\big(\overrightarrow{x\,}_{\nnm n}^{\m (1)}\big)\m h_n\m,\\[3pt]
\overrightarrow{x\,}_{\nnm n + 1} & := \overrightarrow{x\,}_{\nnm n} + \frac{1}{6}\overrightarrow{v\,}_{\nnm n}\m h_n + \frac{2}{3}\overrightarrow{v\,}_{\nnm n}^{\m (1)}\m h_n + \frac{1}{6} \overrightarrow{v\,}_{\nnm n}^{\m (2)}\m h_n\m,\\[3pt]
\overrightarrow{v\,}_{\nnm n + 1} & := \overrightarrow{v\,}_{\nnm n} - \frac{1}{6}u\nabla f\big(\m\overrightarrow{x\,}_{\nnm n}\big)h_n - \frac{2}{3} u\nabla f\big(\m\overrightarrow{x\,}_{\nnm n}^{\m (1)}\big)h_n - \frac{1}{6} u\nabla f\big(\m\overrightarrow{x\,}_{\nnm n}^{\m (2)}\big)h_n\m.
\end{align*}
Our first observation is that $\overrightarrow{x\,}_{\nnm n}^{\m (1)}$ approximates the solution at the midpoint of $[t_n, t_{n+1}]$.
and thus can be improved by including the $O(h^2)$ terms from the Taylor expansion (\ref{eq:complex_x_expand}).
Another improvement would be to replace $\nabla f\big(\m\overrightarrow{x\,}_{\nnm n}^{\m (2)}\big)$ with $\nabla f\big(\m\overrightarrow{x\,}_{\nnm n + 1}\big)$ in the final line. Since $\overrightarrow{x\,}_{\nnm n + 1}$ does not depend on $\overrightarrow{x\,}_{\nnm n}^{(2)}$, this makes $\overrightarrow{x\,}_{\nnm n}^{\m (2)}$ obsolete and thus can be removed.
Finally, we substitute  for $\overrightarrow{v\,}_{\nnm n}^{\m (1)}$ and $\overrightarrow{v\,}_{\nnm n}^{\m (2)}$ in the formula of $\overrightarrow{x\,}_{\nnm n + 1}$ and rearrange the terms.
After performing these improvements, the Runge-Kutta method simplifies to the following:
\begin{align*}
\overrightarrow{x\,}_{\nnm n + \frac{1}{2}} & := \overrightarrow{x\,}_{\nnm n} + \frac{1}{2}\overrightarrow{v\,}_{\nnm n}\m h_n - \frac{1}{8}u\nabla f\big(\m\overrightarrow{x\,}_{\nnm n}\big)h_n^2 \m,\\[3pt]
\overrightarrow{x\,}_{\nnm n + 1} & := \overrightarrow{x\,}_{\nnm n} + \overrightarrow{v\,}_{\nnm n}\m h_n - \frac{1}{6} u \nabla f\big(\m\overrightarrow{x\,}_{\nnm n}\big)h_n^2 - \frac{1}{3}u\nabla f\big(\overrightarrow{x\,}_{\nnm n + \frac{1}{2}}\big)\m h_n^2\m,\\[3pt]
\overrightarrow{v\,}_{\nnm n + 1} & := \overrightarrow{v\,}_{\nnm n} - \frac{1}{6}u\nabla f\big(\m\overrightarrow{x\,}_{\nnm n}\big)h_n - \frac{2}{3} u\nabla f\big(\m\overrightarrow{x\,}_{\nnm n + \frac{1}{2}}\big)h_n - \frac{1}{6} u\nabla f\big(\m\overrightarrow{x\,}_{\nnm n + 1}\big)h_n\m,
\end{align*}
which coincides with the SORT method when $\gamma = \sigma = 0$. If friction is included in (\ref{eq:ham_system}), and we thus write $\frac{d}{dr}v_r^{\m n} = -\gamma v_r^{\m n} - u \nabla f\big(x_r^{n}\big)\m$, we may apply a change of variable to give
\begin{align}\label{eq:ham_system2}
\frac{d}{dr}\Bigg(\begin{matrix} \m x_r^{n}\\[-4pt] \m w_r^{n}\end{matrix}\Bigg) = \Bigg(\begin{matrix} e^{-\gamma r}w_r^{\m n}\\[-3pt] \, -u\m e^{\gamma r}\m \nabla f\big(x_r^{n}\big)\end{matrix}\Bigg),
\end{align}
where $w_r^{\m n} := e^{\gamma r}v_r^{\m n}\m$. By applying the third order Runge-Kutta method to this ODE and applying similar arguments as before, we arrive at the SORT method without noise terms.
Finally, by replacing $\nabla f(\,\cdot\,)$ with $\nabla f(\,\cdot\,) + \sigma\Big(\frac{W_n - 12 K_n}{h_n}\Big)$, we can obtain the desired method.

\subsection{Derivation of the SOFA method} To apply a splitting method, we express the ODE's vector field as a sum of two operators. Due to the structure of the ODE, we write
\begin{equation*}
\frac{d}{dr}\Bigg(\begin{matrix} \,\overline{x}_r^{\m n}\\[-4pt] \,\overline{v}_r^{\m n}\end{matrix}\Bigg) = \underbrace{\Bigg(\begin{matrix} \,\overline{v}_r^{\m n} h_n\\[-4pt] \,0\end{matrix}\Bigg)}_{=:\m A} + \underbrace{\Bigg(\begin{matrix} \,0\\[-4pt] \,- \gamma\m \overline{v}_r^{\m n}\m h_n - u \nabla f\big(\m\overline{x}_r^{\m n}\big)\m h_n + \sigma\big(W_n - 12K_n\big)\end{matrix}\Bigg)}_{=:\m B}.
\end{equation*}
In this case, the ODEs governed by $A$ and $B$ admit the following closed-form solutions:
\begin{align*}
\varphi_t^A\Bigg(\begin{matrix} x\\[-4pt] v\end{matrix}\Bigg) & := \Bigg(\begin{matrix} x + v\m h_n\m t\\[-4pt] v\end{matrix}\Bigg),\\
\varphi_t^B\Bigg(\begin{matrix} x\\[-4pt] v\end{matrix}\Bigg) & := \Bigg(\begin{matrix} \,x\\[-4pt] \,e^{-\gamma\m h_n t}v + \Big(\hspace{-0.25mm}-u\nabla f(x)\m h_n + \sigma\big(W_n - 12K_n\big)\Big)\Big(\frac{1 - e^{-\gamma h_n t}}{\gamma\m h_n}\Big)\end{matrix}\Bigg).
\end{align*}
It is important to note that these closed-form solutions are also valid when $t$ is negative.
Thus, it is straightforward to apply the fourth order splitting due to Forest and Ruth \cite{FourthOrderSplitting}:
\begin{equation*}
\Bigg(\begin{matrix} \,\overline{x}_1^{\m n}\\[-4pt] \,\overline{v}_1^{\m n}\end{matrix}\Bigg) \approx \varphi_{\frac{1}{2}+\phi}^B \circ \varphi_{1+2\phi}^A\circ \varphi_{-\phi}^B\circ \varphi_{-1-4\phi}^A\circ \varphi_{-\phi}^B\circ \varphi_{1+2\phi}^A\circ \varphi_{\frac{1}{2}+\phi}^B\Bigg(\begin{matrix} \,\overline{x}_0^{\m n}\\[-4pt] \,\overline{v}_0^{\m n}\end{matrix}\Bigg),
\end{equation*}
where
\begin{equation*}
\phi := \frac{-1 + \mysqrt{0}{3}{3}{2\m}\,}{2\big(2-\mysqrt{0}{3}{3}{2\m}\,\big)}\m.
\end{equation*}\bigbreak
By applying this splitting method to the shifted ODE, we obtain the SOFA method (definition \ref{def:sofa_method}). In addition to its high order of convergence, the above splitting method is symplectic when used on Hamiltonian systems (which the SDE (\ref{eq:ULD}) becomes with $\gamma = 0$).
Hence all that is required to show the SOFA method is quasi-symplectic in the sense of \cite{MilsteinTretyakov} is to check that the change in phase volume for every step does not depend on $\big(\m\overrightarrow{x\,}_{\nnm n}, \overrightarrow{v\,}_{\nnm n}\big)$:
\begin{align*}
\frac{D\big(\m\overrightarrow{x\,}_{\nnm n+1}, \overrightarrow{v\,}_{\nnm n+1}\big)}{D\big(\m\overrightarrow{x\,}_{\nnm n}, \overrightarrow{v\,}_{\nnm n}\big)} & = \frac{D\big(\m\overrightarrow{x\,}_{\nnm n+1}, \overrightarrow{v\,}_{\nnm n+1}\big)}{D\big(\m\overrightarrow{x\,}_{\nnm n}^{\m(2)}, \overrightarrow{v\,}_{\nnm n}^{\m(4)}\big)}\frac{D\big(\m\overrightarrow{x\,}_{\nnm n}^{\m(2)}, \overrightarrow{v\,}_{\nnm n}^{\m(4)}\big)}{D\big(\m\overrightarrow{x\,}_{\nnm n}^{\m(2)}, \overrightarrow{v\,}_{\nnm n}^{\m(3)}\big)}\frac{D\big(\m\overrightarrow{x\,}_{\nnm n}^{\m(2)}, \overrightarrow{v\,}_{\nnm n}^{\m(3)}\big)}{D\big(\m\overrightarrow{x\,}_{\nnm n}^{\m(2)}, \overrightarrow{v\,}_{\nnm n}^{\m(2)}\big)}\\
&\mmmm\frac{D\big(\m\overrightarrow{x\,}_{\nnm n}^{\m(2)}, \overrightarrow{v\,}_{\nnm n}^{\m(2)}\big)}{D\big(\m\overrightarrow{x\,}_{\nnm n}^{\m(1)}, \overrightarrow{v\,}_{\nnm n}^{\m(1)}\big)}\frac{D\big(\m\overrightarrow{x\,}_{\nnm n}^{\m(1)}, \overrightarrow{v\,}_{\nnm n}^{\m(1)}\big)}{D\big(\m\overrightarrow{x\,}_{\nnm n}, \overrightarrow{v\,}_{\nnm n}^{\m(0)}\big)}\frac{D\big(\m\overrightarrow{x\,}_{\nnm n}, \overrightarrow{v\,}_{\nnm n}^{\m(0)}\big)}{D\big(\m\overrightarrow{x\,}_{\nnm n}, \overrightarrow{v\,}_{\nnm n}\big)}\\[3pt]
& = 1\cdot e^{-\gamma(\frac{1}{2}+\phi)h_n} \cdot e^{\gamma \phi h_n} \cdot e^{\gamma \phi h_n}\cdot e^{-\gamma(\frac{1}{2}+\phi)h_n}\cdot 1\\
& = e^{-\gamma h_n},
\end{align*}
where $\big\{\big(\m\overrightarrow{x\,}_{\nnm n}^{\m(i)}, \overrightarrow{v\,}_{\nnm n}^{\m(i)}\big)\big\}$ are given by definition \ref{def:sofa_method} and we use the same notation as in \cite{MilsteinTretyakov}.
Similarly, any numerical method for ULD obtained by applying a symplectic splitting method to the shifted ODE will be quasi-symplectic. As well as the SOFA method, the Strang splitting method and OBABO scheme discussed in Section \ref{sect:other_methods} are quasi-symplectic.

\subsection{Error analysis of the shifted ODE method} In this section, we present our results concerning the convergence of the shifted ODE method and discuss the techniques used in our error analysis. Recall that we use $\{(\m\widetilde{x}_n, \widetilde{v}_n)\}_{n\m\geq\m 0}$ to denote the Markov chain generated by the shifted ODE method (at times $\{t_n\}$). Our main theorem is given below
\begin{theorem}[\textbf{Convergence of the shifted ODE method}]\label{thm:the_estimates} Let $\big\{(\widetilde{x}_n\m, \widetilde{v}_n)\big\}_{n\m \geq\m 0}$ denote the approximation of the underdamped Langevin diffusion $\big\{(x_t, v_t)\big\}_{t\m\geq\m 0}$ obtained using the shifted ODE method with a fixed step size $h>0$ and the same underlying Brownian motion.
We assume that the potential $f$ is $m$-strongly convex and has an $M$-Lipschitz continuous gradient so that the diffusion process has a unique stationary measure $\pi$. Suppose further that $(x_0, v_0)\sim \pi$ and both processes have the same initial velocity $\widetilde{v}_0 = v_0\sim \mathcal{N}\big(0,u\m I_d\big)$. Let $h_{\max} > 0$ be fixed. Then for $\lambda\in[0,\frac{\gamma}{2})$ there exists a constant $c_1 > 0$, independent of $h$ and $d$ such that for $h\leq h_{\max}$ and $n\geq 0$,
\begin{align}\label{eq:low_estimate}
\|\widetilde{x}_n - x_{t_n}\|_{\L_2} \leq \frac{\sqrt{2\lambda^2 + 2(\gamma - \lambda)^2}}{\gamma - 2\lambda}\,e^{-n\alpha h}\big\|\widetilde{x}_0 - x_0\big\|_{\L_2} + c_1 \sqrt{d\m}\m h^\frac{3}{2}\m,
\end{align}
where the contraction rate $\alpha$ is given by
\begin{align}\label{eq:contraction_rate}
\alpha & = \frac{\big((\gamma - \lambda)^2 - uM\big)\vee \big(um - \lambda^2\big)}{\gamma - 2\lambda}\,.
\end{align}
If we assume further that the Hessian of $f$ is $M_2$-Lipschitz continuous, then there exists a constant $c_2 > 0$, independent of $h$ and $d$ such that for $h\leq h_{\max}$ and $n\geq 0$,
\begin{align}\label{eq:med_estimate}
\|\widetilde{x}_n - x_{t_n}\|_{\L_2} \leq \frac{\sqrt{2\lambda^2 + 2(\gamma - \lambda)^2}}{\gamma - 2\lambda}\,e^{-n\alpha h}\big\|\widetilde{x}_0 - x_0\big\|_{\L_2} + c_2\m d\m h^\frac{5}{2}\m.
\end{align}
If we additionally assume that $f$ is three times continuously differentiable and the third derivative is $M_3$-Lipschitz continuous, then there exists a constant $c_3 > 0$, independent of $h$ and $d$ such that for $h\leq h_{\max}$ and $n\geq 0$,
\begin{align}\label{eq:high_estimate}
\|\widetilde{x}_n - x_{t_n}\|_{\L_2} \leq \frac{\sqrt{2\lambda^2 + 2(\gamma - \lambda)^2}}{\gamma - 2\lambda}\,e^{-\frac{1}{2}n\alpha h}\big\|\widetilde{x}_0 - x_0\big\|_{\L_2} + c_3\m d^{\m\frac{3}{2}}\m h^3.
\end{align}
\end{theorem}
\begin{remark}
These error estimates follow directly using Theorems \ref{thm:const_step_low_estimate}, \ref{thm:const_step_med_estimate} and \ref{thm:const_step_high_estimate}. In these theorems and throughout the appendix, the various constants are given explicitly.
As $W_2\big(\m\widetilde{\nu}_n\m , e^{-f}\m\big) \leq \|\widetilde{x}_n - x_{t_n}\|_{\L_2}$ where $\widetilde{x}_n\sim \widetilde{\nu}_n$, this theorem gives the results in Table \ref{table:ode_convergence}. 
\end{remark}
\begin{proof} From the definition of $\{y_n\}_{n\m\geq\m 0}$ (see the end of Section \ref{sect:notation} or Theorem \ref{thm:exp_contract}), we have
\begin{align*}
(\gamma - 2\lambda)^2\|\m\widetilde{x}_n - x_{t_n}\|_{\L_2}^2 & = \|(\eta - \lambda)(\m\widetilde{x}_n - x_{t_n})\|_{\L_2}^2\\
& = \big\|\big((\eta\m\widetilde{x}_n + \widetilde{v}_n) - (\eta\m x_{t_n} + v_{t_n})\big) - \big((\lambda\m\widetilde{x}_n + \widetilde{v}_n) - (\lambda\m x_{t_n} + v_{t_n})\big)\big\|_{\L_2}^2\\
&\leq 2\big\|(\lambda\m\widetilde{x}_n + \widetilde{v}_n) - (\lambda\m x_n + v_n)\big\|_{\L_2}^2 + 2\big\|(\eta\m\widetilde{x}_n + \widetilde{v}_n) - (\eta\m x_n + v_n)\big\|_{\L_2}^2\\
& = 2\|y_n\|_{\L_2}^2\m,
\end{align*}
where we obtained the third line using Minkoski's inequality. Since we can estimate each $\|\m\widetilde{x}_n - x_{t_n}\|_{\L_2}$ using $\|y_n\|_{\L_2}$, the inequalities (\ref{eq:low_estimate}), (\ref{eq:med_estimate}) and (\ref{eq:high_estimate}) now follow immediately from Theorems \ref{thm:const_step_low_estimate},  \ref{thm:const_step_med_estimate} and \ref{thm:const_step_high_estimate} respectively.
\end{proof}

In this paper, we are primarily interested in the dependence of the $\L_2$ error on both the dimension $d$ and step size $h$. However, under the standard assumptions on the potential ($f$ is $m$-strongly convex and $\nabla f$ is $M$-Lipschitz continuous), one is additionally interested in the dependence on the conditional number $\kappa := \frac{M}{m}$ (see, for example, Table 1 in \cite{MidpointMCMC}).
Using the explicit formulae derived in the appendices, it follows from Corollary \ref{cor:lipschitz_order} that
\begin{align}
c_1 & = \frac{2\m\sigma uM \Big(\frac{\sqrt{2}}{3} + \frac{\sqrt{10} + \sqrt{15}}{15}\Big)}{\big((\gamma - \lambda)^2 - uM\big)\vee \big(um - \lambda^2\big)} + O\big(h_{\max}\big).
\end{align}
Thus the leading term in the error estimate depends linearly on the condition number $\kappa$,
In particular, this means the shifted ODE method matches the $\mathcal{O}\big(\kappa\sqrt{d}\m h^{1.5}\big)$ convergence rate of the randomized midpoint method \cite{MidpointMCMC} (which was shown to be order optimal in \cite{ULDComplexity}).\medbreak

For the rest of the section, we shall outline our approach to deriving the estimates in Theorem \ref{thm:the_estimates}. Our full error analysis of ULD (with calculations) is given in the appendix. 

\medbreak
{\textbf{Step 1. (Change of variable)}} For our analysis, we rewrite the shifted ODE method:
\begin{align*}
\Bigg(\begin{matrix} \,\widetilde{x}_{n+1}\\[-3pt] \,\widetilde{v}_{n+1}\end{matrix}\Bigg)
:=
\Bigg(\begin{matrix} \,\widehat{x}_{t_{n+1}}^{\m n}\\[-3pt] \,\widehat{v}_{t_{n+1}}^{\m n}\end{matrix}\Bigg)
+ 12 K_n\Bigg(\begin{matrix} \,0\\[-3pt] \,\sigma\end{matrix}\,\Bigg),
\end{align*}
where $\big\{\big(\widehat{x}_t^{\m n}, \widehat{v}_t^{\m n}\big)\big\}_{t\m\in\m [t_n, t_{n+1}]}$ solves the following ODE,
\begin{align}\label{eq:linULE_2}
\frac{d}{dt}\Bigg(\begin{matrix} \,\widehat{x}^{\m n}\\[-3pt] \,\widehat{v}^{\m n}\end{matrix}\Bigg)
=
\Bigg(\begin{matrix} \,\widehat{v}^{\m n} + \sigma\big(H_n + 6K_n\big)\\[-3pt] \,-\gamma\big(\m\widehat{v}^{\m n} + \sigma\big(H_n + 6K_n\big)\big) - u\nabla f\big(\,\widehat{x}^{\m n}\big)\end{matrix}\,\Bigg)
+ \frac{W_n - 12K_n}{h_n}\,\Bigg(\begin{matrix} \,0\\[-3pt] \,\sigma\end{matrix}\,\Bigg)\m,
\end{align}
with initial condition $\big(\m\widehat{x}_{t_n}^{\m n}, \widehat{v}_{t_n}^{\m n}\big) := \big(\widetilde{x}_{n}\m, \widetilde{v}_n\big)$. This form of the shifted ODE method is convenient to work with as we can easily Taylor expand (\ref{eq:linULE_2}) from the point $\big(\widetilde{x}_{n}\m, \widetilde{v}_n\big)$.
Moreover, from the above, we see that the shifted ODE method reduces to another ODE approximation known as the ``log-ODE'' method (which is studied in \cite{OptimalPoly}) when $K_n = 0$.\medbreak

{\textbf{Step 2. (2-Wasserstein contractivity)}} Typically local errors propagate along the trajectory of the SDE solution so that we would not expect $\underset{n\m\geq\m 0}{\sup}\m\|\widetilde{x}_n - x_{t_n}\|_{\L_2}$ to exist.\vspace{0.75mm}
However such error estimates (that hold for all $n\geq 0$) can be established when the SDE exhibits 2-Wasserstein contractivity \cite{RungeKuttaMCMC}. For ULD, contractivity does not hold directly for the solution process $\{(x_t,v_t)\}_{t\m\geq\m 0}$ but instead holds after introducing new coordinates (this was first observed in \cite{LangevinContraction}). We use the recent approach given by \cite{KineticLangevinMCMC}, which considers
\begin{align}\label{eq:coordinate_change}
\Bigg\{\Bigg(\begin{matrix} \,\lambda\m x_t + v_t\\[-3pt] \,\eta\m x_t + v_t\end{matrix}\Bigg)\Bigg\}_{t\m\geq\m 0},
\end{align}
where $\lambda\in[0,\frac{1}{2}\gamma)$ and $\eta := \gamma - \lambda\m$, to obtain $2$-Wasserstein contractivity. However, unlike in previous works, we utilized the $2$-Wasserstein contractivity of the approximation process.
That is, using a synchronous coupling and essentially the same proof as in \cite{KineticLangevinMCMC}, we establish a $2$-Wasserstein contraction property for the shifted ODE approximation (Theorem \ref{thm:exp_contract}).
Alternatively we could have used the straightforward coupling approach proposed by \cite{ChengMCMC}, however this would have restricted our choice of parameters for ULD to $\gamma = 2$ and $u = \frac{1}{M}$. \medbreak

We then follow the arguments of \cite{RungeKuttaMCMC} with the slight modification that the contraction is applied to the approximation process instead of the diffusion. This is illustrated below:\vspace{-6mm}
\begin{figure}[H]
\begin{center}
\includegraphics[width=\textwidth]{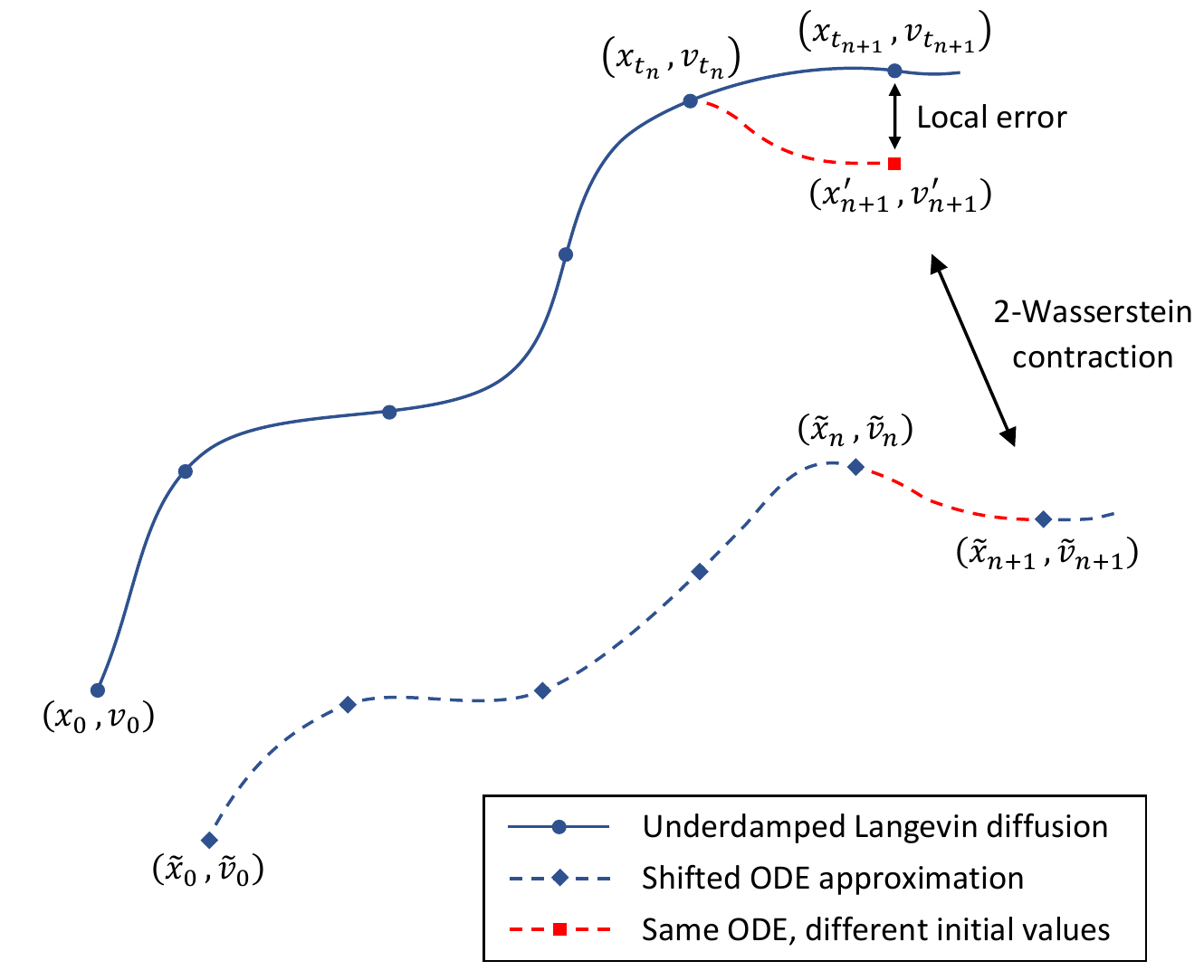}
\caption{Graph outlining the general strategy for our error analysis.}\label{error_diagram}
\end{center}
\end{figure}\vspace{-11mm}

By introducing a third process (obtained by applying the shifted ODE method to the diffusion at each time $t_n$), we can estimate the error between the coordinate changed approximation and diffusion processes at time $t_{n+1}$ in terms of the same error at time $t_n$. \medbreak

{\textbf{Step 3. (Local error estimates without global bounds)}} From figure \ref{error_diagram}, it is clear that we need to derive error estimates between $(x_{t_{n+1}}, v_{t_{n+1}})$ and $(x_{t_{n+1}}^\prime, v_{t_{n+1}}^\prime)$.
Using the Lipschitz continuity of $\nabla f$ along with the fact that both processes satisfy similar differential equations, it will be straightforward to derive $\L_p$ estimates that are $O\big(\sqrt{d\m}\m h^{2.5}\big)$ (Theorem \ref{thm:first_local_estimate}). Using the $2$-Wasserstein contractivity of the shifted ODE approximation, this gives a global $\L_2$ error estimate between $\{(x_{t_n}, v_{t_n})\}$ and $\{(\widetilde{x}_n, \widetilde{v}_n)\}$ that is $\mathcal{O}\big(\sqrt{d\m}\m h^{1.5}\big)$.

\medbreak

{\textbf{Step 4. (Global $\L_p$ bounds)}} To derive the higher order error estimates, we
require that the diffusion and approximation processes are globally bounded (in an $\L_p$ sense).
Since the underdamped Langevin diffusion is well studied, it has already been shown that
\begin{align}\label{eq:intro_bounds1}
\sup_{t\m\geq\m 0}\big\|v_t\big\|_{\L_p} \leq C(p)\sqrt{d\m}\hspace{5mm}\text{and}\hspace{5mm} \sup_{t\m\geq\m 0}\big\|\nabla f(x_t)\big\|_{\L_2}  \leq \sqrt{Md\m},
\end{align}
for $p\geq 1$ where $(x_0\m, v_0)\sim \pi$ and $C(p)$ is a positive constant depending only on $p$ \cite{DalalyanLangevin}.
However, unlike in \cite{KineticLangevinMCMC}, these bounds were not sufficient to derive high order estimates.
Therefore the main result in this part of our error analysis is the following global $\L_4$ bound:
\begin{align}\label{eq:intro_bounds2}
\sup_{t\m\geq\m 0}\big\|\nabla f(x_t)\big\|_{\L_4}  \leq \fourthree\,\sqrt{Md\m}\m.
\end{align}

\medbreak

{\textbf{Step 5. (Local error estimates with global bounds)}} Using these bounds and the Lipschitz continuity of $\nabla f$ and $\nabla^2 f$, we can estimate the local $\L_2$ error as $O\big(d\m h^{3.5}\big)$.
We obtain this estimate by Taylor expanding the ``$\nabla f$ terms'' in both the SDE and ODE.
By the same arguments as before, we arrive at a global $\L_2$ error estimate that is $\mathcal{O}\big(d\m h^{2.5}\big)$.
However, as stated in Theorem \ref{thm:the_estimates}, it is possible to derive a third order $\L_2$ error estimate.
To do this, we use the same strategy of proof as \cite{RungeKuttaMCMC} and consider the local mean deviation:
\begin{align*}
\Bigg\|\E_n\Bigg[\Bigg(\begin{matrix} \, x_{n+1}^\prime - x_{t_{n+1}}\\[-3pt] \,v_{n+1}^\prime - v_{t_{n+1}}\end{matrix}\Bigg)\Bigg]\Bigg\|_{\L_2}\m.
\end{align*}
By extending our Taylor expansion of the ``$\nabla f$ terms'' using the third derivative $\nabla^3 f$, we can show that that the local mean deviation is $O\big(d^{1.5}h^4\big)$. Intuitively, this means the $O\big(h^{3.5}\big)$ component of the local error is unbiased and thus propagates like a random walk.
Since random walks grow like $O(\sqrt{N}\m)$, this leads to a global $\L_2$ error estimate of $\mathcal{O}\big(h^3\big)$.
To apply this argument rigorously, we follow the proof of Theorem 1 in \cite{RungeKuttaMCMC} and obtain the desired $\mathcal{O}\big(d^{1.5}h^3\big)$ estimate. Due to some technical details, this error estimate only has half of the $2$-Wasserstein contraction rate as the $\mathcal{O}\big(\sqrt{d\m}\m h^{1.5}\big)$ and $\mathcal{O}\big(d\m h^{2.5}\big)$ estimates.

\section{Related derivative-free numerical methods for ULD}\label{sect:other_methods}

In this section, we will survey some of the numerical methods recently proposed for ULD.
Before we discuss methods specifically designed for Langevin dynamics, we first note a connection between the shifted ODE method and the ``log-ODE'' method studied in \cite{OptimalPoly}.

\begin{definition}[\textbf{Log-ODE method for ULD}]
When applied to SDE (\ref{eq:ULD}), each step of the log-ODE method (Definition 3.12 in \cite{OptimalPoly}) can be written as $\big(\m\overrightarrow{x\,}_{\nnm n + 1}, \overrightarrow{v\,}_{\nnm n + 1}\big) := \big(Q_1^n, P_1^n\big)$
where $\big\{\big(Q_t^n\m, P_t^n\big)\big\}_{t\in[0,1]}$ solves the following ODE:
\begin{align}\label{eq:logODE}
\frac{d}{dt}\Bigg(\begin{matrix} \m Q^n\\[-3pt] \m P^n\end{matrix}\Bigg)
=
\Bigg(\begin{matrix} \,P^n\\[-3pt] \,-\gamma P^n - u\nabla f\big(Q^n\big)\end{matrix}\,\Bigg)h_n
+ \Bigg(\begin{matrix} \,0\\[-3pt] \,\sigma\end{matrix}\,\Bigg)W_n + \Bigg(\begin{matrix} \,\sigma\m \\[-3pt] \,-\gamma\sigma\end{matrix}\,\Bigg)h_n H_n\m,
\end{align}
with initial condition $\big(Q_0^n\m, P_0^n\big) := \big(\m\overrightarrow{x\,}_{\nnm n}\m, \overrightarrow{v\,}_{\nnm n}\big)$.
\end{definition}

Using the change of variables discussed in the previous section, it follows that the log-ODE (\ref{eq:logODE}) can be obtained from the shifted ODE (\ref{eq:linULE_2}) simply by setting $K_n = 0$.
Since the log-ODE method does not use the random vectors $\{K_n\}_{n\m\geq\m 0}$, it cannot match as many terms in the stochastic Taylor expansion of ULD as the shifted ODE method.
So although we expect the log-ODE method to achieve a second order convergence rate when applied to ULD, it is likely to be inferior to the higher order shifted ODE method.\medbreak

We shall now discuss four different methods for the underdamped Langevin diffusion.

\begin{definition}[{\textbf{Strang splitting}} \cite{IreneStrang}]\label{def:strang} We construct a numerical solution
$\big\{\big(\overrightarrow{x\,}_{\nnm n}, \overrightarrow{v\,}_{\nnm n}\big)\big\}_{n\m\geq\m 0}$
for the SDE (\ref{eq:ULD}) by setting $\big(\overrightarrow{x\,}_{\nnm 0}, \overrightarrow{v\,}_{\nnm 0}\big) := (x_0\m, v_0)$
and for $n\geq 0$, defining $\big(\overrightarrow{x\,}_{\nnm n + 1}, \overrightarrow{v\,}_{\nnm n + 1}\big)$ as
\begin{align*}
\overrightarrow{v\,}_{\nnm n}^{\m(1)} & := \overrightarrow{v\,}_{\nnm n} - \frac{1}{2}u\nabla f\big(\m\overrightarrow{x\,}_{\nnm n}\big)\m h_n\m,\\[2pt]
\Bigg(\begin{matrix} \,\overrightarrow{x\,}_{\nnm n+1}\\[-1pt] \,\overrightarrow{v\,}_{\nnm n}^{\m(2)}\end{matrix}\Bigg)
&=
\Bigg(\begin{matrix} \,\overrightarrow{x\,}_{\nnm n}\\[-1pt] \,e^{-\gamma h_n}\overrightarrow{v\,}_{\nnm n}^{\m(1)}\end{matrix}\,\Bigg) +
\Bigg(\begin{matrix} \,\Big(\frac{1-e^{-\gamma h_n}}{\gamma}\Big)\overrightarrow{v\,}_{\nnm n}^{\m(1)}\\[-1pt] \,0\end{matrix}\,\Bigg) + \sigma\Bigg(\begin{matrix} \int_{t_n}^{t_{n+1}}\nm\int_{t_n}^t e^{-\gamma(t - s)}dW_s\m dt \\[-1pt] \int_{t_n}^{t_{n+1}}e^{-\gamma(t_{n+1} - t)}dW_t\end{matrix}\,\Bigg)\m,\\[2pt]
\overrightarrow{v\,}_{\nnm n + 1} & := \overrightarrow{v\,}_{\nnm n}^{\m(2)} - \frac{1}{2}u\nabla f\big(\m\overrightarrow{x\,}_{\nnm n+1}\big)\m h_n\m.
\end{align*}
\end{definition}

The above splitting method is obtained by expressing the SDE (\ref{eq:ULD}) as the sum of two subsystems but unlike in the derivation of the SOFA method, we split the SDE as follows:
\begin{equation*}
d\Bigg(\begin{matrix} x_t\\[-3pt] v_t\end{matrix}\Bigg) = \underbrace{\Bigg(\begin{matrix} \,v_t\,dt\\[-3pt] -\gamma\m v_t\,dt + \sigma\, dW_t \end{matrix}\Bigg)}_{=:\m A} + \underbrace{\Bigg(\begin{matrix} \,0\\[-3pt] - u \nabla f(x_t)\,dt\end{matrix}\Bigg)}_{=:\m B}.
\end{equation*}

The subsystem $A$ corresponds to a physical Brownian motion \cite{PhysicalBM} and thus can be sampled exactly. The subsystem $B$ can be solved simply by evaluating the gradient $\nabla f$.
The stochastic integrals which appear in definition \ref{def:strang} follow a (joint) normal distribution.
\begin{theorem}[Distribution of integrals required to simulate physical Brownian motion] Let $0\leq s\leq t$. Then
\begin{align*}
\begin{pmatrix}
\int_s^t e^{-\gamma(t - \tau)}dW_\tau\\
\int_s^t\int_s^\tau e^{-\gamma(\tau - r)}dW_r\, d\tau
\end{pmatrix} & \sim \mathcal{N}\left(\Bigg(\begin{matrix}0\\[-3pt] 0\end{matrix}\Bigg)\,,\begin{pmatrix} \frac{1-e^{-2\gamma h}}{2\gamma}\m I_d & \frac{(1 - e^{-\gamma h})^2}{2\gamma^2}\m I_d\\[3pt]
\frac{(1 - e^{-\gamma h})^2}{2\gamma^2}\m I_d &  \frac{4e^{-\gamma h} - e^{-2\gamma h} + 2\gamma h - 3}{2\gamma^3}\m I_d\end{pmatrix}\right),
\end{align*}
where $h := t -s$.
\end{theorem}
We refer the reader to Appendix \ref{appen:integrals} for additional details on the stochastic integrals used by numerical methods for ULD. We note that the second integral is often written as
\begin{equation*}
\int_s^t\int_s^\tau e^{-\gamma(\tau - r)}dW_r\, d\tau = \int_s^t \bigg(\frac{1 - e^{-\gamma(t - r)}}{\gamma}\bigg)\,dW_r\m,
\end{equation*}
by Fubini's theorem. We prefer the first form as it is clearer how it comes from SDE $A$.\medbreak

By considering the stochastic Taylor expansion of the above Strang splitting, we would expect the method to achieve an $\L_2$ convergence rate of $\mathcal{O}\big(\sqrt{d\m}h\m\big)$ when $\nabla f$ is $M$-Lipschitz continuous and $\mathcal{O}\big(d\m h^2\m\big)$ if, in addition, $\nabla^2 f$ is also assumed to be Lipschitz continuous. This conjecture is supported by our numerical experiment, where the Strang splitting exhibited a second order rate of convergence. Another particularly attractive feature of this numerical method is that it only requires one additional evaluation of $\nabla f$ per step (since, just as for our methods, $\nabla f\big(\m\overrightarrow{x\,}_{\nnm n}\big)$ can be computed during the previous step).
Whilst this Strang splitting does not appear to be extensively studied in the literature,
we believe that it offers an attractive compromise between accuracy and computational cost.\medbreak

As previously stated, the splitting used to derive the SOFA method differs from the one used to derive the above Strang splitting. In some sense, the second splitting is more desirable as the associated numerical methods reduce to the exact solution when $\nabla f = 0$.
We derived the SOFA method using the first operator splitting simply because it results in less computations for the position component. However, this may be negligible in practice since both splitting approaches require the same number of gradient evaluations per step.
In any case, the choice of method for the shifted ODE is likely to be problem dependent. \medbreak

The following numerical scheme can be viewed as a variant of the Strang splitting.
Although we expect it only have a first order $\L_2$ convergence rate, it has the additional advantage that it can be Metropolis-adjusted (in a similar fashion to standard HMC).
Moreover, it was shown in \cite{OBABOtheory} to achieve a second order $2$\hspace{0.125mm}-Wasserstein convergence rate.

\begin{definition}[{\textbf{(Unadjusted) OBABO scheme}} \cite{OBABOphysics, OBABOtheory, HamsMC}] Similar to Strang splitting,
we construct a numerical solution $\big\{\big(\overrightarrow{x\,}_{\nnm n}, \overrightarrow{v\,}_{\nnm n}\big)\big\}_{n\m\geq\m 0}$ for (\ref{eq:ULD}) by setting $\big(\overrightarrow{x\,}_{\nnm 0}, \overrightarrow{v\,}_{\nnm 0}\big) := (x_0\m, v_0)$
and for $n\geq 0$, defining $\big(\overrightarrow{x\,}_{\nnm n + 1}, \overrightarrow{v\,}_{\nnm n + 1}\big)$ as
\begin{align*}
\hspace{5mm}\overrightarrow{v\,}_{\nnm n}^{\m(0)} & := e^{-\frac{1}{2}\gamma h_n} \overrightarrow{v\,}_{\nnm n} + \sigma\int_{t_n}^{t_n + \frac{1}{2}h_n}e^{-\gamma(t_n + \frac{1}{2}h_n - t)}dW_t\m,\hspace{5.25mm}(O)\\[3pt]
\hspace{5mm}\overrightarrow{v\,}_{\nnm n}^{\m(1)} & := \overrightarrow{v\,}_{\nnm n}^{\m(0)} - \frac{1}{2}u\nabla f\big(\m\overrightarrow{x\,}_{\nnm n}\big)\m h_n\m,\hspace{38.5mm}(B)\\[3pt]
\hspace{5mm}\overrightarrow{x\,}_{\nnm n + 1} & := \overrightarrow{x\,}_{\nnm n} + \overrightarrow{v\,}_{\nnm n}^{\m(1)}h_n\m,\hspace{53.25mm}(A)\\[3pt]
\hspace{5mm}\overrightarrow{v\,}_{\nnm n}^{\m(2)} & := \overrightarrow{v\,}_{\nnm n}^{\m(1)} - \frac{1}{2}u\nabla f\big(\m\overrightarrow{x\,}_{\nnm n+1}\big)\m h_n\m,\hspace{35mm}(B)\\[3pt]
\hspace{5mm}\overrightarrow{v\,}_{\nnm n + 1} & := e^{-\frac{1}{2}\gamma h_n} \overrightarrow{v\,}_{\nnm n}^{\m(2)} + \sigma\int_{t_n + \frac{1}{2}h_n}^{t_{n+1}}e^{-\gamma(t_{n+1} - t)}dW_t\m,\hspace{9.25mm}(O)\\[1pt]
\hspace{5mm}\overrightarrow{x\,}_{\nnm n + \frac{1}{2}} & := \overrightarrow{x\,}_{\nnm n} + \frac{1}{4}\big(\overrightarrow{v\,}_{\nnm n}^{\m(0)} + \overrightarrow{v\,}_{\nnm n}^{\m(1)}\big)h_n\m.
\end{align*}
\end{definition}

Just as before, the OBABO scheme requires one additional evaluation of $\nabla f$ per step.
Its name is derived from the order in which the damping/sampling (O), acceleration (A) and free transport parts (B) of the dynamics are applied. For example, there are also BAOAB and ABOBA methods \cite{BAOmethods} (which we view as being similar to Strang splitting).
We note the scheme is straightforward to implement as the stochastic integrals used in the (O) operations are independent (since Brownian motion has independent increments).\medbreak

We refer the reader to \cite{OBABOtheory, HamsMC} for details on the Metropolis-adjusted OBABO scheme.
In particular, the adjusted OBABO scheme was shown to empirically outperform the standard HMC method in \cite{HamsMC} (though it was only the third best performing approach).
Based on their stochastic Taylor expansions, the Strang splitting should outperform the unadjusted OBABO scheme for $\L_2$ convergence and this is supported by our experiment.
We also note that $\big\{\overrightarrow{x\,}_{\nnm n + \frac{1}{2}}\big\}$ resembles a chain produced by an ABOBA splitting, which has second order $\L_2$ convergence \cite{BAOmethods}.
Since the OBABO scheme and Strang splitting have the same $2$\hspace{0.125mm}-Wasserstein convergence rates, it is not clear how they would compare in practice.\medbreak

It was shown in \cite{ULDComplexity} that the following method achieves an optimal $\L_2$ convergence rate of $\mathcal{O}\big(\sqrt{d\m}\m h^{1.5}\big)$ under minimal regularity assumptions (i.e.~$\nabla f$ is $M$-Lipschitz continuous).
Unlike the previously discussed methods, here the gradient $\nabla f$ is evaluated at a uniformly sampled point within each interval $[t_n, t_{n+1}]$ (referred to as the ``randomized midpoint'').

\begin{definition}[{\textbf{Randomized midpoint method}} \cite{MidpointMCMC, MidpointMCMCFollowup}] 
Let $\{\alpha_n\}_{n\m\geq\m 0}$ be a sequence of independent random variables with $\alpha_n\sim U[0,1]$. We construct a numerical solution $\big\{\big(\overrightarrow{x\,}_{\nnm n}, \overrightarrow{v\,}_{\nnm n}\big)\big\}_{n\geq 0}$ by setting $\big(\overrightarrow{x\,}_{\nnm 0}, \overrightarrow{v\,}_{\nnm 0}\big) := (x_0\m, v_0)$
and for $n\geq 0$, defining $\big(\overrightarrow{x\,}_{\nnm n + 1}, \overrightarrow{v\,}_{\nnm n + 1}\big)$ as
\begin{align*}
\overrightarrow{x\,}_{\nnm n}^{(1)} & := \overrightarrow{x\,}_{\nnm n} + \bigg(\frac{1-e^{-\gamma \alpha_n h_n}}{\gamma}\bigg)\overrightarrow{v\,}_{\nnm n} - u\bigg(\frac{e^{-\gamma \alpha_n h_n} + \gamma\m \alpha_n h_n - 1}{\gamma^2}\bigg)\nabla f\big(\m\overrightarrow{x\,}_{\nnm n}\big)\\
&\hspace{11.05mm} + \sigma\int_{t_n}^{t_n + \alpha_n h_n}\nm\int_{t_n}^t e^{-\gamma(t - s)}dW_s\m dt,\\[3pt]
\overrightarrow{x\,}_{\nnm n+1} & := \overrightarrow{x\,}_{\nnm n} + \bigg(\frac{1-e^{-\gamma h_n}}{\gamma}\bigg)\overrightarrow{v\,}_{\nnm n} - u h_n\bigg(\frac{1-e^{-\gamma(1-\alpha_n) h_n}}{\gamma}\bigg)\nabla f\big(\m\overrightarrow{x\,}_{\nnm n}^{(1)}\big)\\
&\hspace{11mm} + \sigma\int_{t_n}^{t_{n+1}}\nm\int_{t_n}^t e^{-\gamma(t - s)}dW_s\m dt,\\[3pt]
\overrightarrow{v\,}_{\nnm n+1} & := e^{-\gamma h_n} \overrightarrow{v\,}_{\nnm n} - u\big(e^{-\gamma(1-\alpha_n) h_n}h_n\big)\nabla f\big(\m\overrightarrow{x\,}_{\nnm n}^{(1)}\big) + \sigma\int_{t_n}^{t_{n+1}}e^{-\gamma(t_{n+1} - t)}dW_t\m.
\end{align*}
\end{definition}

Unlike the previously discussed approaches, the randomized midpoint method does not exploit the smoothness of $f$ and requires two additional evaluations of $\nabla f$ per step. 
However, if the gradient $\nabla f$ is only Lipschitz continuous, the randomized midpoint method should outperform the other unadjusted Langevin MCMC methods discussed in this paper.
In our numerical experiment, the randomized midpoint method demonstrates  $O(h^{1.5})$ convergence, but is outperformed by the SORT, SOFA and Strang splitting methods since
the target log-density comes from a logistic regression and is thus infinitely differentiable. 
Generating the stochastic integrals is also straightforward and is detailed in Appendix \ref{appen:integrals}.\medbreak

The following numerical method for ULD (and the error analysis used to analyse it)
inspired much of our research. Since it does not yet appear to have a widely used name, we shall refer to this approach as the ``left-point method''. The left-point method only requires one evaluation of $\nabla f$ per step and achieves an $\L_2$ convergence rate of  $\mathcal{O}\big(\sqrt{d\m}\m h\big)$.
It is obtained by solving (\ref{eq:ULD}) exactly with $\nabla f(x_t)$ assumed to be constant on $[t_n\m, t_{n+1}]$.
Thus, it differs from the standard Euler\hspace{0.125mm}-Maruyama method as only the $\nabla f$ term is fixed.

\begin{definition}[{\textbf{Left-point method}} \cite{ChengMCMC, KineticLangevinMCMC, ChengMCMCfollowup}] We define a numerical solution
$\big\{\big(\overrightarrow{x\,}_{\nnm n}, \overrightarrow{v\,}_{\nnm n}\big)\big\}$
for the SDE (\ref{eq:ULD}) by setting $\big(\overrightarrow{x\,}_{\nnm 0}, \overrightarrow{v\,}_{\nnm 0}\big) := (x_0\m, v_0)$
and for $n\geq 0$, defining $\big(\overrightarrow{x\,}_{\nnm n + 1}, \overrightarrow{v\,}_{\nnm n + 1}\big)$ as
\begin{align*}
\overrightarrow{x\,}_{\nnm n+1} & := \overrightarrow{x\,}_{\nnm n} + \bigg(\frac{1-e^{-\gamma h_n}}{\gamma}\bigg)\overrightarrow{v\,}_{\nnm n} - \bigg(\frac{e^{-\gamma h_n} + \gamma h_n - 1}{\gamma^2}\bigg)u\nabla f\big(\m\overrightarrow{x\,}_{\nnm n}\big) + \sigma\int_{t_n}^{t_{n+1}}\nm\int_{t_n}^t e^{-\gamma(t - s)}dW_s\m dt,\\[3pt]
\overrightarrow{v\,}_{\nnm n+1} & := e^{-\gamma h_n} \overrightarrow{v\,}_{\nnm n} - \bigg(\frac{1-e^{-\gamma h_n}}{\gamma}\bigg)u\nabla f\big(\m\overrightarrow{x\,}_{\nnm n}\big) + \sigma\int_{t_n}^{t_{n+1}}e^{-\gamma(t_{n+1} - t)}dW_t\m.
\end{align*}
\end{definition}

\section{Numerical experiment}\label{sect:experiment}

In this section, we shall empirically compare the proposed SOFA method with the other unadjusted Langevin MCMC algorithms previously discussed. In our experiment, the target density $\pi(\theta)\propto \exp(-f(\theta))$ comes from a logistic regression on a real-world dataset (German credit data in \cite{UCI}) where there are $m = 1000$ individuals, each with $d = 49$ features $x_i\in\R^d$ and a label $y_i\in\{-1,1\}$ indicating whether they are creditworthy or not.
The model states that $\mathbb{P}(Y_i = y_i\m |\m x_i) = \big(1 + e^{-y_i\m x_i^\T \theta}\m\big)^{-1}$ and so the function $f$ is defined as
\begin{equation*}
f(\theta) = \frac{\delta}{2}\|\theta\|_2^2 + \sum_{i=1}^m \log\hspace{-0.25mm}\big(1 + \exp(-y_i \m x_i^\T\theta)\big)\m,
\end{equation*}
where $\delta$ is a regularization parameter that we set to $\delta = 0.1$. We also set $\gamma = 2$ and $u = 1$ so that the SDE governing the underdamped Langevin diffusion (equation (\ref{eq:ULD})) becomes
\begin{align}\label{eq:ULD2}
d\theta_t = v_t\,dt,\hspace{7.5mm} dv_t & = -2\m v_t\,dt - \nabla f(\theta_t)\,dt + 2\, dW_t\m.
\end{align}
In the regression, the weight vector $\theta$ takes values in $\R^{49}$. For example, this means that $W$ is a $d$-dimensional Brownian motion with $d = 49$. We generate $\theta_0$ using a normal prior:
\begin{equation*}
\theta_0 \sim \mathcal{N}\Big(0, 10 I_d\Big)\m.
\end{equation*}

We use the following Monte Carlo estimator as a proxy for the $\L_2$ approximation error.

\begin{definition}[\textbf{Strong error estimator}] Let $T > 0$ be a fixed time and $\big\{\overrightarrow{\theta\,}_{\nnm k}^{h\begin{matrix}\\[-12pt]\end{matrix}}\big\}_{k\m\geq\m 0}$ denote a numerical solution of (\ref{eq:ULD2}) computed at times $t_k := kh$ using step size $h = \frac{T}{N}$ where $N\geq 1$. Similarly, let $\big\{\overrightarrow{\theta\,}_{\nnm k}^{\frac{1}{2}h}\big\}_{k\m\geq\m 0}$ be the approximation obtained using step size $\frac{1}{2}h$.
We generate $n$ samples of these numerical solutions (at time $T$) and define the estimator
\begin{equation}\label{eq:strong_estimator}
S_{N,n} := \sqrt{\,\frac{1}{n}\sum_{i=1}^n\big\|\overrightarrow{\theta\,}_{\nnm N, i}^{h\begin{matrix}\\[-12pt]\end{matrix}} - \overrightarrow{\theta\,}_{\nnm 2N, i}^{\frac{1}{2}h}\big\|_2^2\,}\,,
\end{equation}
where each pair $\big(\overrightarrow{\theta\,}_{\nnm N, i}^{h\begin{matrix}\\[-12pt]\end{matrix}}\m, \overrightarrow{\theta\,}_{\nnm 2N, i}^{\frac{1}{2}h}\m\big)$ is computed from the same sample path of Brownian motion
and each initial condition $\overrightarrow{\theta\,}_{\nnm 0, i}^{h\begin{matrix}\\[-12pt]\end{matrix}} = \overrightarrow{\theta\,}_{\nnm 0, i}^{\frac{1}{2}h}$ is sampled from the normal distribution $\mathcal{N}\big(0, 10\m I_d\big)$.
\end{definition}
\begin{remark}
By the law of large numbers, the estimator (\ref{eq:strong_estimator}) converges as $n\rightarrow\infty$ to
\begin{equation*}
S_{N} := \big\|\overrightarrow{\theta\,}_{\nnm N}^{h\begin{matrix}\\[-12pt]\end{matrix}} - \overrightarrow{\theta\,}_{\nnm 2N}^{\frac{1}{2}h}\big\|_{\L_2}\,,
\end{equation*}
almost surely. For large step sizes, $S_N$ may be not close to the true $\L_2$ error $\big\|\overrightarrow{\theta\,}_{\nnm N}^{h\begin{matrix}\\[-12pt]\end{matrix}} - \theta_{\m T}\big\|_{\L_2}$.
\end{remark}

For the experiment, we use a time horizon of $T = 1000$, which is greater than the mixing time of Hamiltonian Monte Carlo for this specific problem (see example 3.2 in \cite{LLag}).
We compute the error estimator (\ref{eq:strong_estimator}) using $n = 100$ independent samples of $\big(\overrightarrow{\theta\,}_{\nnm N}^{h\begin{matrix}\\[-12pt]\end{matrix}}\m, \overrightarrow{\theta\,}_{\nnm 2N}^{\frac{1}{2}h}\m\big)$.\medbreak

We will now present our results for the numerical experiment that is described above.
Code for this experiment can be found at \href{https://github.com/james-m-foster/high-order-langevin}{github.com/james-m-foster/high-order-langevin}.\vspace{-6mm}

\begin{figure}[H]
\begin{center}
\includegraphics[width=\textwidth]{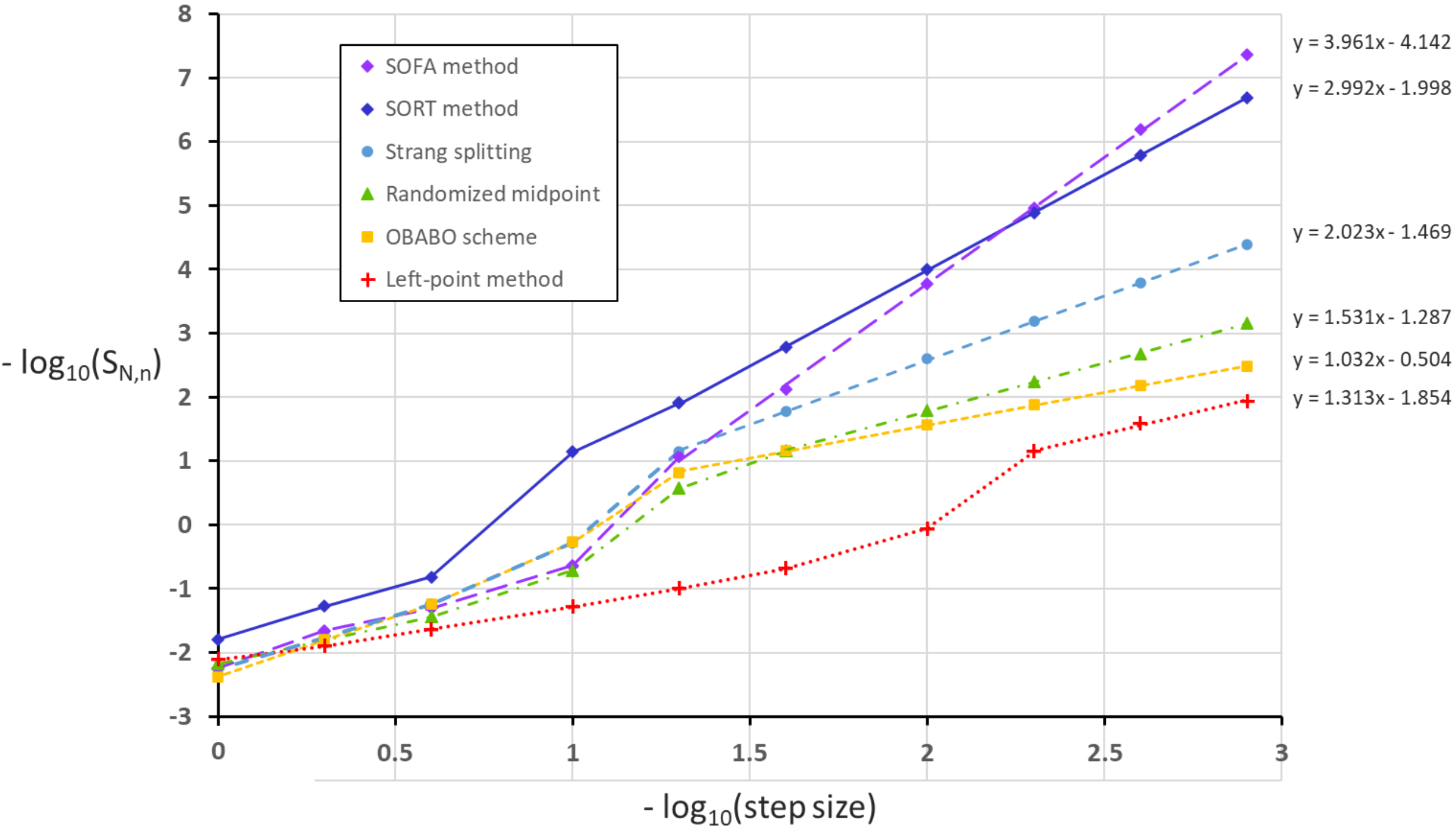}
\caption{Graph showing $S_{N,n}$ computed for various numerical methods and step sizes $h = \frac{T}{N}\m$.}\label{fig:convergence_graph}
\end{center}
\end{figure}\vspace{-11mm}

From the above graph, we see that the SORT and SOFA methods are the most accurate when a sufficiently small step size is used. For example, with a step size of $h = 0.005$,
the ``shifted ODE'' methods are roughly 50 times more accurate than the Strang splitting.
Apart from the SOFA method, all of the numerical methods exhibit the convergence rates that one would expect (see Section \ref{sect:other_methods} for details). In particular, we note that the OBABO scheme can achieve second order $2$-Wasserstein convergence to the target distribution $\pi$.
However, since the estimator (\ref{eq:strong_estimator}) corresponds to the $\L_2$ convergence to the SDE solution, the OBABO scheme only achieves first order convergence in our numerical experiment.
Since it is difficult to estimate the $2$-Wasserstein distance between approximations and solutions of high-dimensional underdamped Langevin dynamics (see example 4.3 in \cite{CouplingForSDEs}), we were unable to empirically show the second order convergence of the OBABO scheme.\medbreak

Perhaps the most surprising aspect of Figure \ref{fig:convergence_graph} is that the SOFA method appears to exhibit a fourth order convergence rate (whereas the theory only suggests a third order).
One possible explanation for this is that the constant in front of the ``$\hspace{0.125mm} O(h^3)$ component'' of $S_N$ is negligible for the range of steps sizes used in the numerical experiment. That is,
the noise term in the SDE (\ref{eq:ULD2}) has a small impact on the accuracy of the approximation (recall that the SOFA method will reduce to a fourth order splitting method when $\sigma = 0$).\medbreak

Overall the shifted ODE and Strang splitting methods exhibit the fastest convergence. Furthermore, the SORT and SOFA methods with a step size of $h = 0.01$ are comparable
to the Strang splitting with smaller step size of $h = 0.025$ (and are thus more efficient).
For MCMC-type applications, the OBABO scheme may also perform well since it can be Metropolis-adjusted and achieves a second order $2$-Wasserstein convergence rate when $\nabla f$, $\nabla^2 f$ are Lipschitz continuous. If the target log-density $f$ only has Lipschitz gradient, then the randomized midpoint method has the optimal order of strong convergence \cite{ULDComplexity}.
However, for applications where approximate Langevin MCMC is suitable and the target log-density $f$ has some degree of smoothness, the theory suggests that the shifted ODE is state-of-the-art and this is supported by our experiment for the SORT and SOFA methods.

\section{Conclusion}

In this paper, we present an ODE approximation for underdamped Langevin dynamics
and establish non-asymptotic error estimates (in $L^2(\mathbb{P})$) under different assumptions on $f$
which then give bounds on the $2$-Wasserstein convergence to the stationary distribution.
Most notably, we show that the ``shifted ODE'' method exhibits third order convergence when $\nabla f, \nabla^2 f$ and $\nabla^3 f$ are all Lipschitz continuous (that is, it achieves $W_2\big(\m\widetilde{x}_n\m , e^{-f}\m\big) \leq \varepsilon$ in $\mathcal{O}\big(\m\sqrt{d\m}/\varepsilon^{\frac{1}{3}}\m\big)$ steps). To the best of our knowledge, this is the first approximation of
ULD to achieve $1/\varepsilon^{\frac{1}{3}}$ dependence without requiring one to evaluate the higher derivatives of $f$.
By discretizing the shifted ODE using standard high order ODE solvers, we derive practical MCMC algorithms that exhibited third order convergence in our numerical experiment.
In addition to the proposed methods, we found that the Strang splitting (definition \ref{def:strang}) is a conceptually simple method that does not appear to be extensively studied in the literature yet offers an attractive compromise between accuracy and computational cost.\medbreak

\section{Future work}

We believe the SORT method requires further analysis and can be potentially improved.
For example, by modifying the noise terms, we arrive at the following numerical method:

\begin{definition}[\textbf{A SORT-like method with midpoint computation}] Let $\{t_n\}_{n\m\geq\m 0}$ be a sequence of times with $t_0 = 0$,
$t_{n+1} > t_{n}$ and step sizes $h_n = t_{n+1} - t_n\m$, for $n\geq 0$. We construct a numerical solution
$\{(\overrightarrow{x\,}_{\nnm n}\m, \overrightarrow{x\,}_{\nnm n+\frac{1}{2}}\m, \overrightarrow{v\,}_{\nnm n})\}_{n\m\geq\m 0}$ for the SDE (\ref{eq:ULD}) by setting $(\overrightarrow{x\,}_{\nnm 0}\m, \overrightarrow{v\,}_{\nnm 0}) := (x_0\m, v_0)$
where $v_0\sim \mathcal{N}\big(0,u\m I_d\big)$ and for $n\geq 0$, defining $(\overrightarrow{x\,}_{\nnm n + 1}\m, \overrightarrow{x\,}_{\nnm n+\frac{1}{2}}\m, \overrightarrow{v\,}_{\nnm n + 1})$ by
\begin{align*}
\overrightarrow{x\,}_{\nnm n}^{\m(1)} & := \overrightarrow{x\,}_{\nnm n} + \bigg(\frac{1 - e^{-\frac{1}{2}\gamma h_n}}{\gamma}\bigg)\overrightarrow{v\,}_{\nnm n} - \bigg(\frac{e^{-\frac{1}{2}\gamma h_n} + \frac{1}{2}\gamma h_n - 1}{\gamma^2}\bigg)u\nabla f\big(\m\overrightarrow{x\,}_{\nnm n}\big)\\
&\hspace{11mm} + \sqrt{2\gamma u}\m\left(\m\frac{3}{2}\int_{t_n}^{t_{n+1}}\nnm\bigg(\frac{e^{-\gamma(t_{n+1}-t)} +\gamma (t_{n+1} - t) - 1}{\gamma^2 h_n}\bigg) dW_t\right.\\
&\hspace{35mm} - \left.\frac{1}{4}\int_{t_n}^{t_{n+1}}\nnm\bigg(\frac{1 - e^{-\gamma(t_{n+1} - t)}}{\gamma}\bigg)\,dW_t\m\right),\\[3pt]
\overrightarrow{x\,}_{\nnm n + \frac{1}{2}} & := \overrightarrow{x\,}_{\nnm n} + \bigg(\frac{1 - e^{-\frac{1}{2}\gamma h_n}}{\gamma}\bigg)\overrightarrow{v\,}_{\nnm n} - \bigg(\frac{e^{-\frac{1}{2}\gamma h_n} + \frac{1}{2}\gamma h_n - 1}{\gamma^2}\bigg)\bigg(\frac{2}{3}u\nabla f\big(\m\overrightarrow{x\,}_{\nnm n}\big) + \frac{1}{3}u\nabla f\big(\m\overrightarrow{x\,}_{\nnm n}^{\m(1)}\big)\bigg)\\
&\hspace{11mm} + \sqrt{2\gamma u}\int_{t_n}^{t_n+\frac{1}{2}h_n}\nnm\bigg(\frac{1 - e^{-\gamma(t_n+\frac{1}{2}h_n - t)}}{\gamma}\bigg)\,dW_t\m,\\[3pt]
\overrightarrow{x\,}_{\nnm n+1} & := \overrightarrow{x\,}_{\nnm n} + \bigg(\frac{1 - e^{-\gamma h_n}}{\gamma}\bigg)\overrightarrow{v\,}_{\nnm n} - \bigg(\frac{e^{-\gamma h_n} + \gamma h_n - 1}{\gamma^2}\bigg)\bigg(\frac{1}{3}\m u\nabla f\big(\m\overrightarrow{x\,}_{\nnm n}\big) + \frac{2}{3}\m u\nabla f\big(\m\overrightarrow{x\,}_{\nnm n}^{\m(1)}\big)\bigg)\\
&\hspace{11mm} + \sqrt{2\gamma u}\int_{t_n}^{t_{n+1}}\nnm\bigg(\frac{1 - e^{-\gamma(t_{n+1} - t)}}{\gamma}\bigg)\,dW_t\m,\\[3pt]
\overrightarrow{v\,}_{\nnm n + 1} & := e^{-\gamma h_n}\overrightarrow{v\,}_{\nnm n} - \frac{1}{6}\m e^{-\gamma h_n} u\nabla f\big(\m\overrightarrow{x\,}_{\nnm n}\big)h_n - \frac{2}{3}\m e^{-\frac{1}{2}\gamma h_n} u\nabla f\big(\m\overrightarrow{x\,}_{\nnm n}^{\m(1)}\big)h_n - \frac{1}{6}u\nabla f\big(\m\overrightarrow{x\,}_{\nnm n + 1}\big)h_n\\
&\hspace{11mm} + \sqrt{2\gamma u}\int_{t_n}^{t_{n+1}} e^{-\gamma(t_{n+1} - t)}dW_t\m.
\end{align*}
\end{definition}

Here the stochastic integrals are jointly normal and thus straightforward to generate.
Similar to the SORT method, we conjecture that this method achieves a $2$-Wasserstein error of $\varepsilon$ in $\mathcal{O}\big(\sqrt{d\m}/\varepsilon^{\frac{1}{k}}\m\big)$ steps when the first $k$ derivatives of $f$ are Lipschitz continuous ($k =1,2,3$). However, the key advantage of the above unadjsuted Langevin MCMC algorithm is that it produces $2N$ samples $\{\overrightarrow{x\,}_{\nnm\frac{1}{2}n}\}_{1\m\leq\m n\m\leq\m 2N}$ using only $2N$ evaluations of $\nabla f$.\medbreak

Another avenue for future research would be to explore if the ``shifted ODE'' approach
can be applied to related SDEs, such as high-order \cite{HighOrderLangevin} or adaptive \cite{AdaptiveLangevin} Langevin dynamics.

\begin{large}
{\textbf{Acknowledgements}}
\end{large} All the authors were supported by the DataSig programme under ESPRC grant EP/S026347/1 and by the Alan Turing Institute under EPSRC grant EP/N510129/1. The first author would like to thank Mufan (Bill) Li and Niloy Biswas for several interesting discussions on Langevin dynamics and the coupling of Markov chains.
We are also grateful to Sam Power for his valuable comments regarding ODE integrators.
\bigbreak

\Addresses\vspace{-2.5mm}

\bibliographystyle{plain}
\bibliography{references}

\newpage
\appendix

\section{Preliminary theorems}

We first recall the following standard results, which will be used within our error analysis.

\begin{theorem}\label{thm:triangle_ineq} For $p\geq 1$ and $a_1,\cdots , a_k\in\R^n$, we have 
\begin{align*}
\Bigg\|\sum_{i=1}^k a_i\m\Bigg\|_2^p \leq k^{p-1}\sum_{i=1}^k \|a_i\|_2^p\m.
\end{align*}
\end{theorem}
\begin{proof}
For any $p\geq 1$, the function $g(x) = x^p$ is convex and strictly increasing on $[0,\infty)$.
Applying the triangle inequality followed by Jensen's inequality gives
\begin{align*}
g\Bigg(\m\frac{1}{k}\m\Bigg\|\sum_{i=1}^k a_i\m\Bigg\|_2\m\Bigg) \leq g\Bigg(\frac{1}{k}\sum_{i=1}^k \|a_i\|_2\Bigg) \leq \frac{1}{k}\sum_{i=1}^k g\big(\|a_i\|_2\big).
\end{align*}
The result immediately follows from the definition of $g(x)$.
\end{proof}

\begin{theorem}\label{thm:integral_ineq} Let $z : [s,t]\rightarrow\R^n$ be such that $\|z\|_2^p$ is integrable  on $[s,t]$ for some $p\geq 1$. Then
\begin{align*}
\bigg\|\int_s^t z_r\, dr\,\bigg\|_2^p \leq (t-s)^{p-1}\int_s^t \|z_r\|_2^p\, dr.
\end{align*}
\end{theorem}
\begin{proof}
For any $p\geq 1$, the function $g(x) = x^p$ is convex and strictly increasing on $[0,\infty)$.
By viewing $\m\frac{1}{t-s}\int_s^t\,\cdot\,\,d\tau\m$ as an expectation with respect to the uniform measure on $[s,t]$,
we can apply Jensen's inequality to give
\begin{align*}
\frac{1}{t-s}\int_s^t g\big(\|z_r\|_2\big)\,dr\, \leq\, g\bigg(\frac{1}{t-s}\int_s^t \|z_r\|_2\,dr\bigg).
\end{align*}
The result immediately follows from the definition of $g(x)$.
\end{proof}

\begin{theorem}[Minkowski's inequality]\label{thm:mink_ineq} Let $X$ and $Y$ be $\R^n$-valued random variables that are $\L_p$-integrable for some $p\geq 1$. Then
\begin{align}\label{eq:mink_ineq}
\big\|X+Y\big\|_{\L_p} \leq \big\|X\big\|_{\L_p} + \big\|Y\big\|_{\L_p}.
\end{align}
\end{theorem}
\begin{proof}
Setting $z := X\cdot \1_{[0,1]} + Y\cdot \1_{[1,2]}$ and $[s,t] = [0,2]$ in Theorem \ref{thm:mink_integral_ineq} gives the result.
\end{proof}

\begin{theorem}[Minkowski's inequality for integrals]\label{thm:mink_integral_ineq} Let $z$ denote a $\R^n$-valued stochastic process which is $\L_p$-integrable over an interval $[s,t]$ for some $p\geq 1$. Then
\begin{align}\label{eq:mink_integral_ineq}
\bigg\|\int_s^t z_r\,dr\,\bigg\|_{\L_p} \leq \int_s^t \|z_r\|_{\L_p}\,dr
\end{align}
\end{theorem}
\begin{proof} By Theorem \ref{thm:integral_ineq}, it is clear that the left hand side of (\ref{eq:mink_integral_ineq}) is finite. Moreover, when $p=1$, the result directly follows by the linearity of expectation. For $p > 1$, we have
\begin{align*}
\bigg\|\int_s^t z_r\,dr\m\bigg\|_{\L_p}^p & = \E\Bigg[\m\bigg\|\int_s^t z_\tau\,d\tau\m\bigg\|_2\m\bigg\|\int_s^t z_r\,dr\m\bigg\|_2^{p-1}\m\Bigg]\\[2pt]
&\leq \E\Bigg[\m \int_s^t \|z_\tau\|_2\,d\tau\m\bigg\|\int_s^t z_r\,dr\m\bigg\|_2^{p-1}\m\Bigg]\\
& = \int_s^t \E\Bigg[\m\|z_\tau\|_2\bigg\|\int_s^t z_r\,dr\m\bigg\|_2^{p-1}\m\Bigg]\m d\tau\\
& \leq \int_s^t \E\Big[\|z_\tau\|_2^p\Big]^{\frac{1}{p}}\E\Bigg[\m \bigg\|\int_s^t z_r\,dr\m\bigg\|_2^{(p-1)\cdot\frac{p}{p-1}}\m\Bigg]^{1-\frac{1}{p}}\m d\tau\\
& = \int_s^t\|z_\tau\|_{\L_p}\,d\tau\,\bigg\|\int_s^t z_r\,dr\m\bigg\|_{\L_p}^{1-\frac{1}{p}},
\end{align*}
where the penultimate line comes from H\"{o}lder's inequality. The result now follows.
\end{proof}
\begin{remark}
We also use the Minkowski's inequalities for the $\F_n$-conditional $\L_p$ norm:
\begin{align*}
\|X\|_{\L_p^n} := \E\big[\|X\|_2^p\,\big|\m\F_n\big]^{\frac{1}{p}}.
\end{align*}
\end{remark}
\begin{theorem}\label{thm:cond_lp_estimate}
Let $Z$ be a random variable in $\R^n$ and $\F$ be a $\sigma$-algebra. Then for $p\geq 1$,
\begin{align*}
\big\|\E[Z|\m\F\,]\big\|_{\L_p} \leq \big\|Z\big\|_{\L_p}.
\end{align*}
\end{theorem}
\begin{proof}
By Jensen's inequality and the Tower law, we have
\begin{align*}
\big\|\E[Z|\m\F\,]\big\|_{\L_p} = \E\Big[\big\|\E[Z|\m\F\,]\big\|_2^p\Big]^\frac{1}{p} \leq \E\Big[\m\E\big[\|Z\|_2^p\,\big|\m\F\,\big]\Big]^\frac{1}{p} = \E\Big[\|Z\|_2^p\Big]^\frac{1}{p} = \big\|Z\big\|_{\L_p},
\end{align*}
which gives the desired result.
\end{proof}
Just as in \cite{RungeKuttaMCMC}, we shall apply the following theorem to establish a global error estimate.

\begin{theorem}\label{thm:globalestimatetrick}
Suppose $X, Y, Z\in\R^n$ are square-integrable random variables and that $Z$ is $\F$-measurable (where $\F$ is some $\sigma$-algebra). Then for constants $c,h>0$, we have
\begin{align*}
\|X+Y\|_{\L_2}^2 \m\leq\m \|X\|_{\L_2}^2 + 2\m\big\|X - \E[X|\m\F\,]\big\|_{\L_2}\|Y - Z\|_{\L_2} + \frac{c}{h}\big\|\m\E[X|\m\F\,]\big\|_{\L_2}^2 + \Big(1 + \frac{h}{c}\Big)\|Y\|_{\L_2}^2\m.
\end{align*}
\end{theorem}
\begin{proof} By the Cauchy-Schwarz and Young's inequalities, we have
\begin{align*}
\Big|\Big\langle\frac{c}{h}x, y\Big\rangle\Big| \leq \Big\|\frac{c}{h}x\Big\|_2\|y\|_2 \leq \frac{1}{2}\Big\|\frac{c}{h}x\Big\|_2^2 + \frac{1}{2}\|y\|_2^2\, ,
\end{align*}
for vectors $x,y\in\R^n$. Rearranging the above produces the below inequality:
\begin{align*}
\big|\langle x, y\rangle\big| \leq \frac{c}{2h}\|x\|_2^2 + \frac{h}{2c}\|y\|_2^2\,.
\end{align*}
Applying this to the random variables $\E[X|\m\F\,]$ and $Y$ gives\vspace{-1mm}
\begin{align*}
\big\|\E\big[\langle X, Y\m\rangle\big|\m\F\,\big]\big\|_{\L_1} & \leq \big\|\E\big[\langle X - \E[X|\m\F\,], Y\m\rangle |\m\F\,\big]\big\|_{\mathbb{L}_1} +\big\|\E\big[\langle \m\E[X|\m\F\,], Y\m\rangle\big]|\m\F\,\big]\big\|_{\mathbb{L}_1}\\[3pt]
& \leq \big\|\E\big[\langle X - \E[X|\m\F\,], Y - Z\rangle |\m\F\,\big]\big\|_{\mathbb{L}_1} +\big\|\langle \m\E[X|\m\F\,], Y\m\rangle \big\|_{\mathbb{L}_1}\\[3pt]
& \leq \big\|X - \E[X|\m\F\,]\big\|_{\mathbb{L}_2}\big\|Y - Z\big\|_{\mathbb{L}_2} + \frac{c}{2h}\|\E[X|\m\F\,]\|_{\mathbb{L}_2}^2 + \frac{h}{2c}\|Y\|_{\mathbb{L}_2}^2\m,
\end{align*}
by the Cauchy-Schwarz inequality. The second line follows by employing the Tower law
(for the first term) and Theorem \ref{thm:cond_lp_estimate} (for the second term). Expanding $\|X+Y\|_{\L_2}^2$ gives\vspace{-1mm}
\begin{align*}
\|X+Y\|_{\L_2}^2 & = \|X\|_{\L_2}^2 + 2\m\E\big[\langle X, Y \rangle\big] + \|Y\|_{\L_2}^2\\[1pt]
& = \|X\|_{\L_2}^2 + 2\m\E\big[\E\big[\langle X, Y \rangle \big|\m\F\,\big]\big] + \|Y\|_{\L_2}^2\\[1pt]
& \leq \|X\|_{\L_2}^2 + 2\m\big\|\E\big[\langle X, Y\m\rangle\big|\m\F\,\big]\big\|_{\L_1} + \|Y\|_{\L_2}^2\m,
\end{align*}
and the result follows.
\end{proof}

Along\hspace{0.25mm} with\hspace{0.25mm} Brownian\hspace{0.25mm} motion,\hspace{0.25mm} we\hspace{0.25mm} also\hspace{0.25mm} use\hspace{0.25mm} Brownian\hspace{0.25mm} bridge\hspace{0.25mm} processes\hspace{0.25mm} in\hspace{0.25mm} our\hspace{0.25mm} analysis.

\begin{definition}\label{def_bridges}
Let $W$ be a standard $d$-dimensional Brownian motion. Then for $n\geq 0$, we consider the Brownian bridge
\begin{align*}
B_t = W_t - \bigg(W_{t_n} + \bigg(\frac{t-t_n}{h_n}\bigg)\m W_{t_n, t_{n+1}}\bigg),
\end{align*}
and define the ``shifted'' Brownian bridge
\begin{align*}
\widetilde{B}_{t} := B_{t} + 12K_n\bigg(\frac{t-t_n}{h_n}\bigg),
\end{align*}
for $t\in[t_n, t_{n+1}]$ where $h_n := t_{n+1} - t_n\m$.
\end{definition}

We now give some basic $\L_p$ estimates for these processes as well as the random variable $H_n + 6K_n$ which appears in the shifted ODE (and thus appears throughout our analysis).

\begin{theorem}\label{thm:bm_scaling}
Let $W, B$ and $\widetilde{B}$ be the processes given by definition \ref{def_bridges}. Then for all $n\geq 0$ and $t\in[t_n, t_{n+1}]$, we have
\begin{align*}
\big\|\m B_{t_n, t}\big\|_{\L_p} & \leq C(p)\sqrt{d\m}\m (t-t_n)^\frac{1}{2},\\[3pt]
\big\|\m\widetilde{B}_{t_n, t}\big\|_{\L_p} & \leq C_B(p)\sqrt{d\m}\m (t-t_n)^\frac{1}{2},\\[3pt]
\big\|H_n + 6K_n\big\|_{\L_p} & = C_K(p)\sqrt{d\m}\m (h_n)^\frac{1}{2},
\end{align*}
for $p\in\{2,4,8\}$ where the constants $C_B(p)$ and $C_K(p)$ are defined as\vspace{-1mm}
\begin{align*}
C(p) & := \begin{cases}
\m 1 & \text{if}\,\,\, p = 2\\
\m\big(1 + 2d^{\m -1}\big)^\frac{1}{4} & \text{if}\,\,\, p = 4\\
\m\big(1 + 12d^{\m -1} + 44d^{\m -2} + 48d^{\m -3}\big)^\frac{1}{8} & \text{if}\,\,\, p = 8
\end{cases}\,,\\[3pt]
C_B(p) & := \left(1 + \frac{\sqrt{5}}{5}\right)\hspace{-0.5mm}C(p),\\[3pt]
C_K(p) & := \frac{\sqrt{30}}{15}C(p).
\end{align*}
(Note that $\m C(p)\sqrt{d\,}$ is the $\L_p$ norm of a standard $d$-dimensional normal random variable).
\end{theorem}
\begin{proof}
To begin, we will consider a standard $d$-dimensional normal random variable $Z$.
The $\L_2$ and $\L_4$ norms of $Z = \big(Z^{(1)}, \cdots, Z^{(d)}\big)$ are straightforward to explicitly compute.
\begin{align*}
\big\|Z\m\big\|_{\L_2}^2 & = \E\Bigg[\sum_{k\m=\m0}^d\big(Z^{(k)}\big)^2\Bigg]\\
& = d,&\\[6pt]
\big\|Z\m\big\|_{\L_4}^4 & = \E\left[\left(\,\sum_{k\m=\m0}^d\big(Z^{(k)}\big)^2\right)^2\,\right]\\
& = \sum_{\substack{0\m\leq\m i,\m j\m\leq\m d \\[1pt] i\neq j}}\E\Big[\big(Z^{(i)}\big)^2\big(Z^{(j)}\big)^2\Big] + \sum_{k\m=\m0}^d\E\Big[\big(Z^{(k)}\big)^4\Big]\\
& = d(d-1) + 3d\\[3pt]
& = d^2 + 2d.
\end{align*}

Note that we used the fact that centred normal distributions have trivial odd moments.
Similarly, one can compute the $\L_8$ norm of $Z$ (though the calculation is more involved).
\begin{align*}
\big\|Z\m\big\|_{\L_8}^8 & = \E\left[\left(\,\sum_{k\m=\m0}^d\big(Z^{(k)}\big)^2\right)^4\,\right]\\
& =  \sum_{\substack{0\m\leq\m i,\m j,\m k,\m l\m\leq\m d \\[1pt] |\{i, j, k, l\}|=4}}\E\Big[\big(Z^{(i)}\big)^2\big(Z^{(j)}\big)^2\big(Z^{(k)}\big)^2\big(Z^{(l)}\big)^2\Big] + 6\sum_{\substack{0\m\leq\m i,\m j,\m k\m\leq\m d \\[1pt] |\{i, j, k\}|=3}}\E\Big[\big(Z^{(i)}\big)^4\big(Z^{(j)}\big)^2\big(Z^{(k)}\big)^2\Big] \\
&\mmm + 3\sum_{\substack{0\m\leq\m i,\m j\m\leq\m d  \\[1pt] i\neq j}}\E\Big[\big(Z^{(i)}\big)^4\big(Z^{(j)}\big)^4\Big] + 4\sum_{\substack{0\m\leq\m i,\m j\m\leq\m d \\[1pt] i\neq j}}\E\Big[\big(Z^{(i)}\big)^6\big(Z^{(j)}\big)^2\Big] + \sum_{k\m=\m0}^d\E\Big[\big(Z^{(k)}\big)^8\Big]\\
& = d(d-1)(d-2)(d-3) + 18d(d-1)(d-2) + 27d(d-1) + 60d(d-1) + 105d\\[3pt]
& = d^4 + 12d^3 + 44d^2 + 48d.
\end{align*}

Since the Brownian bridge on $[0,1]$ is a Gaussian process with covariance function
$K_B(s,t) = \min(s,t) -st$, it is clear that a rescaled bridge process on $[t_n, t_{n+1}]$ has variance
\begin{align*}
\var\big(B_{t_n, t}\big) = \frac{(t_{n+1} - t)(t - t_n)}{h_n}\m,
\end{align*}
for $t\in[t_n, t_{n+1}]$. Hence for $p\in\{2,4,8\}$, we have
\begin{align*}
\big\|\m B_{t_n, t}\m\big\|_{\L_p}  = \bigg(\frac{(t_{n+1} - t)(t - t_n)}{h_n}\bigg)^{\frac{1}{2}}\big\|Z\m\big\|_{\L_p}\leq C(p)\sqrt{d\m}\m (t-t_n)^\frac{1}{2}.
\end{align*}
for $t\in[t_n, t_{n+1}]$. Finally, since $H_n$ and $K_n$ are independent normal random variables with
\begin{align*}
H_n\sim \mathcal{N}\bigg(0, \frac{1}{12}h_n\bigg), \mm K_n\sim \mathcal{N}\bigg(0, \frac{1}{720}h_n\bigg),
\end{align*}
we have $\|H_n + 6K_n\|_{\L_p} = \frac{\sqrt{30}}{15}(h_n)^\frac{1}{2}\|Z\m\|_{\L_p}$ for $n\geq 0$ and the result follows.
\end{proof}

\begin{theorem}\label{thm:bm_integrals}
For $n\geq 0$, the triple iterated integral of the shifted Brownian bridge on $[t_n, t_{n+1}]$ can be expressed as
\begin{align*}
\frac{1}{2}h_n^2\big(H_n + 6K_n\big) = \int_{t_n}^{t_{n+1}}\nm\int_{t_n}^t \widetilde{B}_{t_k, s}\,ds\, dt.
\end{align*}
\end{theorem}
\begin{proof}
The result follows by using integration by parts with the definitions of $H_n$ and $K_n$.
\begin{align*}
\frac{1}{2}h_n^2 H_n + 3h_n^2 K_n & = \frac{1}{2}h_n\int_{t_n}^{t_{n+1}}B_{t_n, t}\,dt + \int_{t_n}^{t_{n+1}}\bigg(\frac{1}{2}h_n - (t - t_n)\bigg)B_{t_n, t}\,dt + 2h_n^2 K_n\\[3pt]
& = \int_{t_n}^{t_{n+1}}\big(h_n - (t-t_n)\big)B_{t_n , t}\,dt + \int_{t_n}^{t_{n+1}}\nm\int_{t_n}^{t}12K_n\bigg(\frac{s-t_n}{h_n}\bigg)\,ds\,dt\\[3pt]
& = \Bigg[\,\big(h_n - (t-t_n)\big)\int_{t_n}^t B_{t_n ,s}\,ds\,\Bigg]_{t\m =\m t_n}^{t\m =\m t_{n+1}} - \int_{t_n}^{t_{n+1}}-\int_{t_n}^t B_{t_n, s}\,ds\,dt\\[3pt]
&\mmmm + \int_{t_n}^{t_{n+1}}\nm\int_{t_n}^{t}12K_n\bigg(\frac{s-t_n}{h_n}\bigg)\,ds\,dt
\end{align*}
\hspace{35.8mm}$\displaystyle = \int_{t_n}^{t_{n+1}}\nm\int_{t_n}^t \bigg(B_{t_k, s} + 12K_n\bigg(\frac{s-t_n}{h_n}\bigg)\bigg)\,ds\, dt$.
\end{proof}\bigbreak
Recall that each step of the shifted ODE approximation (over $[t_n,t_{n+1}]$) is given by
\begin{align*}
\Bigg(\begin{matrix} \,x_{n+1}^\prime\\[-3pt] \,v_{n+1}^\prime\end{matrix}\Bigg)
:=
\Bigg(\begin{matrix} \,\wx_{t_{n+1}}\\[-3pt] \,\wv_{t_{n+1}}\end{matrix}\Bigg)
+ 12 K_n\Bigg(\begin{matrix} 0\\[-3pt] \sigma\end{matrix}\,\Bigg),
\end{align*}
where $\big\{\big(\wx_t, \wv_t\big)\big\}_{t\in [t_n, t_{n+1}]}$ solves the following ODE,
\begin{align*}
\frac{d}{dt}\Bigg(\begin{matrix} \,\wx\\[-3pt] \,\wv\end{matrix}\Bigg)
=
\Bigg(\begin{matrix} \,\wv + \sigma\big(H_n + 6K_n\big)\\[-3pt] \,-\gamma\big(\m\wv + \sigma\big(H_n + 6K_n\big)\big) - u\nabla f\big(\,\wx\big)\end{matrix}\,\Bigg)
+ \frac{W_n - 12K_n}{h_n}\,\Bigg(\begin{matrix} \,0\\[-3pt] \,\sigma\end{matrix}\,\Bigg)\m,
\end{align*}
with initial condition $\big(\m\wx_{t_n}\m, \wv_{t_n}\big)\m :=\m \big(x_n\m, v_n\big)$. In addition, we will set $(\m x_{0}^\prime\m, v_{0}^\prime\m)\m :=\m (x_0\m, v_0)$. We shall now present some Taylor expansions for the above ODE and the true diffusion.

\begin{theorem}\label{thm:v_integral_expansion}
For $n\geq 0$, the integral of $\big(\wv - v\big)$ over $[t_n, t_{n+1}]$ can be expressed as
\begin{align}\label{eq:v_integral_expansion}
&\int_{t_n}^{t_{n+1}}\big(\m\wv_t - v_t\big)\,dt + \sigma h_n\big(H_n + 6K_n\big)\\
&\mm = \gamma^2 \int_{t_n}^{t_{n+1}}\nm\int_{t_n}^t\,\int_{t_n}^s \big(v_r - \wv_r\big)\,dr\,ds\,dt + u \int_{t_n}^{t_{n+1}}\nm\int_{t_n}^t\big(\m\nabla f(x_s) - \nabla f(\wx_s)\big)\,ds\,dt\nonumber\\[3pt]
&\mm\mm + \frac{1}{6}\gamma^2\sigma (h_n)^3\big(H_n + 6K_n\big) + \gamma u\int_{t_n}^{t_{n+1}}\nm\int_{t_n}^t\,\int_{t_n}^s \big(\m\nabla f(x_r) - \nabla f(\wx_r)\big)\,dr\,ds\,dt.\nonumber
\end{align}
\end{theorem}
\begin{proof}
Since $\wv$ and $v$ both satisfy differential equations, $(\wv_t - v_t)$ can be expanded as
\begin{align}\label{eq:v_expansion}
\wv_t - v_t & = \gamma\int_{t_n}^t \big(v_s - \wv_s\big)\,ds + u\int_{t_n}^t \big(\m\nabla f(x_s) - \nabla f(\wx_s)\big)\,ds\\
&\mmm - \sigma B_{t_n, t} - \gamma\sigma\big(H_n + 6K_n\big)(t-t_n) - 12\sigma K_n\bigg(\frac{t-t_n}{h_n}\bigg)\,.\nonumber
\end{align}
Hence, by integrating the above expression of $\wv_t - v_t$ over the interval $[t_n, t_{n+1}]$, we have
\begin{align*}
\int_{t_n}^{t_{n+1}}\big(\wv_t - v_t\big)\,dt & = \gamma\int_{t_n}^{t_{n+1}}\nm\int_{t_n}^t\big(v_s - \wv_s\big)\,ds\,dt + u\int_{t_n}^{t_{n+1}}\nm\int_{t_n}^t \big(\m\nabla f(x_s) - \nabla f(\wx_s)\big)\,ds\,dt\\[3pt]
&\mmm - \sigma h_n H_n - \frac{1}{2}\gamma\sigma h_n^2\big(H_n + 6K_n\big) - 6\sigma h_n K_n\m.
\end{align*}
So by Theorem \ref{thm:bm_integrals}, we can express the final term using a triple iterated integral of $B$.
\begin{align*}
&\int_{t_n}^{t_{n+1}}\big(\wv_t - v_t\big)\,dt + \sigma h_n H_n + \sigma h_n\big(H_n + 6K_n\big)\\
&\mmm = \gamma\int_{t_n}^{t_{n+1}}\nm\int_{t_n}^t\big(v_s - \wv_s\big)\,ds\,dt + u\int_{t_n}^{t_{n+1}}\nm\int_{t_n}^t \big(\m\nabla f(x_s) - \nabla f(\wx_s)\big)\,ds\,dt\\[3pt]
&\mmmmm - \gamma\sigma\int_{t_n}^{t_{n+1}}\nm\int_{t_n}^t B_{t_k, s}\,ds\, dt - 2\gamma\sigma h_n^2 K_n\\[3pt]
&\mmm = \gamma\int_{t_n}^{t_{n+1}}\nm\int_{t_n}^t\bigg(v_s - \wv_s - \sigma B_{t_n, s} - 12\sigma K_n\bigg(\frac{s-t_n}{h_n}\bigg)\bigg)\,ds\,dt\\[3pt]
&\mmmmm + u\int_{t_n}^{t_{n+1}}\nm\int_{t_n}^t \big(\m\nabla f(x_s) - \nabla f(\wx_s)\big)\,ds\,dt.
\end{align*}
The result follows by expanding the integrand of the first iterated integral using (\ref{eq:v_expansion}).
\end{proof}

Similarly, we present an expansion for the integrals involving the position components.

\begin{theorem}\label{thm:fx_integral_expansion}
For $n\geq 0$, we can write the integral of $\nabla f(\wx) - \nabla f(x)$ over $[t_n, t_{n+1}]$ as
\begin{align}\label{eq:fx_integral_expansion}
&\int_{t_n}^{t_{n+1}}\big(\m\nabla f(\wx_t) - \nabla f(x_t)\big)\,dt\\[3pt]
&\mmm = \int_{t_n}^{t_{n+1}}\nm\int_{t_n}^t \big(\m\nabla^2 f(\wx_s) - \nabla^2 f(x_s)\big)\m\wv_s\,ds\,dt\nonumber\\[3pt]
&\mmmmm + \int_{t_n}^{t_{n+1}}\nm\int_{t_n}^t \big(\m\nabla^2 f(x_s) - \nabla^2 f(x_{t_n})\big)(\m\wv_s - v_s)\,ds\,dt\nonumber\\[3pt]
&\mmmmm + \int_{t_n}^{t_{n+1}}\nm\int_{t_n}^t \nabla^2 f(x_{t_n})\Big(\m\wv_s - v_s + \sigma \widetilde{B}_{t_n, s}\Big)\,ds\,dt\nonumber\\[3pt]
&\mmmmm + \int_{t_n}^{t_{n+1}}\nm\int_{t_n}^t \sigma\big(\m\nabla^2 f(x_s) - \nabla^2 f(x_{t_n})\big)\big(H_n+6K_n\big)\,ds\,dt.\nonumber
\end{align}
\end{theorem}
\begin{proof}
By It\^{o}'s lemma and the standard chain rule, we have that for $t\in[t_n, t_{n+1}]$,
\begin{align*}
\nabla f(x_t) = \nabla f(x_{t_n}) + \int_{t_n}^t\nabla^2 f(x_s)\,dx_s
\hspace{4mm}\text{and}\hspace{5mm}
\nabla f(\wx_t) = \nabla f(x_{t_n}) + \int_{t_n}^t\nabla^2 f(\wx_s)\,d\m\wx_s\m.
\end{align*}
Since $\wx$ and $x$ satisfy differential equations, we can further expand the above integrals.
\begin{align*}
\int_{t_n}^{t_{n+1}}\big(\m\nabla f(\wx_t) - \nabla f(x_t)\big)\,dt
& = \int_{t_n}^{t_{n+1}}\Bigg(\int_{t_n}^t\nabla^2 f(\wx_s)\,d\m\wx_s - \int_{t_n}^t\nabla^2 f(x_s)\,dx_s\Bigg)\,dt\\[3pt]
& = \int_{t_n}^{t_{n+1}}\nm\int_{t_n}^t \nabla^2 f(\wx_s)\Big(\wv_s + \sigma\big(H_n+6K_n\big)\Big) - \nabla^2 f(x_s)\m v_s\,ds\,dt.
\end{align*}
By adding and subtracting integrals involving $\nabla^2 f(x_{t_n})$, we obtain the following formula:
\begin{align*}
\int_{t_n}^{t_{n+1}}\big(\m\nabla f(\wx_t) - \nabla f(x_t)\big)\,dt
& = \int_{t_n}^{t_{n+1}}\nm\int_{t_n}^t\big(\m\nabla^2 f(\wx_s) - \nabla^2 f(x_s)\big)\m\wv_s\,ds\,dt\\[3pt]
&\mmm + \int_{t_n}^{t_{n+1}}\nm\int_{t_n}^t\big(\m\nabla^2 f(x_s) - \nabla^2 f(x_{t_n})\big)\m\wv_s\,ds\,dt\\[3pt]
&\mmm + \int_{t_n}^{t_{n+1}}\nm\int_{t_n}^t \nabla^2 f(x_{t_n})\big(\wv_s - v_s\big)\,ds\,dt\\[3pt]
&\mmm + \int_{t_n}^{t_{n+1}}\nm\int_{t_n}^t\sigma\big(\m\nabla^2 f(\wx_s) - \nabla^2 f(x_{t_n})\big)\big(H_n+6K_n\big)\,ds\,dt\\[3pt]
&\mmm + \int_{t_n}^{t_{n+1}}\nm\int_{t_n}^t\sigma\m\nabla^2 f(x_{t_n})\big(H_n+6K_n\big)\,ds\,dt.
\end{align*}
Evaluating the final iterated integral and applying Theorem \ref{thm:bm_integrals} gives
\begin{align*}
\int_{t_n}^{t_{n+1}}\nm\int_{t_n}^t\sigma\m\nabla^2 f(x_{t_n})\big(H_n+6K_n\big)\,ds\,dt & = \frac{1}{2}\sigma h_n^2\m\nabla^2 f(x_{t_n})\big(H_n+6K_n\big)\\
& = \sigma\nabla^2 f(x_{t_n})\int_{t_n}^{t_{n+1}}\nm\int_{t_n}^t \widetilde{B}_{t_k, s}\,ds\, dt\\[3pt]
& = \int_{t_n}^{t_{n+1}}\nm\int_{t_n}^t\sigma\m\nabla^2 f(x_{t_n})\widetilde{B}_{t_k, s}\,ds\, dt.
\end{align*}
The result follows by plugging this into the formula for the integral of $\nabla f(\wx) - \nabla f(x)$.
\end{proof}

If we assume that the third derivative of $f$ exists, we can use the following expansion:
\begin{theorem}\label{thm:fx_expansion}
For $n\geq 0$ and $t\in[t_n, t_{n+1}]$, we can Taylor expand $\nabla f(x_t)$ and $\nabla f(\wx_t)$ as
\begin{align}
\nabla f(x_t) & = \nabla f(x_{t_n}) + \nabla^2 f(x_{t_n})\big(x_t - x_{t_n}\big) + \nabla^3 f(x_{t_n})\int_{t_n}^t\big(x_s - x_{t_n}\big)\,dx_s\label{eq:fx_expansion}\\
&\mmm + \int_{t_n}^t\int_{t_n}^s\big(\m\nabla^3 f(x_r) - \nabla^3 f(x_{t_n})\big)\,dx_r\,dx_s\m,\nonumber\\[6pt]
\nabla f(\wx_t) & = \nabla f(x_{t_n}) + \nabla^2 f(x_{t_n})\big(\m \wx_t - x_{t_n}\big) + \nabla^3 f(x_{t_n})\int_{t_n}^t\big(\m\wx_s - x_{t_n}\big)\,d\m\wx_s\label{eq:fwx_expansion}\\
&\mmm + \int_{t_n}^t\int_{t_n}^s\big(\m\nabla^3 f(\wx_r) - \nabla^3 f(x_{t_n})\big)\,d\m\wx_r\,d\m\wx_s\m.\nonumber
\end{align}
\end{theorem}
\begin{proof}
The result is a consequence of Itô’s lemma and the integration of parts formula.
\end{proof}

As one might expect, Theorems \ref{thm:fx_integral_expansion} and \ref{thm:fx_expansion} will respectively lead to useful $\L_2$ estimates after we impose Lipschitz regularity on the Hessian and third derivative of $f$.
Without these additional assumptions, we shall apply Lipschitz estimates directly to (\ref{eq:v_integral_expansion}).

\section{Dimension-free estimates for the shifted ODE}
In order to establish the convergence rates given in Table \ref{table:ode_convergence}, we shall prove a contractivity
result for the shifted ODE using the arguments given by the proof of Proposition 1 in \cite{KineticLangevinMCMC}.
Recall that we define the shifted ODE approximation $\big\{\big(\widetilde{x}_{n}\m, \widetilde{v}_n\m\big)\big\}_{n\m\geq\m 0}$ for the SDE (\ref{eq:ULD}) as
\begin{align*}
\Bigg(\begin{matrix} \,\widetilde{x}_{n+1}\\[-3pt] \,\widetilde{v}_{n+1}\end{matrix}\Bigg)
:=
\Bigg(\begin{matrix} \,\widehat{x}_{t_{n+1}}^{\m n}\\[-3pt] \,\widehat{v}_{t_{n+1}}^{\m n}\end{matrix}\Bigg)
+ 12 K_n\Bigg(\begin{matrix} \,0\\[-3pt] \,\sigma\end{matrix}\,\Bigg),
\end{align*}
where $\big\{\big(\widehat{x}_t^{\m n}, \widehat{v}_t^{\m n}\big)\big\}_{t\m\in\m [t_n, t_{n+1}]}$ solves the following ODE,
\begin{align*}
\frac{d}{dt}\Bigg(\begin{matrix} \,\widehat{x}^{\m n}\\[-3pt] \,\widehat{v}^{\m n}\end{matrix}\Bigg)
=
\Bigg(\begin{matrix} \,\widehat{v}^{\m n} + \sigma\big(H_n + 6K_n\big)\\[-3pt] \,-\gamma\big(\m\widehat{v}^{\m n} + \sigma\big(H_n + 6K_n\big)\big) - u\nabla f\big(\,\widehat{x}^{\m n}\big)\end{matrix}\,\Bigg)
+ \frac{W_n - 12K_n}{h_n}\,\Bigg(\begin{matrix} \,0\\[-3pt] \,\sigma\end{matrix}\,\Bigg)\m,
\end{align*}
with initial condition $\big(\m\widehat{x}_{t_n}^{\m n}, \widehat{v}_{t_n}^{\m n}\big) := \big(\widetilde{x}_{n}\m, \widetilde{v}_n\big)$.
\begin{theorem}[Exponential contractivity of the shifted ODE]\label{thm:exp_contract}
For each $n\geq 0$, we define
\begin{align*}
y_n & := \begin{pmatrix}\big(\lambda\m\widetilde{x}_n + \widetilde{v}_n\big) - \big(\lambda\m x_{t_n} + v_{t_n}\big)\\
\big(\eta\m\widetilde{x}_n + \widetilde{v}_n\big) - \big(\eta\m x_{t_n} + v_{t_n}\big)\end{pmatrix},\mm
\widetilde{y}_n := \begin{pmatrix}\big(\lambda\m\widetilde{x}_n + \widetilde{v}_n\big) - \big(\lambda\m x_n^\prime + v_n^\prime\big)\\
\big(\eta\m\widetilde{x}_n + \widetilde{v}_n\big) - \big(\eta\m x_n^\prime + v_n^\prime\big)\end{pmatrix},
\end{align*}
where $\lambda\in[0,\frac{1}{2}\gamma)$ and $\eta := \gamma - \lambda$. Suppose $f$ is $m$-strongly convex and $\nabla f$ is $M$-Lipschitz. Then
\begin{align}\label{eq:exp_contract}
\big\|\widetilde{y}_{n+1}\big\|_{\L_2^n}\leq e^{-\alpha h_n}\big\|y_n\big\|_2\m,
\end{align}
where the constant $\alpha$ is given by
\begin{align*}
\alpha & := \frac{(\eta^2 - uM)\vee (um - \lambda^2)}{\gamma - 2\lambda}\m.
\end{align*}
\end{theorem}
\begin{proof}
We will follow the same arguments that were used to establish Proposition 1 in \cite{KineticLangevinMCMC}.
We first consider the following random variables:
\begin{align*}
w_t^{n} := \big(\lambda\m\widehat{x}_t^{\m n} + \widehat{v}_t^{\m n}\big) - \big(\lambda\m\wx_t + \wv_t\big),\\
z_t^{n} := \big(\eta\m\widehat{x}_t^{\m n} + \widehat{v}_t^{\m n}\big) - \big(\eta\m\wx_t + \wv_t\big).
\end{align*}
Note that $f$ is twice differentiable, so we can apply Taylor's theorem to give
\begin{align*}
\nabla f(\wx_t) - \nabla f(\widehat{x}_t^{\m n}) = \underbrace{\int_{0}^{1}\nabla^2 f\big(\m\widehat{x}_t^{\m n} + r(\m\widehat{x}_t^{\m n} - \wx_t)\big)\,dr}_{=:\,\h_t^n}\big(\m\wx_t - \widehat{x}_t^{\m n}\big).
\end{align*}
Using this formula with the ODEs for $(\m\widehat{x}^{\m n}, \widehat{v}^{\m n})$ and $(\m\wx, \wv)$, the derivatives of $w^n$ and $z^n$ become
\begin{align*}
\frac{dw^n}{dt} & = (\gamma - \lambda)\m\big(\m \wv_t - \widehat{v}_t^{\m n}\big) + u\m\big(\m \nabla f(\wx_t) - \nabla f(\widehat{x}_t^{\m n})\big)\\[3pt]
& = \eta\,\frac{\lambda z_t^n - \eta w_t^n}{\eta - \lambda} + u\h_t\,\frac{w_t^n - z_t^n}{\eta - \lambda}\\[3pt]
& = \frac{\big(u\h_t^n-\eta^2 I_d\big)w_t^n + \big(\eta\lambda I_d- u\h_t\big)z_t^n}{\eta - \lambda}\,,
\end{align*}
and
\begin{align*}
\frac{dz^n}{dt} & = (\gamma - \eta)\m\big(\m \wv_t - \widehat{v}_t^{\m n}\big) + u\m\big(\m \nabla f(\wx_t) - \nabla f(\widehat{x}_t^{\m n})\big)\\[3pt]
& = \lambda\,\frac{\lambda z_t^n - \eta w_t^n}{\eta - \lambda} + u\h_t^n\,\frac{w_t^n - z_t^n}{\eta - \lambda}\\[3pt]
& = \frac{\big(u\h_t^n - \eta\lambda I_d\big)w_t^n - \big(u\h_t^n - \lambda^2 I_d\big)z_t^n}{\eta - \lambda}\,,
\end{align*}
where $I_d$ is the $d\times d$ identity matrix. Note that we have used the fact that $\eta + \lambda = \gamma$.
Using these formulae along with the chain rule, we have
\begin{align*}
\frac{d}{dt}\Big(\|w_t^n\|_2^2 + \|z_t^n\|_2^2\Big) & = 2\m(w_t^n)^T\m\frac{dw_t^n}{dt} + 2\m(z_t^n)^T\m\frac{dz_t^n}{dt}\\[3pt]
& = \frac{2}{\eta - \lambda}\Big((w_t^n)^T\big(u\h_t^n - \eta^2 I_d\big)(w_t^n) - (z_t^n)^T\big(u\h_t^n - \lambda^2 I_d\big)(z_t^n)\Big).
\end{align*}
Since $f$ is $m$-strongly convex and its gradient $\nabla f$ is $M$-Lipschitz continuous, we have $m I_d\preccurlyeq\nabla^2 f(x) \preccurlyeq M I_d$
for $x\in\R^d$. In particular, this implies $m I_d\preccurlyeq \h_t \preccurlyeq M I_d$ and hence
\begin{align*}
\frac{d}{dt}\Big(\|w_t^n\|_2^2 + \|z_t^n\|_2^2\Big) & \leq \frac{2}{\eta - \lambda}\Big(\big(uM - \eta^2\big)\|w_t^n\|_2^2 + \big(\lambda^2 - um\big)\|z_t^n\|_2^2\Big)\\[3pt]
& \leq \frac{2\big((uM - \eta^2)\vee (\lambda^2 - um)\big)}{\eta - \lambda}\Big(\|w_t^n\|_2^2 + \|z_t^n\|_2^2\Big).
\end{align*}
Therefore by Gr\"{o}nwall's inequality (Corollary 3 in \cite{Gronwall}), we have that
\begin{align*}
\|w_{t_{n+1}}^n\|_2^2 + \|z_{t_{n+1}}^n\|_2^2 \leq \exp\hspace{-0.25mm}\left(-\frac{2\big((\eta^2 - uM)\vee (um - \lambda^2)\big)}{\eta - \lambda}\,h_n\right)\hspace{-0.25mm}\Big(\|w_{t_n}^n\|_2^2 + \|z_{t_n}^n\|_2^2\Big).
\end{align*}
The result now follows as $\m\widetilde{x}_{n+1} - x_{n+1}^\prime = \widehat{x}_{t_{n+1}}^{\m n} - \wx_{t_{n+1}}$ and $\m\widetilde{v}_{n+1} - v_{n+1}^\prime = \widehat{v}_{t_{n+1}}^{\m n} - \wv_{t_{n+1}}$.
\end{proof}
\begin{remark}
Any contraction theorem for ULD whose proof uses a synchronous coupling of Brownian motions is likely to hold for the shifted ODE as its driven by a linear path.
\end{remark}

\begin{remark}
Provided that $\gamma^2 > uM$, the contraction rate $\alpha$ becomes positive if $\lambda = 0$.
\end{remark}

In addition to the above theorem, we will prove some further dimension-free estimates.
\begin{theorem}\label{thm:error_flow}
Suppose the gradient $\nabla f$ is $M$-Lipschitz. Then for each $n\geq 0$, we have
\begin{align}\label{eq:error_flow}
\big\|\m\widetilde{y}_{n+1} - y_n\big\|_{\L_2^n} \leq C_1 h_n\big\|y_n\big\|_2\m,
\end{align}
where the constant $C_1$ is given by
\begin{align*}
C_1 := \frac{\sqrt{2}\m\eta}{\eta - \lambda}\exp\bigg(\gamma h_n + \frac{1}{2}uM h_n^2\bigg)\bigg(\gamma + uM h_n\bigg).
\end{align*}
\end{theorem}
\begin{proof}
As $(\m\widehat{x}^{\m n}, \widehat{v}^{\m n})$ and $(\m\wx, \wv)$ satisfy ODEs, the $M$-Lipschitz regularity of $\nabla f$ gives
\begin{align*}
\int_{t_n}^{t}\big\|\m\widehat{v}_s^{\m n} - \wv_s\big\|_{\L_2^n}\,ds & \leq (t-t_n)\big\|\m\widetilde{v}_n - v_{t_n}\big\|_{2} + \gamma\int_{t_n}^{t}\int_{t_n}^{s}\big\|\m\widehat{v}_r^{\m n} - \wv_r\big\|_{\L_2^n}\,dr\,ds\\
&\mmm + u\int_{t_n}^{t}\int_{t_n}^{s}\big\|\m\nabla f(\widehat{x}_r^{\m n}) - \nabla f(\wx_r)\big\|_{\L_2^n}\,dr\,ds\\[3pt]
& \leq (t-t_n)\big\|\m\widetilde{v}_n - v_{t_n}\big\|_{2} + \gamma\int_{t_n}^{t}\int_{t_n}^{s}\big\|\m\widehat{v}_r^{\m n} - \wv_r\big\|_{\L_2^n}\,dr\,ds\\
&\mmm + uM\int_{t_n}^{t}\int_{t_n}^{s}\int_{t_n}^{r}\big\|\m\widehat{v}_w^{\m n} - \wv_w\big\|_{\L_2^n}\,dw\,dr\,ds\\[3pt]
& \leq (t-t_n)\big\|\m\widetilde{v}_n - v_{t_n}\big\|_{2}\\
&\mmm + \int_{t_n}^{t}\big(\gamma + uM(s-t_k)\big)\int_{t_n}^{s}\big\|\m\widehat{v}_r^{\m n} - \wv_r\big\|_{\L_2^n}\,dr\,ds\m.
\end{align*}
It then follows from Gronwall's inequality (Corollary 2 in \cite{Gronwall}) that
\begin{align*}
\int_{t_n}^t\big\|\m\widehat{v}_s^{\m n} - \wv_s\big\|_{\L_2^n}\,ds \leq \exp\bigg(\gamma h_n + \frac{1}{2}uM h_n^2\bigg)\m (t-t_n)\m\big\|\m\widetilde{v}_n - v_{t_n}\big\|_{2}\m.
\end{align*}
The first component of $\widetilde{y}_{n+1} - y_n$ can be estimated as
\begin{align*}
&\Big\|\Big(\big(\lambda\m\widetilde{x}_{n+1} + \widetilde{v}_{n+1}\big) - \big(\lambda\m x_{n+1}^\prime + v_{n+1}^\prime\big)\Big) - \Big(\big(\lambda\m\widetilde{x}_n + \widetilde{v}_n\big) - \big(\lambda\m x_{t_n} + v_{t_n}\big)\Big)\Big\|_{\L_2^n}\\
&\mm = \Big\|\Big(\big(\lambda\big(\widehat{x}_{t_{n+1}}^{\m n} - \widetilde{x}_n\big) + \big(\widehat{v}_{t_{n+1}}^{\m n} - \widetilde{v}_n\big)\big) - \big(\lambda\big(\wx_{t_{n+1}}-x_{t_n}\big) + \big(\wv_{t_{n+1}} - v_{t_n}\big)\big)\Big)\Big\|_{\L_2^n}\\
&\mm = \bigg\|\big(\lambda-\gamma\big)\int_{t_n}^{t_{n+1}} \big(\m\widehat{v}_t^{\m n} - \wv_t\big)\,dt  - u\int_{t_n}^{t_{n+1}} \big(\m\nabla f(\widehat{x}_t^{\m n}) - \nabla f(\wx_t)\big)\,dt\,\bigg\|_{\L_2^n}\\
&\mm \leq \big(\gamma - \lambda\big)\int_{t_n}^{t_{n+1}}\big\|\m\widehat{v}_t^{\m n} - \wv_t\big\|_{\L_2^n}\,dt + uM\int_{t_n}^{t_{n+1}}\int_{t_n}^{t}\big\|\m\widehat{v}_s^{\m n} - \wv_s\big\|_{\L_2^n}\,ds\,dt\\
&\mm \leq \exp\bigg(\gamma h_n + \frac{1}{2}uM h_n^2\bigg)\bigg(\big(\gamma - \lambda\big)\m h_n + \frac{1}{2}uM h_n^2\bigg)\big\|\m\widetilde{v}_n - v_{t_n}\big\|_{2}\m.
\end{align*}
The above holds for the second component of $\widetilde{y}_{n+1} - y_n$ (simply by replacing $\lambda$ with $\eta$).
So by the triangle inequality and the fact that $\lambda + \eta = \gamma$, we have the following inequality
\begin{align*}
\big\|\m\widetilde{y}_{n+1} - y_n\big\|_{\L_2^n} \leq \exp\bigg(\gamma h_n + \frac{1}{2}uM h_n^2\bigg)\bigg(\gamma + uM h_n\bigg) h_n \big\|\m\widetilde{v}_n - v_{t_n}\big\|_2\m.
\end{align*}
The result immediately follows as
\begin{align*}
\big\|\m\widetilde{v}_n - v_{t_n}\big\|_2
& = \frac{1}{\eta - \lambda}\bigg\|\m\eta\Big(\big(\lambda\m\widetilde{x}_n + \widetilde{v}_n\big) - \big(\lambda\m x_{t_n} + v_{t_n}\big)\Big) - \lambda\Big(\big(\eta\m\widetilde{x}_n + \widetilde{v}_n\big) - \big(\eta\m x_{t_n} + v_{t_n}\big)\Big)\bigg\|_2\\[3pt]
& \leq \frac{\eta}{\eta - \lambda}\Big(\big\|\big(\lambda\m\widetilde{x}_n + \widetilde{v}_n\big) - \big(\lambda\m x_{t_n} + v_{t_n}\big)\big\|_2 + \big\|\big(\eta\m\widetilde{x}_n + \widetilde{v}_n\big) - \big(\eta\m x_{t_n} + v_{t_n}\big)\big\|_2\Big)\\[3pt]
& \leq \frac{\sqrt{2}\m\eta}{\eta - \lambda}\Big(\big\|\big(\lambda\m\widetilde{x}_n + \widetilde{v}_n\big) - \big(\lambda\m x_{t_n} + v_{t_n}\big)\big\|_2^2 + \big\|\big(\eta\m\widetilde{x}_n + \widetilde{v}_n\big) - \big(\eta\m x_{t_n} + v_{t_n}\big)\big\|_2^2\Big)^\frac{1}{2},
\end{align*}
where the final line is a direct consequence of Theorem \ref{thm:triangle_ineq}.
\end{proof}

\begin{theorem}\label{thm:error_propagation}
Suppose that $f$ is $m$-strongly convex and its gradient $\nabla f$ is $M$-Lipschitz. For $n\geq 0$, we define
\begin{align*}
y_n^\prime & := \begin{pmatrix}\big(\lambda\m x_n^\prime + v_n^\prime\big) - \big(\lambda\m x_{t_n} + v_{t_n}\big)\\
\big(\eta\m x_n^\prime + v_n^\prime\big) - \big(\eta\m x_{t_n} + v_{t_n}\big)\end{pmatrix},
\end{align*}
where $\lambda\in[0,\frac{1}{2}\gamma)$ and $\eta := \gamma - \lambda$ are the constants that were used in Theorem \ref{thm:exp_contract}. Then
\begin{align}\label{eq:error_propagation}
\big\|y_{n+1}\big\|_{\L_2}^2 & \leq \big\|y_{n+1}^\prime\big\|_{\L_2}^2 + 2\m C_1 h_n\big\|y_{n+1}^\prime\big\|_{{\L_2}}\big\|y_n\big\|_{\L_2}\\[3pt]
&\mmm + \frac{c}{h_n}\big\|\m\E_n\big[y_{n+1}^\prime\big]\big\|_{\L_2}^2 + \Big(1 + \frac{1}{c}h_n\Big)e^{-2\alpha h_n}\big\|y_n\big\|_{\L_2}^2,\nonumber
\end{align}
for any $c>0$.
\end{theorem}
\begin{proof} By setting $X = y_{n+1}^\prime\m, Y = \widetilde{y}_{n+1}\m, Z = y_{n}$ and $\F = \F_{t_n}$ in Theorem \ref{thm:globalestimatetrick}, we have
\begin{align*}
\big\|y_{n+1}^\prime + \widetilde{y}_{n+1}\big\|_{\L_2}^2 & \leq \big\|y_{n+1}^\prime\big\|_{\L_2}^2 + 2\m\big\|y_{n+1}^\prime - \E_n\big[y_{n+1}^\prime\big]\big\|_{\L_2}\big\|\widetilde{y}_{n+1} - y_n\big\|_{\L_2}\\[3pt]
&\mmm + \frac{c}{h_n}\big\|\E_n\big[y_{n+1}^\prime\big]\big\|_{\L_2}^2 + \Big(1 + \frac{1}{c}h_n\Big)\big\|\widetilde{y}_{n+1}\big\|_{\L_2}^2.
\end{align*}
Note that $y_{n+1} = y_{n+1}^\prime + \widetilde{y}_{n+1}\m$, and by the Tower law we can estimate the second term as
\begin{align*}
\big\|y_{n+1}^\prime - \E_n\big[y_{n+1}^\prime\big]\big\|_{\L_2}^2 
& = \big\|y_{n+1}^\prime\big\|_{\L_2}^2 + \big\|\E_n\big[y_{n+1}^\prime\big]\big\|_{\L_2}^2 - 2\m\E\Big[\big\langle y_{n+1}^\prime, \E_n\big[y_{n+1}^\prime\big]\big\rangle\Big]\\[3pt]
& = \big\|y_{n+1}^\prime\big\|_{\L_2}^2 + \big\|\E_n\big[y_{n+1}^\prime\big]\big\|_{\L_2}^2 - 2\m\E\Big[\E_n\big[\big\langle y_{n+1}^\prime, \E_n\big[y_{n+1}^\prime\big]\big\rangle\big]\Big]\\[3pt]
& = \big\|y_{n+1}^\prime\big\|_{\L_2}^2 - \big\|\E_n\big[y_{n+1}^\prime\big]\big\|_{\L_2}^2\\[3pt]
& \leq \big\|y_{n+1}^\prime\big\|_{\L_2}^2.
\end{align*}
The result now immediately follows by Theorems \ref{thm:exp_contract} and \ref{thm:error_flow}.
\end{proof}

\section{Estimates for shifted ODE under standard assumptions}\label{appen:standard}

To begin this section, we will establish some local $\L_p$ error estimates for the shifted ODE that only require the gradient $\nabla f$ to be $M$-Lipschitz. Our bounds grow linearly with $\sqrt{d\m}$. 

\begin{theorem}\label{thm:inside_estimate} Suppose that $\nabla f$ is $M$-Lipschitz. Then for all $n\geq 0$ and $t\in[t_n, t_{n+1}]$,
\begin{align}
\int_{t_n}^t \big\|v_s - \wv_s\big\|_{\L_p^n}\,ds & \leq C_2(p)\sqrt{d\m}\m (t-t_n)^\frac{3}{2},\label{eq:inside_estimate_intv}\\
\big\|x_t - \wx_t\big\|_{\L_p^n} & \leq C_3(p)\sqrt{d\m}\m (h_n)^\frac{1}{2} (t-t_n),\label{eq:inside_estimate_x}\\[3pt]
\big\|v_t - \wv_t\big\|_{\L_p^n} & \leq C_4(p)\sqrt{d\m}\m (t-t_n)^\frac{1}{2},\label{eq:inside_estimate_v}
\end{align}
for $p\in\{2,4,8\}$ where the constants $C_2(p), C_3(p)$ and $C_4(p)$ are given by
\begin{align*}
C_2(p) & := \sigma\bigg(\m\frac{2}{3}\m C_B(p) + \frac{1}{2}\gamma\m C_K(p) h_n + \frac{1}{6}uM C_K(p)h_n^2\bigg)\exp\bigg(\gamma h_n + \frac{1}{2}uM h_n^2\bigg),\\
C_3(p) & := C_2(p) + \sigma\m C_K(p),\\[3pt]
C_4(p) & := C_B(p) + C_3(p)\Big(\gamma + \frac{1}{2}uM h_n\Big)h_n.
\end{align*}
\end{theorem}
\begin{proof}
By Theorems \ref{thm:mink_ineq} and \ref{thm:mink_integral_ineq} along with the $M$-Lipschitz continuity of $\nabla f$, we have
\begin{align*}
&\int_{t_n}^t \big\|v_s - \wv_s\big\|_{\L_p^n}\,ds\\
&\mm = \int_{t_n}^t \Bigg\|\int_{t_n}^s \gamma\m\Big(\wv_r - v_r + \sigma\big(H_n + 6K_n\big)\Big)\,dr + \int_{t_n}^s u\big(\m\nabla f(\wx_r) - \nabla f(x_r)\big)\,dr + \sigma \widetilde{B}_{t_n,s}\m\Bigg\|_{\L_p^n}ds\\[3pt]
&\mm \leq \gamma\int_{t_n}^t\int_{t_n}^s \big\|v_r - \wv_r\big\|_{\L_p^n}\,dr\,ds + uM\int_{t_n}^t\int_{t_n}^s \big\|\m x_r - \wx_r\big\|_{\L_p^n}\,dr\,ds + \int_{t_n}^t \big\|\m\widetilde{B}_{t_n, s}\big\|_{\L_p^n}\, ds\\
&\mmmm + \int_{t_n}^t \big\|\m\gamma\m\sigma\big(H_n + 6K_n\big)(s - t_n)\big\|_{\L_p^n}\,ds\\[3pt]
&\mm \leq \gamma\int_{t_n}^t\int_{t_n}^s \big\|v_r - \wv_r\big\|_{\L_p^n}\,dr\,ds + uM\int_{t_n}^t\int_{t_n}^s\int_{t_n}^r \big\|\m v_w - \wv_w + \sigma\big(H_n + 6K_n\big)\big\|_{\L_p^n}\,dw\,dr\,ds\\
&\mmmm + \frac{2}{3}\sigma\m C_B(p)\sqrt{d\m}\m (t-t_n)^\frac{3}{2} + \frac{1}{2}\gamma\m\sigma C_K(p)\sqrt{d\m}\m (h_n)^\frac{1}{2}(t-t_n)^2\\[3pt]
&\mm \leq \int_{t_n}^t\Big(\gamma + uM(s-t_n)\Big)\int_{t_n}^s\big\|v_r - \wv_r\big\|_{\L_p^n}\,dr\,ds\\
&\mmmm + \bigg(\m\frac{2}{3}\sigma\m C_B(p) + \frac{1}{2}\gamma\m\sigma\m C_K(p)h_n + \frac{1}{6}\sigma uM C_K(p)h_n^2\bigg)\sqrt{d\m}\m (t-t_n)^{\frac{3}{2}}.
\end{align*}
Therefore by Gronwall's inequality (Corollary 2 in \cite{Gronwall}), we have
\begin{align*}
\int_{t_n}^t \big\|v_s - \wv_s\big\|_{\L_p^n}\,ds & \leq \sigma\bigg(\m\frac{2}{3}\m C_B(p) + \frac{1}{2}\gamma\m C_K(p) h_n + \frac{1}{6}uM C_K(p)h_n^2\bigg)\exp\bigg(\gamma h_n + \frac{1}{2}uM h_n^2\bigg)\\
&\mmm\cdot\sqrt{d\m}\m (t-t_n)^\frac{3}{2}.
\end{align*}
Using the above inequality, it is now straightforward to estimate the difference $(x_t - \wx_t)$.
\begin{align*}
\big\|x_t - \wx_t\big\|_{\L_p^n} & \leq \int_{t_n}^t \big\|v_s - \wv_s - \sigma\big(H_n + 6K_n\big)\big\|_{\L_p^n}\,ds\\[3pt]
& \leq \int_{t_n}^t \big\|v_s - \wv_s\big\|_{\L_p^n}\,ds + \sigma\int_{t_n}^t \big\|H_n + 6K_n\big\|_{\L_p^n}\,ds\\[3pt]
& \leq \Big(C_2(p) + \sigma\m C_K(p)\Big)\sqrt{d\m}\m (h_n)^\frac{1}{2} (t-t_n).
\end{align*}
Finally, we can derive a similar estimate for the velocity components as
\begin{align*}
&\big\|v_t - \wv_t\big\|_{\L_p^n}\\
&\mm = \Bigg\|\int_{t_n}^t \gamma\m\Big(\wv_s - v_s + \sigma\big(H_n + 6K_n\big)\Big)\,ds + \int_{t_n}^t u\big(\m\nabla f(\wx_s) - \nabla f(x_s)\big)\,ds + \sigma \widetilde{B}_{t_n,t}\m\Bigg\|_{\L_p^n}\\
&\mm \leq \gamma\int_{t_n}^t \big\|\wv_s - v_s + \big(H_n + 6K_n\big)\big\|_{\L_p^n}\,ds + uM\int_{t_n}^t\big\|\m\wx_s - x_s\big\|_{\L_p^n}\,ds + \sigma\big\|\widetilde{B}_{t_n, t}\big\|_{\L_p^n}
\end{align*}
$\displaystyle \hspace{10.75mm}\leq \gamma\m C_3(p)\sqrt{d\m}\m (h_n)^\frac{1}{2} (t-t_n) + \frac{1}{2}uMC_3(p)\sqrt{d\m}\m (h_n)^\frac{1}{2} (t-t_n)^2 + C_B(p)\sqrt{d\m}\m (t-t_n)^\frac{1}{2}.$
\end{proof}
\begin{remark}
By applying the same arguments, we see that Theorem \ref{thm:inside_estimate} still holds when we use the (unconditional) $\L_p$ norm as the proof only required Minkowski's inequalities.
\end{remark}
\begin{remark}
A particularly nice feature of these error estimates is that they are uniform.
That is, the estimates (\ref{eq:inside_estimate_intv}), (\ref{eq:inside_estimate_x}) and (\ref{eq:inside_estimate_v}) are independent of the initial value $\big(x_{t_n}, v_{t_n}\big)$.
\end{remark}

\begin{theorem}\label{thm:inside_mean_estimate} Suppose that $\nabla f$ is $M$-Lipschitz. Then for all $n\geq 0$ and $t\in[t_n, t_{n+1}]$,
\begin{align}
\big\|\E_n\big[v_t - \wv_t\big]\big\|_{\L_p} & \leq C_5(p)\sqrt{d\m}\m (h_n)^\frac{1}{2} (t-t_n)^2,\label{eq:inside_mean_estimate_v}\\[3pt]
\big\|\E_n\big[x_t - \wx_t\big]\big\|_{\L_p} & \leq \frac{1}{3}\m C_5(p)\sqrt{d\m}\m (h_n)^\frac{1}{2} (t-t_n)^3,\label{eq:inside_mean_estimate_x}
\end{align}
for $p\in\{2,4,8\}$ where the constant $C_5(p)$ is defined as
\begin{align*}
C_5(p) := uM\m C_3(p)\exp\big(\gamma h_n\big).
\end{align*}
\end{theorem}
\begin{proof}
Since $W$ is a centered Gaussian process and $W_n, H_n, K_n$ all have zero expectation,
\begin{align*}
\big\|\E_n\big[\m\wv_t - v_t\big]\big\|_{\L_p}
& = \bigg\|\gamma\int_{t_n}^t\E_n\big[\m v_s - \wv_s\big]\,ds + u\int_{t_n}^t\E_n\big[\m \nabla f(x_s) - \nabla f(\wx_s)\big]\,ds\bigg\|_{\L_p}\\[3pt]
& \leq \gamma\int_{t_n}^t\big\|\E_n\big[\m\wv_s - v_s\big]\big\|_{\L_p}\,ds + u\int_{t_n}^t\big\|\E_n\big[\m \nabla f(x_s) - \nabla f(\wx_s)\big]\big\|_{\L_p}\,ds\\[3pt]
& \leq \gamma\int_{t_n}^t\big\|\E_n\big[\m\wv_s - v_s\big]\big\|_{\L_p}\,ds + u\int_{t_n}^t\big\|\m \nabla f(x_s) - \nabla f(\wx_s)\big\|_{\L_p}\,ds\\[3pt]
& \leq \gamma\int_{t_n}^t\big\|\E_n\big[\m\wv_s - v_s\big]\big\|_{\L_p}\,ds + uM\int_{t_n}^t\big\|\m x_s - \wx_s\big\|_{\L_p}\,ds\\[3pt]
& \leq \gamma\int_{t_n}^t\big\|\E_n\big[\m\wv_s - v_s\big]\big\|_{\L_p}\,ds + uM\m C_3(p)\sqrt{d\m}\m (h_n)^\frac{1}{2}(t-t_n)^2.
\end{align*}
By Gronwall's inequality (Corollary 2 in \cite{Gronwall}), we have
\begin{align*}
\big\|\E_n\big[\m\wv_t - v_t\big]\big\|_{\L_p} \leq uM\m C_3(p)\exp\big(\gamma h_n\big)\sqrt{d\m}\m (h_n)^\frac{1}{2}(t-t_n)^2.
\end{align*}
Using the above, we can estimate the position component as
\begin{align*}
\big\|\E_n\big[\m\wx_t - x_t\big]\big\|_{\L_p} & \leq \int_{t_n}^t\big\|\E_n\big[\m\wv_s - v_s + \sigma\big(H_n + 6K_n\big)\big]\big\|_{\L_p}\,ds\\[3pt]
& = \int_{t_n}^t\big\|\E_n\big[\m\wv_s - v_s\big]\big\|_{\L_p}\,ds\\[3pt]
& \leq \frac{1}{3}\m C_5(p)\sqrt{d\m}\m (h_n)^\frac{1}{2} (t-t_n)^3.
\end{align*}
\end{proof}

Due to the exponential contracitivty of the shifted ODE detailed by Theorem \ref{thm:exp_contract},
we can use Theorem \ref{thm:inside_estimate} to estimate the process $\big\{\big(\lambda\m x_{t_{n+1}} + v_{t_{n+1}}\big) - \big(\lambda\m x_{t_{n+1}}^\prime + v_{t_{n+1}}^\prime\big)\big\}_{n \geq 0}\m$.

\begin{theorem}\label{thm:first_local_estimate}
Suppose that $\nabla f$ is $M$-Lipschitz. Then for $n\geq 0$ and $\lambda\in [0,\gamma]$, we have
\begin{align}
\Big\|\big(\lambda\m x_{t_{n+1}} + v_{t_{n+1}}\big) - \big(\lambda\m x_{t_{n+1}}^\prime + v_{t_{n+1}}^\prime\big)\Big\|_{\L_p^n} \leq C_6(p, \lambda)\m\sqrt{d\m}\m (h_n)^\frac{5}{2}\label{eq:first_local_estimate}
\end{align}
for $p\in\{2,4,8\}$ where the constant $C_6(p, \lambda)$ is given by
\begin{align*}
C_6(p, \lambda) & := \frac{1}{2}uM C_3(p) + \big(\gamma - \lambda \big)\bigg(\bigg(\m\frac{4}{35}\m\gamma^2 +  \frac{1}{6}uM\bigg)C_2(p) +  \frac{1}{6}\m\gamma^2\sigma\m C_K(p)\bigg)h_n\m\\[3pt]
&\mmm + \frac{1}{24}\gamma u M\big(\gamma - \lambda\big) C_3(p)\m h_n^2.
\end{align*}
\end{theorem}
\begin{proof}
By Theorem \ref{thm:v_integral_expansion} and the definition of the shifted ODE, we have the expansion:
\begin{align}\label{eq:key_expansion}
&\big(\lambda\m x_{t_{n+1}} + v_{t_{n+1}}\big) - \big(\lambda\m x_{t_{n+1}}^\prime + v_{t_{n+1}}^\prime\big)\\[3pt]
&\mm = \big(\gamma - \lambda \big)\Bigg(\int_{t_n}^{t_{n+1}}\big(\m\wv_t - v_t\big)\,dt + \sigma h_n \big(H_n + 6K_n\big)\Bigg) + u\int_{t_n}^{t_{n+1}}\big(\m\nabla f(\wx_t) - \nabla f(x_t)\big)\,dt\nonumber\\[3pt]
&\mm = \big(\gamma - \lambda \big)\Bigg(\gamma^2 \int_{t_n}^{t_{n+1}}\nm\int_{t_n}^t\,\int_{t_n}^s \big(v_r - \wv_r\big)\,dr\,ds\,dt + u \int_{t_n}^{t_{n+1}}\nm\int_{t_n}^t\big(\m\nabla f(x_s) - \nabla f(\wx_s)\big)\,ds\,dt\nonumber\\
&\hspace{26.5mm} + \frac{1}{6}\gamma^2\sigma (h_n)^3\big(H_n + 6K_n\big) + \gamma u\int_{t_n}^{t_{n+1}}\nm\int_{t_n}^t\,\int_{t_n}^s \big(\m\nabla f(x_r) - \nabla f(\wx_r)\big)\,dr\,ds\,dt\Bigg)\nonumber\\
&\mmmm + u\int_{t_n}^{t_{n+1}}\big(\m\nabla f(\wx_t) - \nabla f(x_t)\big)\,dt.\nonumber
\end{align}
Using the $M$-Lipschitz regularity of $\nabla f$ and Theorem \ref{thm:mink_integral_ineq}, we can estimate these integrals.
\begin{align*}
&\Big\|\big(\lambda\m x_{t_{n+1}} + v_{t_{n+1}}\big) - \big(\lambda\m x_{t_{n+1}}^\prime + v_{t_{n+1}}^\prime\big)\Big\|_{\L_p^n}\\[3pt]
&\mm \leq uM\int_{t_n}^{t_{n+1}}\big\|\m\wx_t - x_t\big\|_{\L_p^n}\,dt + \gamma^2\big(\gamma - \lambda \big)\int_{t_n}^{t_{n+1}}\nm\int_{t_n}^t\,\int_{t_n}^s \big\|v_r - \wv_r\big\|_{\L_p^n}\,dr\,ds\,dt\\[3pt]
&\mmmm + uM\big(\gamma - \lambda \big)\int_{t_n}^{t_{n+1}}\nm\int_{t_n}^t\big\|x_s - \wx_s\big\|_{\L_p^n}\,ds\,dt + \frac{1}{6}\gamma^2\sigma\big(\gamma - \lambda\big)(h_n)^3\m\big\| H_n + 6K_n\big\|_{\L_p^n}\\[3pt]
&\mmmm + \gamma u M\big(\gamma - \lambda \big)\int_{t_n}^{t_{n+1}}\nm\int_{t_n}^t\,\int_{t_n}^s \big\|x_r - \wx_r\big\|_{\L_p^n}\,dr\,ds\,dt\\[3pt]
&\mm\leq \frac{1}{2}uM C_3(p)\sqrt{d\m}\m (h_n)^\frac{5}{2} + \frac{4}{35}\gamma^2\big(\gamma - \lambda \big)C_2(p)\sqrt{d\m}\m (h_n)^\frac{7}{2} + \frac{1}{24}\gamma u M\big(\gamma - \lambda \big) C_3(p)\sqrt{d\m}\m (h_n)^\frac{9}{2}\\[3pt]
&\mmmm + \frac{1}{6}uM\big(\gamma - \lambda \big)C_2(p)\sqrt{d\m}\m (h_n)^\frac{7}{2} + \frac{1}{6}\gamma^2\sigma\big(\gamma - \lambda\big)\m C_K(p)\sqrt{d\m}\m (h_n)^\frac{7}{2},
\end{align*}
where the last line follows directly from the estimates (\ref{eq:inside_estimate_intv}) and (\ref{eq:inside_estimate_x}) in Theorem \ref{thm:inside_estimate}.
\end{proof}
\begin{theorem}\label{thm:first_local_estimate_y} Let $\big\{y_n^\prime\big\}$ be the error process defined in Theorem \ref{thm:error_propagation}. Suppose that $f$ is $m$-strongly convex and its gradient $\nabla f$ is $M$-Lipschitz continuous. Then for all $n\geq 0$,
\begin{align}
\big\|y_n^\prime\big\|_{\L_2^n} \leq C_6\m\sqrt{d\m}\m (h_n)^\frac{5}{2},\label{eq:first_local_estimate_y}
\end{align}
where the constant $C_6$ is given by
\begin{align*}
C_6 := uM C_3(2) + \frac{1}{6}\gamma\bigg(\big(\gamma^2 + uM\big)C_2(2) + \gamma^2\sigma\m C_K(2) + \frac{1}{4}\gamma u M C_3(2) h_n\bigg)h_n\m.
\end{align*}
\end{theorem}
\begin{proof}
The result follows immediately from Minkowski's inequality and Theorem \ref{thm:first_local_estimate}.
\end{proof}
\begin{remark}
By applying the same arguments, we see that Theorems \ref{thm:inside_estimate}, \ref{thm:first_local_estimate} and \ref{thm:first_local_estimate_y} still hold when we use the (unconditional) $\L_p$ norm. This will be utilized in several proofs. 
\end{remark}
Using the above result, we obtain our first global error estimate for the shifted ODE.
\begin{theorem}[Global error estimate for the shifted ODE under standard assumptions]\label{thm:const_step_low_estimate}
Suppose the approximation $\big\{(\widetilde{x}_n\m, \widetilde{v}_n)\big\}_{n\m \geq\m 0}$ was obtained using a constant step size $h > 0$.
Let $\big\{y_n\big\}_{n\m \geq\m 0}$ be the error process between $(\m\widetilde{x}, \widetilde{v}\m)$ and the true diffusion $(x, v)$ defined in Theorem \ref{thm:exp_contract}.
Suppose $f$ is $m$-strongly convex and its gradient $\nabla f$ is $M$-Lipschitz. Then
\begin{align}\label{eq:const_step_low_estimate}
\big\|y_n\big\|_{\L_2} \leq e^{-n\alpha h}\big\|y_0\big\|_{\L_2} + \frac{1 - e^{-n\alpha h}}{1-e^{-\alpha h}}\,C_6\m\sqrt{d\m}\m h^\frac{5}{2}\m,
\end{align}
for $n\geq 0$, where the contraction rate $\alpha$ is given by
\begin{align*}
\alpha & = \frac{\big((\gamma - \lambda)^2 - uM\big)\vee \big(um - \lambda^2\big)}{\gamma - 2\lambda}\,.
\end{align*}
\end{theorem}
\begin{proof} Using the processes $\big(\big\{y_n\big\}$, $\big\{\widetilde{y}_n\big\}$, $\big\{y_n^\prime\big\}\big)$ defined in Theorems \ref{thm:exp_contract} and \ref{thm:error_propagation}, we have
\begin{align*}
\big\|y_{k+1}\big\|_{\L_2} & \leq \big\|y_{k+1} - y_{k+1}^\prime\big\|_{\L_2} + \big\|y_{k+1}^\prime\big\|_{\L_2}\\[3pt]
&= \big\|\widetilde{y}_{k+1}\big\|_{\L_2} + \big\|y_{k+1}^\prime\big\|_{\L_2}\\[3pt]
&\leq e^{-\alpha h}\big\|y_k\big\|_{\L_2} + C_6\m\sqrt{d\m}\m h^\frac{5}{2}\m,
\end{align*}
for $k\geq 0$, where the final line is the direct consequence of Theorems \ref{thm:exp_contract} and \ref{thm:first_local_estimate_y}. Therefore
\begin{align*}
\big\|y_n\big\|_{\L_2} & \leq e^{-\alpha h}\big\|y_{n-1}\big\|_{\L_2} + C_6\m\sqrt{d\m}\m h^\frac{5}{2}\\
&\hspace{2mm}\vdots\\[-10pt]
& \leq e^{-n\alpha h}\big\|y_0\big\|_{\L_2} + \sum_{k=0}^{n-1}\Big(e^{-k\alpha h}C_6\m\sqrt{d\m}\m h^\frac{5}{2}\Big),
\end{align*}
and the result follows.
\end{proof}
\begin{corollary}\label{cor:lipschitz_order}
Suppose that the underdamped Langevin diffusion $\big\{(x_t, v_t)\big\}_{t\geq 0}$ and the shifted ODE approximation $\big\{(\widetilde{x}_n\m, \widetilde{v}_n)\big\}_{n\m \geq\m 0}$ have the same initial velocity. Then for $n\geq 0$,
\begin{align}\label{eq:lipschitz_order}
\big\|y_n\big\|_{\L_2} \leq \sqrt{\lambda^2 + (\gamma - \lambda)^2}\,e^{-n\alpha h}\big\|\widetilde{x}_0 - x_0\big\|_{\L_2} + \Big(\m\frac{1}{\alpha} + h\Big)\m C_6\m \sqrt{d\m}\m h^\frac{3}{2}\m.
\end{align}
\end{corollary}
\begin{proof} Applying the inequalities $e^{-n\alpha h}\geq 0$ and $\alpha h \leq (1+\alpha h)(1-e^{-\alpha h})$ to (\ref{eq:const_step_low_estimate}) yields
\begin{align*}
\big\|y_n\big\|_{\L_2} \leq e^{-n\alpha h}\big\|y_0\big\|_{\L_2} + \Big(\m\frac{1}{\alpha} + h\Big)\m C_6\m \sqrt{d\m}\m h^\frac{3}{2}\m.
\end{align*}
Since we assume that $\widetilde{v}_0 = v_0$ (which is achievable in practice as $v_0 \sim \mathcal{N}(0, u\m I_d)$), we have
\begin{align*}
\big\|y_0\big\|_{\L_2}^2 & = \big\|\big(\lambda\m\widetilde{x}_0 + \widetilde{v}_0\big) - \big(\lambda\m x_0 + v_0\big)\big\|_{\L_2}^2 + \big\|\big((\gamma - \lambda)\m\widetilde{x}_0 + \widetilde{v}_0\big) - \big((\gamma - \lambda)\m x_0 + v_0\big)\big\|_{\L_2}^2\\[3pt]
& = \big\|\lambda\m(\widetilde{x}_0 - x_0)\big\|_{\L_2}^2 + \big\|(\gamma - \lambda)(\widetilde{x}_0 - x_0)\big\|_{\L_2}^2\\[3pt]
& = \Big(\lambda^2 + (\gamma - \lambda)^2\Big)\big\|\widetilde{x}_0 - x_0\big\|_{\L_2}^2,
\end{align*}
and the result follows.
\end{proof}
\begin{remark}
The above estimates holds if we set $K_n = 0$ in the definition of $\big\{(\widetilde{x}_n\m, \widetilde{v}_n)\big\}$.
Moreover, removing the $K_n$ random variable in the definition of the shifted ODE leads to the log-ODE method proposed by Castell and Gaines in \cite{CastellGaines} and recently studied in \cite{OptimalPoly}.
\end{remark}

In order to derive higher order estimates in cases where $f$ has additional smoothness,
we shall first present some global bounds for both the diffusion and shifted ODE processes.
It will be helpful to assume the diffusion starts at its invariant measure (i.e.~$(x_0, v_0)\sim \pi$).

\begin{theorem}[Global energy bounds for the ULD and shifted ODE]\label{thm:global_bounds} Suppose that $f$ is continuously differentiable, $m$-strongly convex and with an $M$-Lipschitz gradient $\nabla f$.
Consider an underdamped Langevin diffusion $(x, v)$ with its initial value $(x_0, v_0)\sim \pi$,
where $\pi$ denotes the stationary distribution. Then for $n\geq 0$, we have the following bounds:
\begin{align}
\big\|v_t\big\|_{\L_p} & = C_{|v|}(p)\sqrt{d\m}\m,\label{eq:global_v_bound} \\[3pt]
\big\|\m\wv_t\big\|_{\L_p} & \leq C_{|\wideparen{v}|}(p)\sqrt{d\m}\m,\label{eq:global_wv_bound}\\[3pt]
\big\|\nabla f(x_t)\big\|_{\L_2} & \leq \sqrt{Md\m}\m,\label{eq:global_fx_bound_1}\\[3pt]
\big\|\nabla f(x_t)\big\|_{\L_4} & \leq \fourthree\,\sqrt{Md\m}\m,\label{eq:global_fx_bound_2}
\end{align}
for $t\in[t_n, t_{n+1}]$ and $p\in\{2,4,8\}$, where the constants $C_{|v|}(p)$ and $C_{|\wideparen{v}|}(p)$ are defined as
\begin{align*}
C_{|v|}(p) & := C(p)\sqrt{u\m},\\[3pt]
C_{|\wideparen{v}|}(p) & := C_{|v|}(p) + C_4(p)\m (h_n)^\frac{1}{2}.
\end{align*}
\end{theorem}
\begin{proof}
The density function for the stationary distribution $\pi$ can be explicitly written as
\begin{align*}
\pi(x, v)\propto \exp\bigg(-f(x) - \frac{1}{2u}\|v\|_2^2\bigg).
\end{align*}
Thus $v_t\,\sfrac{}{}\sqrt{u}$ is a standard $d$-dimensional normal random variable and by Theorem \ref{thm:bm_scaling},
the first result (\ref{eq:global_v_bound}) follows. The second inequality can then be obtained using Minkowski's inequality and (\ref{eq:inside_estimate_v}) in Theorem \ref{thm:inside_estimate}.
The estimate (\ref{eq:global_fx_bound_1}) is given by Lemma 2 in \cite{DalalyanLangevin}.
In addition, we shall adapt the proof of Lemma 2 in \cite{DalalyanLangevin} to establish the estimate (\ref{eq:global_fx_bound_2}).
Thus, we will first consider the case when $d = 1$. Note that we already have the inequality
\begin{align*}
\int_{\R}\big(f^\prime(x)\big)^2\,\pi(x)\,dx \leq M,
\end{align*}
where $\pi(x)\propto \exp(-f(x))$ is the stationary measure for the position component of ULD.
Therefore, we can define the unnormalized density function:
\begin{align*}
\widetilde{\pi}(x) := \big(f^\prime(x)\big)^2\pi(x)\m.
\end{align*}
Since $f$ is $M$-Lipschitz continuous, it is differentiable almost everywhere and $|f^{\prime\prime}(x)|\leq M$ for all $x\in\R$ where the second derivative exists. By the fundamental theorem of calculus, (Theorem 7.20 in \cite{Rudin}), we have
\begin{align*}
f^\prime(x) - f^\prime(0) = \int_0^x f^{\prime\prime}(y)\,dy,
\end{align*}
for all $x\in\R$. It is also worth noting that
\begin{align*}
\widetilde{\pi}^{\m\prime}(x) & = \frac{d}{dx}\Big(\big(f^\prime(x)\big)^2\pi(x)\Big) =  2f^\prime(x)f^{\prime\prime}(x)\m\pi(x) - \big(f^\prime(x)\big)^3\pi^\prime(x),
\end{align*}
which implies that
\begin{align*}
f^\prime(x)\m\widetilde{\pi}(x) = 2f^\prime(x)f^{\prime\prime}(x)\m\pi(x) - \widetilde{\pi}^{\m\prime}(x).
\end{align*}
Following the proof of Lemma 2 in \cite{DalalyanLangevin} (but using $\widetilde{\pi}$ instead of $\pi$), we have the calculation:
\begin{align*}
&\int_{\R}\big(f^\prime(x)\big)^4\m\pi(x)\,dx\\
&\mm = \int_{\R}f^\prime(x)\big(f^\prime(x)\widetilde{\pi}(x)\big)\,dx\\[3pt]
&\mm = \int_{\R}f^\prime(x)\big(2f^\prime(x)f^{\prime\prime}(x)\m\pi(x) - \widetilde{\pi}^{\m\prime}(x)\big)\,dx\\[3pt]
&\mm = 2\int_{\R}\big(f^\prime(x)\big)^2f^{\prime\prime}(x)\m\pi(x)\,dx - \int_{\R}f^\prime(0)\m\widetilde{\pi}^{\m\prime}(x)\,dx - \int_{\R}\big(f^\prime(x)- f^\prime(0)\big)\m\widetilde{\pi}^{\m\prime}(x)\,dx\\[3pt]
&\mm = 2\int_{\R}\big(f^\prime(x)\big)^2f^{\prime\prime}(x)\m\pi(x)\,dx - f^\prime(0)\int_{\R}\widetilde{\pi}^{\m\prime}(x)\,dx - \int_{\R}\bigg(\int_0^x f^{\prime\prime}(y)\,dy\bigg)\widetilde{\pi}^{\m\prime}(x)\,dx\m.
\end{align*}
Since $\int_{\R}\widetilde{\pi}(x)\,dx \leq M$, we have $\widetilde{\pi}(x)\rightarrow 0$ as $x\rightarrow\pm\infty$. In particular, this means that the second term $f^\prime(0)\int_{\R}\widetilde{\pi}^{\m\prime}(x)\,dx = f^\prime(0)\int_{\R}d\widetilde{\pi}$ will be zero. Rewriting the last integral gives 
\begin{align*}
\int_{\R}\big(f^\prime(x)\big)^4\m\pi(x)\,dx & = 2\int_{\R}\big(f^\prime(x)\big)^2f^{\prime\prime}(x)\m\pi(x)\,dx\\[3pt]
&\mm - \int_0^\infty\bigg(\int_0^x f^{\prime\prime}(y)\,dy\bigg)\widetilde{\pi}^{\m\prime}(x)\,dx + \int_{-\infty}^0\bigg(\int_x^0 f^{\prime\prime}(y)\,dy\bigg)\widetilde{\pi}^{\m\prime}(x)\,dx.
\end{align*}
Therefore, it follows from Fubini's theorem that
\begin{align*}
\int_{\R}\big(f^\prime(x)\big)^4\m\pi(x)\,dx
& = 2\int_{\R}\big(f^\prime(x)\big)^2f^{\prime\prime}(x)\m\pi(x)\,dx + \int_0^\infty f^{\prime\prime}(y)\m\widetilde{\pi}(y)\,dy + \int_{-\infty}^0 f^{\prime\prime}(y)\m\widetilde{\pi}(y)\,dy.
\end{align*}
As the integral of $\widetilde{\pi} = (f^\prime)^2\pi$ is bounded by $M$ and $|f^{\prime\prime}|\leq M$ almost everywhere, we have
\begin{align*}
\bigg|\int_{\R}\big(f^\prime(x)\big)^4\m\pi(x)\,dx\m\bigg|
\leq 2\,\bigg|\int_{\R}\big(f^\prime(x)\big)^2f^{\prime\prime}(x)\m\pi(x)\,dx\m\bigg| + \bigg|\int_{\R} f^{\prime\prime}(y)\m\widetilde{\pi}(y)\,dy\m\bigg| \leq 3M^2.
\end{align*}
In the $d$-dimensional setting, we can apply this bound to each partial derivative of $f$
since every marginal distribution of $\pi$ will have the form $\exp(-g(x_i))$ where $g$ is $M$-Lipschitz.
Hence by Theorem \ref{thm:triangle_ineq}, we have
\begin{align*}
\E\Big[\big\|\nabla f(x)\big\|_2^4 \Big] = \E\Bigg[\Bigg(\sum_{i=1}^d \bigg|\frac{\partial f}{\partial x_i}\bigg|^2\Bigg)^2\,\Bigg]  \leq d\,\E\Bigg[\sum_{i=1}^d \bigg|\frac{\partial f}{\partial x_i}\bigg|^4\Bigg] = d\sum_{i=1}^d\E\Bigg[\bigg|\frac{\partial f}{\partial x_i}\bigg|^4\Bigg] \leq 3\m d^2 M^2,
\end{align*}
where $x\sim\pi$.
\end{proof}

In the next section, we shall often compare the diffusion and shifted ODE processes at times $t$ and $t_n$ (with $t\neq t_n$). As a result, we will require $\L_p$ estimates for such quantities. 

\begin{theorem}[Local growth estimates for diffusion and ODE approximation processes]\label{thm:growth_bounds}
Let $n\geq 0$ and $t\in[t_n, t_{n+1}]$. Then under the same assumptions as Theorem \ref{thm:global_bounds},
we have
\begin{align}
\big\|\m x_t - x_{t_n}\big\|_{\L_p} & \leq C_{x}(p)\sqrt{d\m}(t-t_n),\label{eq:growth_bounds_x}\\[3pt]
\big\|\m \wx_t - x_{t_n}\big\|_{\L_p} & \leq C_{\wideparen{x}}(p)\sqrt{d\m}\m(t-t_n),\label{eq:growth_bounds_wx}
\end{align}
for $p\in\{2,4,8\}$ and
\begin{align}
\big\|\m v_t - v_{t_n}\big\|_{\L_p} & \leq C_{v}(p)\sqrt{d\m}\m (t-t_n)^\frac{1}{2},\label{eq:growth_bounds_v}\\[3pt]
\big\|\m \wv_t - v_{t_n}\big\|_{\L_p} & \leq C_{\wideparen{v}}(p)\sqrt{d\m}\m (t-t_n)^\frac{1}{2},\label{eq:growth_bounds_wv}
\end{align}
for $p\in\{2,4\}$ where the constants $C_{x}(p)\m,\m C_{\wideparen{x}}(p)\m,\m C_{v}(p)\m,\m C_{\wideparen{v}}(p)$ are given by
\begin{align*}
C_{x}(p) & := C(p)\sqrt{u\m},\\[3pt]
C_{\wideparen{x}}(p) & := C_{x}(p) + C_3(p)(h_n)^\frac{1}{2},\\[3pt]
C_{v}(p) & := \begin{cases}\sigma\m C(2) + \Big(\gamma\m C_{|v|}(2) + u\sqrt{M\m}\m\Big)(h_n)^\frac{1}{2} & \text{if}\,\,\, p = 2\\[3pt]
\sigma\m C(4) + \Big(\gamma\m C_{|v|}(4) +  \fourthree\m u\sqrt{M\m}\m\Big)(h_n)^\frac{1}{2} & \text{if}\,\,\, p = 4
\end{cases}\,,\\[3pt]
C_{\wideparen{v}}(p) & := C_{v}(p) + C_4(p).
\end{align*}
\end{theorem}
\begin{proof}
The first inequality (\ref{eq:growth_bounds_x}) directly follows from the global energy bound (\ref{eq:global_v_bound}).
\begin{align*}
\big\|\m x_t - x_{t_n}\big\|_{\L_p} \leq \int_{t_n}^t \big\|\m v_s\big\|_{\L_p}\,ds \leq C(p)\sqrt{ud\m}\m(t-t_n).
\end{align*}
By applying Minkowski's inequality and the local estimate (\ref{eq:inside_estimate_x}) to the above, we have
\begin{align*}
\big\|\m \wx_t - x_{t_n}\big\|_{\L_p} & \leq \big\|\m x_t - x_{t_n}\big\|_{\L_p}  + \big\|\m \wx_t - x_t\big\|_{\L_p}  \leq C(p)\sqrt{ud\m}\m(t-t_n) + C_3(p)\sqrt{d\m}\m (h_n)^\frac{1}{2} (t-t_n),
\end{align*}
which is the second inequality (\ref{eq:growth_bounds_wx}). Similarly the bounds (\ref{eq:global_v_bound}), (\ref{eq:global_fx_bound_1}) and (\ref{eq:global_fx_bound_2}) give
\begin{align*}
\big\|\m v_t - v_{t_n}\big\|_{\L_2} & \leq \gamma\int_{t_n}^t\big\|\m v_s\big\|_{\L_2}\,ds + u\int_{t_n}^t\big\|\m \nabla f(x_s)\big\|_{\L_2}\,ds + \sigma\big\|W_{t_n, t}\big\|_{\L_2}\\[3pt]
& \leq \gamma\m C_{|v|}(2)\sqrt{d\m}\m(t-t_n) + u\sqrt{Md\m}\m(t-t_n) + \sigma\m C(2)\sqrt{d\m}\m (t-t_n)^\frac{1}{2},\\[9pt]
\big\|\m v_t - v_{t_n}\big\|_{\L_4} & \leq \gamma\int_{t_n}^t\big\|\m v_s\big\|_{\L_4}\,ds + u\int_{t_n}^t\big\|\m \nabla f(x_s)\big\|_{\L_4}\,ds + \sigma\big\|W_{t_n, t}\big\|_{\L_4}\\[3pt]
& \leq \gamma\m C_{|v|}(4)\sqrt{d\m}\m(t-t_n) + \fourthree\m u\sqrt{Md\m}\m(t-t_n) + \sigma\m C(4)\sqrt{d\m}\m (t-t_n)^\frac{1}{2},
\end{align*}
and the third inequality (\ref{eq:growth_bounds_v}) follows. Applying Minkowski's inequality as before yields
\begin{align*}
\big\|\m \wv_t - v_{t_n}\big\|_{\L_p} & \leq \big\|\m v_t - v_{t_n}\big\|_{\L_p} + \big\|\m \wv_t - v_t\big\|_{\L_p}\\[3pt]
&\leq C_v(p)\sqrt{d\m}\m (t-t_n)^\frac{1}{2} + C_4(p)\sqrt{d\m}\m (t-t_n)^\frac{1}{2},
\end{align*}
and the result follows.
\end{proof}
\begin{remark}
As before, these estimates are uniform and thus do not depend on $(x_0, v_0)$.
\end{remark}

\section{Estimates for shifted ODE under additional assumptions}

In this section, we shall derive high order error estimates for the shifted ODE method.
As one would expect, our analysis will require $\nabla f$ to have better than Lipschitz regularity.
Before proceeding to the main results of this section, we shall prove the following theorems:
\begin{theorem}\label{thm:M2_lip_estimate}
Suppose that $T:\R^d\rightarrow \R^{d\times d}$ is $M_2$-Lipschitz with respect to the $\|\cdot\|_2$ operator norm and $X_1, X_2, Y$ are $\R^d$-valued random variables. Then for $p\geq 1$, we have
\begin{align}\label{eq:M2_lip_estimate}
\big\|\big(T(X_1) - T(X_2)\big)Y\big\|_{\L_p} \leq M_2 \big\|X_1 - X_2\big\|_{\L_{2p}}\big\|Y\big\|_{\L_{2p}}.
\end{align}
\end{theorem}
\begin{proof} The result follows directly from the Lipschitz regularity of $T$ and definition of $\|\cdot\|_2\m$.
\begin{align*}
\big\|\big(T(X_1) - T(X_2)\big)Y\big\|_{\L_p} & = \E\Big[\big\|\big(T(X_1) - T(X_2)\big)Y\big\|_2^p\,\Big]^\frac{1}{p}\\[3pt]
& \leq \E\Big[\big\|T(X_1) - T(X_2)\big\|_2^p\,\big\|Y\big\|_2^p\,\Big]^\frac{1}{p}\\[3pt]
& \leq \E\Big[M_2^p\m\big\|X_1 - X_2\big\|_2^p\,\big\|Y\big\|_2^p\,\Big]^\frac{1}{p}\\[3pt]
& \leq M_2\Bigg(\E\bigg[\Big(\big\|X_1 - X_2\big\|_2^p\Big)^2\,\bigg]^\frac{1}{2}\E\bigg[\Big(\big\|Y\big\|_2^p\Big)^2\,\bigg]^\frac{1}{2}\Bigg)^\frac{1}{p},
\end{align*}
where the last line was obtained from H\"{o}lder's inequality.
\end{proof}

\begin{theorem}\label{thm:M3_lip_estimate}
Let $T:\R^d\rightarrow \R^{d\times d\times d}$ be $M_2$-bounded and $M_3$-Lipschitz with respect to the $\|\cdot\|_2$ operator norm and $X_1, X_2, Y_1, Y_2$ be $\R^d$-valued random variables. Then for $p\geq 1$,
\begin{align}
\big\|T(X_1)\big(Y_1, Y_2\big)\big\|_{\L_p} & \leq M_2\big\|Y_1\big\|_{\L_{2p}}\big\|Y_2\big\|_{\L_{2p}}\m,\label{eq:M3_lip_estimate1}\\[3pt]
\big\|\big(T(X_1) - T(X_2)\big)\big(Y_1, Y_2\big)\big\|_{\L_p} & \leq M_3 \big\|X_1 - X_2\big\|_{\L_{2p}}\big\|Y_1\big\|_{\L_{4p}}\big\|Y_2\big\|_{\L_{4p}},\label{eq:M3_lip_estimate2}
\end{align}
\end{theorem}
\begin{proof} The first estimate can be shown from H\"{o}lder's inequality and the definition of $\|\cdot\|_2$.
\begin{align*}
\big\|T(X_1)\big(Y_1, Y_2\big)\big\|_{\L_p}  & = \E\Big[\big\|T(X_1)\big(Y_1, Y_2\big)\big\|_2^p\Big]^\frac{1}{p}\\[3pt]
& \leq \E\Big[\big\|T(X_1)\big\|_2^p\,\big\|Y_1\big\|_2^p\,\big\|Y_2\big\|_2^p\Big]^\frac{1}{p}\\[3pt]
& \leq \E\Big[M_2^p\big\|Y_1\big\|_2^p\,\big\|Y_2\big\|_2^p\Big]^\frac{1}{p}\\[3pt]
& \leq M_2\,\E\Big[\big\|Y_1\big\|_2^{2p}\Big]^\frac{1}{2p}\,\E\Big[\big\|Y_2\big\|_2^{2p}\Big]^\frac{1}{2p}.
\end{align*}
The second bound follows by essentially the same argument as the proof of Theorem \ref{thm:M2_lip_estimate}.
\begin{align*}
\big\|\big(T(X_1) - T(X_2)\big)\big(Y_1, Y_2\big)\big\|_{\L_p} & = \E\Big[\big\|\big(T(X_1) - T(X_2)\big)\big(Y_1, Y_2\big)\big\|_2^p\Big]^\frac{1}{p}\\[3pt]
& \leq \E\Big[\big\|T(X_1) - T(X_2)\big\|_2^p\,\big\|Y_1\big\|_2^p\big\|Y_2\,\big\|_2^p\Big]^\frac{1}{p}\\[3pt]
& \leq M_3\,\E\Big[\big\|X_1 - X_2\big\|_2^p\,\big\|Y_1\big\|_2^p\,\big\|Y_2\big\|_2^p\Big]^\frac{1}{p}\\[3pt]
& \leq M_3\,\E\Big[\big\|X_1 - X_2\big\|_2^{2p}\Big]^\frac{1}{2p}\,\E\Big[\big\|Y_1\big\|_2^{2p}\m\big\|Y_2\big\|_2^{2p}\Big]^\frac{1}{2p}\\[3pt]
& \leq M_3\,\E\Big[\big\|X_1 - X_2\big\|_2^{2p}\Big]^\frac{1}{2p}\,\E\Big[\big\|Y_1\big\|_2^{4p}\Big]^\frac{1}{4p}\,\E\Big[\big\|Y_2\big\|_2^{4p}\Big]^\frac{1}{4p},
\end{align*}
where we applied H\"{o}lder's inequality in the last two lines.
\end{proof}
\begin{remark}
We shall use Theorem \ref{thm:M2_lip_estimate} with $T = \nabla^2 f$ and Theorem \ref{thm:M3_lip_estimate} with $T = \nabla^3 f$.
This will be fruitful as we already have $\L_p$ estimates with $p\in\{2,4,8\}$ from Appendix \ref{appen:standard}.
\end{remark}

\subsection{The Hessian of $f$ is Lipschitz}\label{append:hessian} In this subsection, we make the assumptions:
\begin{itemize}\renewcommand{\labelitemi}{\small$\bullet$}
\item The function $f$ is twice continuously differentiable and $m$-strongly convex where
\subitem \nm $\circ$ the gradient $\nabla f$ is $M$-Lipschitz continuous,
\subitem \nm $\circ$ and the Hessian $\nabla^2 f$ is $M_2$-Lipschitz continuous.
\item The underdamped Langevin diffusion $(x,v)$ has an initial condition $(x_0, v_0)\sim \pi$,
where $\pi$ denotes the stationary distribution. It follows that $(x_t, v_t)\sim \pi$ for $t\geq 0$.
\end{itemize}

Since we are imposing further regularity on $f$, the results from Appendix \ref{appen:standard} still hold.
To begin, we will estimate the ``$\m\nabla f\m$'' integral which appears in the expansion of $\big(\m\wv_t - v_t\big)$.
\begin{theorem}\label{thm:fx_hessian_estimate} For each $n\geq 0$, we have
\begin{align}\label{eq:fx_hessian_estimate}
\Bigg\|\int_{t_n}^{t_{n+1}}\big(\nabla f(\wx_t) - \nabla f(x_t)\big)\,dt\,\Bigg\|_{\L_2} \leq C_7\m d\m (h_n)^\frac{7}{2},
\end{align}
where the constant $C_7$ is defined as
\begin{align*}
C_7 & := M_2\bigg(\frac{1}{6}\m C_3(4)\m C_{|\wideparen{v}|}(4) + \frac{4}{35}\m C_4(4)\m C_{\wideparen{x}}(4)+ \frac{1}{6}\m\sigma \m C_{x}(4)\m C_K(4)\bigg)\\
&\mmm + M d^{\m - \frac{1}{2}}\bigg(\frac{8}{105}\gamma\m C_4(2) + \frac{1}{6}\gamma\sigma\m C_K(2) + \frac{1}{24}uM\m C_3(2)h_n\bigg).
\end{align*}
\end{theorem}
\begin{proof}
Applying Minkowski's inequality to the expansion (\ref{eq:fx_integral_expansion}) in Theorem \ref{thm:fx_integral_expansion} gives
\begin{align*}
&\Bigg\|\int_{t_n}^{t_{n+1}}\big(\nabla f(\wx_t) - \nabla f(x_t)\big)\,dt\,\Bigg\|_{\L_2}\\
&\mmm \leq \Bigg\|\int_{t_n}^{t_{n+1}}\nm\int_{t_n}^t \big(\m\nabla^2 f(\wx_s) - \nabla^2 f(x_s)\big)\m\wv_s\,ds\,dt\,\Bigg\|_{\L_2}\\[3pt]
&\mmmmm + \Bigg\|\int_{t_n}^{t_{n+1}}\nm\int_{t_n}^t \big(\m\nabla^2 f(x_s) - \nabla^2 f(x_{t_n})\big)(\m\wv_s - v_s)\,ds\,dt\,\Bigg\|_{\L_2}\\[3pt]
&\mmmmm + \Bigg\|\int_{t_n}^{t_{n+1}}\nm\int_{t_n}^t \nabla^2 f(x_{t_n})\big(\m\wv_s - v_s + \sigma \widetilde{B}_{t_n, s}\big)\,ds\,dt\,\Bigg\|_{\L_2}\\[3pt]
&\mmmmm + \Bigg\|\int_{t_n}^{t_{n+1}}\nm\int_{t_n}^t \sigma\big(\m\nabla^2 f(x_s) - \nabla^2 f(x_{t_n})\big)\big(H_n+6K_n\big)\,ds\,dt\,\Bigg\|_{\L_2}.
\end{align*}
We shall estimate each of these terms using results from Appendix \ref{appen:standard} and Theorem \ref{thm:M2_lip_estimate}.
\begin{align*}
&\Bigg\|\int_{t_n}^{t_{n+1}}\nm\int_{t_n}^t \big(\m\nabla^2 f(\wx_s) - \nabla^2 f(x_s)\big)\m\wv_s\,ds\,dt\,\Bigg\|_{\L_2}\nm\\
&\mmm \leq \int_{t_n}^{t_{n+1}}\nm\int_{t_n}^t\big\|\big(\m\nabla^2 f(\wx_s) - \nabla^2 f(x_s)\big)\m\wv_s\big\|_{\L_2}\,ds\,dt\\[3pt]
&\mmm \leq M_2 \int_{t_n}^{t_{n+1}}\nm\int_{t_n}^t \big\|\m\wx_s - x_s\big\|_{\L_4}\big\|\m\wv_s\big\|_{\L_4}\,ds\,dt\\[3pt]
&\mmm \leq \frac{1}{6}M_2\m C_3(4)\m C_{|\wideparen{v}|}(4)\m d\m (h_n)^\frac{7}{2},
\end{align*}
where the final line follows by the local estimate (\ref{eq:inside_estimate_x}) and the global bound (\ref{eq:global_wv_bound}).
Similarly, by (\ref{eq:inside_estimate_v}) and the growth estimate (\ref{eq:growth_bounds_wv}), the second term can be estimated as
\begin{align*}
&\Bigg\|\int_{t_n}^{t_{n+1}}\nm\int_{t_n}^t \big(\m\nabla^2 f(x_s) - \nabla^2 f(x_{t_n})\big)(\m\wv_s - v_s)\,ds\,dt\,\Bigg\|_{\L_2}\\
&\mmm \leq \int_{t_n}^{t_{n+1}}\nm\int_{t_n}^t\big\|\big(\m\nabla^2 f(x_s) - \nabla^2 f(x_{t_n})\big)(\m\wv_s - v_s)\big\|_{\L_2}\,ds\,dt\\[3pt]
&\mmm \leq M_2\int_{t_n}^{t_{n+1}}\nm\int_{t_n}^t \big\|\m\wx_s - x_{t_n}\big\|_{\L_4}\big\|\m\wv_s - v_s\big\|_{\L_4}\,ds\,dt\\[3pt]
&\mmm \leq \frac{4}{35}M_2\m C_4(4)\m C_{\wideparen{x}}(4)\m d\m (h_n)^\frac{7}{2}.
\end{align*}
The third term can be directly estimated without the Lipschitz assumption on $\nabla^2 f$.
\begin{align*}
&\Bigg\|\int_{t_n}^{t_{n+1}}\nm\int_{t_n}^t \nabla^2 f(x_{t_n})\big(\m\wv_s - v_s + \sigma \widetilde{B}_{t_n, s}\big)\,ds\,dt\,\Bigg\|_{\L_2}\\
& \leq \int_{t_n}^{t_{n+1}}\nm\int_{t_n}^t\big\|\nabla^2 f(x_{t_n})\big(\m\wv_s - v_s + \sigma \widetilde{B}_{t_n, s}\big)\big\|_{\L_2}\,ds\,dt\\[3pt]
&\leq \int_{t_n}^{t_{n+1}}\nm\int_{t_n}^t\big\|\nabla^2 f(x_{t_n})\big\|_{\L_2}\big\|\m\wv_s - v_s + \sigma \widetilde{B}_{t_n, s}\big\|_{\L_2}\,ds\,dt\\[3pt]
& \leq M\int_{t_n}^{t_{n+1}}\nm\int_{t_n}^t\bigg\|\gamma\int_{t_n}^s \Big(v_r - \wv_r - \sigma\big(H_n + 6K_n\big)\Big)\,dr + u\int_{t_n}^s \big(\m\nabla f(x_r) - \nabla f(\wx_r)\big)\,dr\bigg\|_{\L_2}ds\,dt\\[3pt]
& \leq \gamma M\int_{t_n}^{t_{n+1}}\nm\int_{t_n}^t\,\int_{t_n}^s \big\|\m v_r - \wv_r\big\|_{\L_2}\,dr\,ds\,dt + \gamma\sigma M\int_{t_n}^{t_{n+1}}\nm\int_{t_n}^t\,\int_{t_n}^s \big\| H_n + 6K_n\big\|_{\L_2}\,dr\,ds\,dt\\[3pt]
&\mmm + uM^2\int_{t_n}^{t_{n+1}}\nm\int_{t_n}^t\,\int_{t_n}^s \big\|\m x_r - \wx_r\big\|_{\L_2}\,dr\,ds\,dt\\[3pt]
& \leq \frac{8}{105}\gamma M\m C_4(2)\sqrt{d\m}\m(h_n)^\frac{7}{2} + \frac{1}{6}\gamma\sigma M\m C_K(2)\sqrt{d\m}\m(h_n)^\frac{7}{2} + \frac{1}{24}uM^2\m C_3(2)\sqrt{d\m}\m(h_n)^\frac{9}{2},
\end{align*}
where the upper bound $\big\|\nabla^2 f(x_{t_n})\big\|_{\L_2} \leq M$ follows from the Lipschitz regularity of $\nabla f$.
Just as for the second term, we will estimate the last term using Theorems \ref{thm:growth_bounds} and \ref{thm:M2_lip_estimate}.
\begin{align*}
&\Bigg\|\int_{t_n}^{t_{n+1}}\nm\int_{t_n}^t \sigma\big(\m\nabla^2 f(x_s) - \nabla^2 f(x_{t_n})\big)\big(H_n+6K_n\big)\,ds\,dt\,\Bigg\|_{\L_2}\\
&\mm\leq \sigma\int_{t_n}^{t_{n+1}}\nm\int_{t_n}^t\big\|\big(\m\nabla^2 f(x_s) - \nabla^2 f(x_{t_n})\big)\big(H_n+6K_n\big)\big\|_{\L_2}\,ds\,dt\\
&\mm\leq \sigma M_2\int_{t_n}^{t_{n+1}}\nm\int_{t_n}^t\big\|x_s - x_{t_n}\big\|_{\L_4}\big\|H_n+6K_n\big\|_{\L_4}\,ds\,dt\\
&\mm\leq \frac{1}{6}\sigma M_2\m C_{x}(4)\m C_K(4)\m d\m (h_n)^\frac{7}{2},
\end{align*}
and the result follows.
\end{proof}

Using Theorem \ref{thm:fx_hessian_estimate}, we can derive higher order estimates for the ODE approximation.

\begin{theorem}\label{thm:second_local_estimate}
For $n\geq 0$ and $\lambda\in [0,\gamma]$, we have
\begin{align}
\Big\|\big(\lambda\m x_{t_{n+1}} + v_{t_{n+1}}\big) - \big(\lambda\m x_{t_{n+1}}^\prime + v_{t_{n+1}}^\prime\big)\Big\|_{\L_2} \leq C_8(\lambda)\m d\m (h_n)^\frac{7}{2},\label{eq:second_local_estimate}
\end{align}
where the constant $C_8(\lambda)$ is given by
\begin{align*}
C_8(\lambda) & := u\m C_7 + \big(\gamma - \lambda\big)d^{\m -\frac{1}{2}}\bigg(\m\frac{4}{35}\gamma^2\m C_2(2) + \frac{1}{6}uM C_3(2) + \frac{1}{6}\gamma^2\sigma\m C_K(2) + \frac{1}{24}\gamma\m u MC_3(2)\m h_n\bigg).
\end{align*}
\end{theorem}
\begin{proof}
Just as for Theorem \ref{thm:first_local_estimate}, we shall use the expansion (\ref{eq:key_expansion}) of $(\lambda x + v) - (\lambda x^\prime + v^\prime)$.
\begin{align*}
&\big(\lambda\m x_{t_{n+1}} + v_{t_{n+1}}\big) - \big(\lambda\m x_{t_{n+1}}^\prime + v_{t_{n+1}}^\prime\big)\\[3pt]
&\mm = \big(\gamma - \lambda\big)\Bigg(\gamma^2 \int_{t_n}^{t_{n+1}}\nm\int_{t_n}^t\,\int_{t_n}^s \big(v_r - \wv_r\big)\,dr\,ds\,dt + u \int_{t_n}^{t_{n+1}}\nm\int_{t_n}^t\big(\m\nabla f(x_s) - \nabla f(\wx_s)\big)\,ds\,dt\\
&\hspace{26.5mm} + \frac{1}{6}\gamma^2\sigma (h_n)^3\big(H_n + 6K_n\big) + \gamma u\int_{t_n}^{t_{n+1}}\nm\int_{t_n}^t\,\int_{t_n}^s \big(\m\nabla f(x_r) - \nabla f(\wx_r)\big)\,dr\,ds\,dt\Bigg)\\
&\mmmm + u\int_{t_n}^{t_{n+1}}\big(\m\nabla f(\wx_t) - \nabla f(x_t)\big)\,dt.
\end{align*}
Using Minkowski's inequality and the Lipschitz regularity of $\nabla f$, we have the $\L_2$ estimate:
\begin{align*}
&\Big\|\big(\lambda\m x_{t_{n+1}} + v_{t_{n+1}}\big) - \big(\lambda\m x_{t_{n+1}}^\prime + v_{t_{n+1}}^\prime\big)\Big\|_{\L_2}\\[3pt]
&\mm \leq \gamma^2\big(\gamma - \lambda \big)\int_{t_n}^{t_{n+1}}\nm\int_{t_n}^t\,\int_{t_n}^s \big\|v_r - \wv_r\big\|_{\L_2}\,dr\,ds\,dt + uM\big(\gamma - \lambda \big)\int_{t_n}^{t_{n+1}}\nm\int_{t_n}^t\big\|x_s - \wx_s\big\|_{\L_2}\,ds\,dt\\[3pt]
&\mmmm + \frac{1}{6}\gamma^2\sigma\big(\gamma - \lambda\big)C_K(p)\sqrt{d\m}\m (h_n)^\frac{7}{2} + \gamma u M\big(\gamma - \lambda \big)\int_{t_n}^{t_{n+1}}\nm\int_{t_n}^t\,\int_{t_n}^s \big\|x_r - \wx_r\big\|_{\L_2}\,dr\,ds\,dt\\[3pt]
&\mmmm + u\,\Bigg\|\int_{t_n}^{t_{n+1}}\big(\nabla f(\wx_t) - \nabla f(x_t)\big)\,dt\,\Bigg\|_{\L_2}.
\end{align*}
Therefore by Theorem \ref{thm:fx_hessian_estimate} along with the local estimates (\ref{eq:inside_estimate_intv}) and (\ref{eq:inside_estimate_x}), it follows that
\begin{align*}
&\Big\|\big(\lambda\m x_{t_{n+1}} + v_{t_{n+1}}\big) - \big(\lambda\m x_{t_{n+1}}^\prime + v_{t_{n+1}}^\prime\big)\Big\|_{\L_2}\\[3pt]
&\mm \leq u\m C_6\m d\m (h_n)^\frac{7}{2} + \frac{4}{35}\gamma^2\big(\gamma - \lambda \big)C_2(2)\sqrt{d\m}\m (h_n)^\frac{7}{2} + \frac{1}{6}uM\big(\gamma - \lambda\big)C_3(2)\sqrt{d\m}\m (h_n)^\frac{7}{2}\\[3pt]
&\mmmm + \frac{1}{6}\gamma^2\sigma\big(\gamma - \lambda\big)C_K(2)\sqrt{d\m}\m (h_n)^\frac{7}{2} + \frac{1}{24}\gamma u M\big(\gamma - \lambda \big)C_3(2)\sqrt{d\m}\m (h_n)^\frac{9}{2},
\end{align*}
which gives the desired result.
\end{proof}

\begin{theorem}\label{thm:second_local_estimate_y} Let $\big\{y_n^\prime\big\}$ be the error process defined in Theorem \ref{thm:error_propagation}. Then for $n\geq 0$,
\begin{align}
\big\|y_n^\prime\big\|_{\L_2} \leq C_8\m d\m (h_n)^\frac{7}{2},\label{eq:second_local_estimate_y}
\end{align}
where the constant $C_8$ is defined as
\begin{align*}
C_8 := 2u\m C_6 + \frac{1}{6}\gamma\m d^{\m -\frac{1}{2}}\bigg(\gamma^2\m C_2(2) + uM C_3(2) + \gamma^2\sigma\m C_K(2) + \frac{1}{4}\gamma\m u MC_3(2)\m h_n\bigg).
\end{align*}
\end{theorem}
\begin{proof}
The result follows immediately from Minkowski's inequality and Theorem \ref{thm:second_local_estimate}.
\end{proof}

Using the above, we arrive at our second global error estimate for the shifted ODE.

\begin{theorem}[Global error estimate for the shifted ODE with a Lipschitz Hessian]\label{thm:const_step_med_estimate}
Suppose the approximation $\big\{(\widetilde{x}_n\m, \widetilde{v}_n)\big\}_{n\m \geq\m 0}$ was obtained using a constant step size $h > 0$.
Let $\big\{y_n\big\}_{n\m \geq\m 0}$ be the error process between $(\m\widetilde{x}\m, \widetilde{v}\m)$ and the true diffusion $(x, v)$ defined in Theorem \ref{thm:exp_contract}.
Then under the assumptions detailed at the start of this subsection, we have
\begin{align}\label{eq:const_step_med_estimate}
\big\|y_n\big\|_{\L_2} \leq e^{-n\alpha h}\big\|y_0\big\|_{\L_2} + \frac{1 - e^{-n\alpha h}}{1-e^{-\alpha h}}\,C_7\m d\m h^\frac{7}{2}\m,
\end{align}
for $n\geq 0$, where the contraction rate $\alpha$ is given by
\begin{align*}
\alpha & = \frac{\big((\gamma - \lambda)^2 - uM\big)\vee \big(um - \lambda^2\big)}{\gamma - 2\lambda}\,.
\end{align*}
\end{theorem}
\begin{proof}
The result follows using exactly the same argument as the proof of Theorem \ref{thm:const_step_low_estimate}.
The only difference is that we use Theorem \ref{thm:second_local_estimate_y} to estimate $y^\prime$ instead of Theorem \ref{thm:first_local_estimate_y}.
\end{proof}

Just as in Appendix \ref{appen:standard}, we can deduce the order of convergence for the shifted ODE.

\begin{corollary}\label{cor:hessian_order}
Suppose that the underdamped Langevin diffusion $\big\{(x_t, v_t)\big\}_{t\geq 0}$ and the shifted ODE approximation $\big\{(\widetilde{x}_n\m, \widetilde{v}_n)\big\}_{n\m \geq\m 0}$ have the same initial velocity. Then for $n\geq 0$,
\begin{align}\label{eq:hessian_order}
\big\|y_n\big\|_{\L_2} \leq \sqrt{\lambda^2 + (\gamma - \lambda)^2}\,e^{-n\alpha h}\big\|\widetilde{x}_0 - x_0\big\|_{\L_2} + \Big(\m\frac{1}{\alpha} + h\Big)\m C_7\m d\m h^\frac{5}{2}\m.
\end{align}
\end{corollary}
\begin{proof} Applying the inequalities $e^{-n\alpha h}\geq 0$ and $\alpha h \leq (1+\alpha h)(1-e^{-\alpha h})$ to (\ref{eq:const_step_med_estimate}) yields
\begin{align*}
\big\|y_n\big\|_{\L_2} \leq e^{-n\alpha h}\big\|y_0\big\|_{\L_2} + \Big(\m\frac{1}{\alpha} + h\Big)\m C_7\m d\m h^\frac{5}{2}\m.
\end{align*}
Since we assume that $\widetilde{v}_0 = v_0$ (which is achievable in practice as $v_0 \sim \mathcal{N}(0, u\m I_d)$), we have
\begin{align*}
\big\|y_0\big\|_{\L_2}^2 & = \big\|\big(\lambda\m\widetilde{x}_0 + \widetilde{v}_0\big) - \big(\lambda\m x_0 + v_0\big)\big\|_{\L_2}^2 + \big\|\big((\gamma - \lambda)\m\widetilde{x}_0 + \widetilde{v}_0\big) - \big((\gamma - \lambda)\m x_0 + v_0\big)\big\|_{\L_2}^2\\[3pt]
& = \big\|\lambda\m(\widetilde{x}_0 - x_0)\big\|_{\L_2}^2 + \big\|(\gamma - \lambda)(\widetilde{x}_0 - x_0)\big\|_{\L_2}^2\\[3pt]
& = \Big(\lambda^2 + (\gamma - \lambda)^2\Big)\big\|\widetilde{x}_0 - x_0\big\|_{\L_2}^2,
\end{align*}
and the result follows.
\end{proof}

\subsection{The Hessian and third derivative of $f$ are Lipschitz}\label{append:third} We now assume that
\begin{itemize}\renewcommand{\labelitemi}{\small$\bullet$}
\item The function $f$ is three times continuously differentiable and $m$-strongly convex where
\subitem \nm $\circ$ the gradient $\nabla f$ is $M$-Lipschitz continuous,
\subitem \nm $\circ$ the Hessian $\nabla^2 f$ is $M_2$-Lipschitz continuous,
\subitem \nm $\circ$ and the third derivative $\nabla^3 f$ is $M_3$-Lipschitz continuous.

\item The underdamped Langevin diffusion $(x,v)$ has an initial condition $(x_0, v_0)\sim \pi$,
where $\pi$ denotes the stationary distribution. So just as before, $(x_t, v_t)\sim \pi$ for $t\geq 0$.
\end{itemize}

Since this includes previous assumptions, the results from Appendices \ref{appen:standard} and \ref{append:hessian} hold.
Our approach will follow the same strategy as the proof of Theorem 1 in \cite{RungeKuttaMCMC} where estimates are obtained using both local mean-square and mean deviation error estimates.
Since we have already obtained good $\L_2$ estimates, we shall focus on mean deviation errors.

\begin{theorem}[High order Taylor expansions]\label{thm:high_order_taylor} For all $n\geq 0$ and $t\in [t_n, t_{n+1}]$, we have
\begin{align*}
\nabla f(x_t) & = \nabla f(x_{t_n}) + \nabla^2 f(x_{t_n})\big(x_t - x_{t_n}\big) + \nabla^3 f(x_{t_n})\int_{t_n}^t\big(x_t - x_{t_n}\big)\,dx_t + R_n(t),\\[3pt]
\nabla f(\wx_t) & =  \nabla f(x_{t_n}) + \nabla^2 f(x_{t_n})\big(\m\wx_t - x_{t_n}\big) + \nabla^3 f(x_{t_n})\int_{t_n}^t\big(\wx_t - x_{t_n}\big)\,d\m\wx_t + \wideparen{R}_n(t),
\end{align*}
where the remaining terms $R_n(t)$ and $\wideparen{R}_n(t)$ satisfy the following estimates: 
\begin{align}
\big\|R_n(t)\big\|_{\L_2} & \leq C_8^{(1)}d^\frac{3}{2}(t-t_n)^3,\label{eq:small_remainder_r}\\[3pt]
\big\|\wideparen{R}_n(t)\big\|_{\L_2} & \leq C_8^{(2)}d^\frac{3}{2}(t-t_n)^3,\label{eq:small_remainder_wr}
\end{align}
with the constants $C_8^{(1)}$ and $C_8^{(2)}$ defined as
\begin{align*}
C_8^{(1)} & := \frac{1}{6}M_3\m C_{x}(4)\m C_{|v|}^2(8)\m, \\[3pt]
C_8^{(2)} & := \frac{1}{6}M_3\m C_{\wideparen{x}}(4)\Big(C_{|\wideparen{v}|}(8) + \sigma\m C_K(8)\m (h_n)^\frac{1}{2}\Big)^2.
\end{align*}
\end{theorem}
\begin{proof}
We first recall the Taylor expansions (\ref{eq:fx_expansion}) and (\ref{eq:fwx_expansion}) given by Theorem \ref{thm:fx_expansion}.
\begin{align*}
\nabla f(x_t) & = \nabla f(x_{t_n}) + \nabla^2 f(x_{t_n})\big(x_t - x_{t_n}\big) + \nabla^3 f(x_{t_n})\int_{t_n}^t\big(x_s - x_{t_n}\big)\,dx_s\\
&\mmm + \int_{t_n}^t\int_{t_n}^s\big(\m\nabla^3 f(x_r) - \nabla^3 f(x_{t_n})\big)\,dx_r\,dx_s\m,\\[6pt]
\nabla f(\wx_t) & = \nabla f(x_{t_n}) + \nabla^2 f(x_{t_n})\big(\m \wx_t - x_{t_n}\big) + \nabla^3 f(x_{t_n})\int_{t_n}^t\big(\m\wx_s - x_{t_n}\big)\,d\m\wx_s\\
&\mmm + \int_{t_n}^t\int_{t_n}^s\big(\m\nabla^3 f(\wx_r) - \nabla^3 f(x_{t_n})\big)\,d\m\wx_r\,d\m\wx_s\m.
\end{align*}
We can estimate the remainder term $R_n$ using the bounds from Theorems \ref{thm:global_bounds} and \ref{thm:growth_bounds}.
\begin{align*}
\big\|R_n(t)\big\|_{\L_2}
& = \Bigg\|\int_{t_n}^t\int_{t_n}^s\big(\m\nabla^3 f(x_r) - \nabla^3 f(x_{t_n})\big)\big(v_r, v_s\big)\,dr\,ds\m\Bigg\|_{\L_2}\\[3pt]
& \leq \int_{t_n}^t\int_{t_n}^s\big\|\big(\m\nabla^3 f(x_r) - \nabla^3 f(x_{t_n})\big)\big(v_r, v_s\big)\big\|_{\L_2}\,dr\,ds\\[3pt]
& \leq M_3\int_{t_n}^t\int_{t_n}^s\big\|x_r - x_{t_n}\big\|_{\L_4}\m\big\|v_r\big\|_{\L_8}\m\big\|v_s\big\|_{\L_8}\,dr\,ds\\[3pt]
& \leq \frac{1}{6}M_3\m C_{x}(4)\m C_{|v|}^2(8)\m d^\frac{3}{2}(t-t_n)^3,
\end{align*}
where the penultimate line follows from Theorem \ref{thm:M3_lip_estimate}. Similarly, we can estimate $\wideparen{R}_n$ as
\begin{align*}
&\big\|\wideparen{R}_n(t)\big\|_{\L_2}\\[3pt]
&\mm = \Bigg\|\int_{t_n}^t\int_{t_n}^s\big(\m\nabla^3 f(\wx_r) - \nabla^3 f(x_{t_n})\big)\big(\wv_r + \sigma\big(H_n + 6K_n\big), \wv_s + \sigma\big(H_n + 6K_n\big)\big)\,dr\,ds\m\Bigg\|_{\L_2}\\[3pt]
&\mm \leq \int_{t_n}^t\int_{t_n}^s\big\|\big(\m\nabla^3 f(\wx_r) - \nabla^3 f(x_{t_n})\big)\big(\wv_r + \sigma\big(H_n + 6K_n\big), \wv_s + \sigma\big(H_n + 6K_n\big)\big)\big\|_{\L_2}\,dr\,ds\\[3pt]
&\mm \leq M_3\int_{t_n}^t\int_{t_n}^s\big\|\wx_r - x_{t_n}\big\|_{\L_4}\m\big\|\m\wv_r + \sigma\big(H_n + 6K_n\big)\big)\big\|_{\L_8}\m\big\|\m\wv_s + \sigma\big(H_n + 6K_n\big)\big)\big\|_{\L_8}\,dr\,ds\\[3pt]
&\mm \leq \frac{1}{6}M_3\m C_{\wideparen{x}}(4)\Big(C_{|\wideparen{v}|}(8) + \sigma\m C_K(8)\m (h_n)^\frac{1}{2}\Big)^2 d^\frac{3}{2}(t-t_n)^3,
\end{align*}
which gives the desired result.
\end{proof}

We now turn our attention to the mean deviation error between $\nabla f(\wx_t)$ and $\nabla f(x_t)$.

\begin{theorem}\label{thm:fx_third_estimate}
For $n\geq 0$ and $t\in[t_n, t_{n+1}]$, we have
\begin{align}\label{eq:fx_third_estimate}
&\Big\|\m\E_n\Big[\nabla f(\wx_t) - \nabla f(x_t)\Big]\Big\|_{\L_2} \leq C_8\,d^\frac{3}{2}\m (t-t_n)^3 + C_9\, d\m h_n(t-t_n)^2,
\end{align}
where $\displaystyle C_8 := C_8^{(1)} + C_8^{(2)}$ and the constant $C_9$ is given by
\begin{align*}
C_9 & := M_2\bigg(\frac{2}{9}\m C_{\wideparen{v}}(4)\m C_{\wideparen{v}}(4) + \frac{14}{25}\m C_{\wideparen{v}}(4)\m C_K(4) + \frac{1}{2}\big(C_K(4)\big)^2 + \frac{1}{6}\big(C_v(4)\big)^2\bigg)\\[3pt]
&\mmmm +  \frac{1}{3}M C_5(p)d^{\m-\frac{1}{2}}\m (h_n)^\frac{1}{2} + \frac{1}{3}M_2\m C_5(4)\m C_{|v|}(4) (h_n)^\frac{3}{2}.
\end{align*}
\end{theorem}
\begin{proof} It follows from the Taylor expansions presented in Theorem \ref{thm:high_order_taylor} that
\begin{align*}
&\E_n\Big[\nabla f(\wx_t) - \nabla f(x_t)\Big] \\[3pt]
& = \E_n\Bigg[\m\nabla^2 f(x_{t_n})\big(\m\wx_t - x_t\big) + \nabla^3 f(x_{t_n})\int_{t_n}^t\big(\wx_s - x_{t_n}\big)\,d\m\wx_s - \nabla^3 f(x_{t_n})\int_{t_n}^t\big(x_s - x_{t_n}\big)\,dx_s\Bigg]\\[3pt]
&\mmm + \E_n\big[\wideparen{R}_n(t) - R_n(t)\big]\\[3pt]
& = \nabla^2 f(x_{t_n})\,\E_n\big[\m\wx_t - x_t\big] - \int_{t_n}^t\int_{t_n}^s \E_n\Big[\m\nabla^3 f(x_{t_n})\big(v_r,v_s\big)\Big]\,dr\,ds + \E_n\big[\wideparen{R}_n(t)\big] - \E_n\big[R_n(t)\big]\\[3pt]
&\mmm +\int_{t_n}^t\int_{t_n}^s \E_n\Big[\m\nabla^3 f(x_{t_n})\big(\wv_r + \sigma\big(H_n + 6K_n\big), \wv_s + \sigma\big(H_n + 6K_n\big)\big)\Big]\,dr\,ds\\[3pt]
& = \int_{t_n}^t\int_{t_n}^s \E_n\Big[\m\nabla^3 f(x_{t_n})\big(\wv_r - v_{t_n} + \sigma\big(H_n + 6K_n\big), \wv_s - v_{t_n} + \sigma\big(H_n + 6K_n\big)\big)\Big]\,dr\,ds\\[3pt]
&\mmm + \int_{t_n}^t\int_{t_n}^s \E_n\Big[\m\nabla^3 f(x_{t_n})\big(v_{t_n}, \wv_s - v_s + \sigma\big(H_n + 6K_n\big)\big)\Big]\,dr\,ds + \E_n\big[\wideparen{R}_n(t)\big]\\[3pt]
&\mmm + \int_{t_n}^t\int_{t_n}^s \E_n\Big[\m\nabla^3 f(x_{t_n})\big(\wv_r - v_r + \sigma\big(H_n + 6K_n\big), v_{t_n}\big)\Big]\,dr\,ds - \E_n\big[R_n(t)\big]\\[3pt]
&\mmm - \int_{t_n}^t\int_{t_n}^s \E_n\Big[\m\nabla^3 f(x_{t_n})\big(v_r - v_{t_n},v_s - v_{t_n}\big)\Big]\,dr\,ds + \nabla^2 f(x_{t_n})\,\E_n\big[\m\wx_t - x_t\big]\\[3pt]
& = \int_{t_n}^t\int_{t_n}^s \E_n\Big[\m\nabla^3 f(x_{t_n})\big(\wv_r - v_{t_n} + \sigma\big(H_n + 6K_n\big), \wv_s - v_{t_n} + \sigma\big(H_n + 6K_n\big)\big)\Big]\,dr\,ds\\[3pt]
&\mmm + \int_{t_n}^t\int_{t_n}^s \nabla^3 f(x_{t_n})\big(v_{t_n}, \E_n\big[\m\wv_s - v_s\big]\big)\,dr\,ds + \E_n\big[\wideparen{R}_n(t)\big]\\[3pt]
&\mmm + \int_{t_n}^t\int_{t_n}^s \nabla^3 f(x_{t_n})\big(\E_n\big[\m\wv_r - v_r\big], v_{t_n}\big)\,dr\,ds - \E_n\big[R_n(t)\big]\\[3pt]
&\mmm - \int_{t_n}^t\int_{t_n}^s \E_n\Big[\m\nabla^3 f(x_{t_n})\big(v_r - v_{t_n},v_s - v_{t_n}\big)\Big]\,dr\,ds + \nabla^2 f(x_{t_n})\,\E_n\big[\m\wx_t - x_t\big].
\end{align*}
Since we can estimate the $\L_2$ norm of conditional expectations by Theorem \ref{thm:cond_lp_estimate} and bounded linear operators by Theorems \ref{thm:M2_lip_estimate} and \ref{thm:M3_lip_estimate}, we have
\begin{align*}
&\Big\|\m\E_n\Big[\nabla f(\wx_t) - \nabla f(x_t)\Big]\Big\|_{\L_2}\\[3pt]
&\mm \leq  \int_{t_n}^t\int_{t_n}^s \big\|\m\nabla^3 f(x_{t_n})\big(\wv_r - v_{t_n} + \sigma\big(H_n + 6K_n\big), \wv_s - v_{t_n} + \sigma\big(H_n + 6K_n\big)\big)\big\|_{\L_2}\,dr\,ds\\[3pt]
&\mmmm + \int_{t_n}^t\int_{t_n}^s \big\|\m\nabla^3 f(x_{t_n})\big(v_{t_n}, \E_n\big[\m\wv_s - v_s\big]\big)\big\|_{\L_2}\,dr\,ds + \big\|R_n(t)\big\|_{\L_2} \\[3pt]
&\mmmm + \int_{t_n}^t\int_{t_n}^s \big\|\m\nabla^3 f(x_{t_n})\big(\E_n\big[\m\wv_r - v_r\big], v_{t_n}\big)\big\|_{\L_2}\,dr\,ds + \big\|\wideparen{R}_n(t)\big\|_{\L_2}\\[3pt]
&\mmmm + \int_{t_n}^t\int_{t_n}^s \big\|\m\nabla^3 f(x_{t_n})\big(v_r - v_{t_n},v_s - v_{t_n}\big)\big\|_{\L_2}\,dr\,ds + \big\|\nabla^2 f(x_{t_n})\,\E_n\big[\m\wx_t - x_t\big]\big\|_{\L_2}\\[3pt]
&\mm \leq M_2 \int_{t_n}^t\int_{t_n}^s \big\|\m\wv_r - v_{t_n} + \sigma\big(H_n + 6K_n\big)\big\|_{\L_4}\big\|\m\wv_s - v_{t_n} + \sigma\big(H_n + 6K_n\big)\big)\big\|_{\L_4}\,dr\,ds\\[3pt]
&\mmmm + M_2\int_{t_n}^t\int_{t_n}^s \big\|v_{t_n}\big\|_{\L_4}\big\|\m\E_n\big[\m\wv_s - v_s\big]\big\|_{\L_4}\,dr\,ds + \big\|R_n(t)\big\|_{\L_2}\\[3pt]
&\mmmm + M_2\int_{t_n}^t\int_{t_n}^s \big\|v_{t_n}\big\|_{\L_4}\big\|\m\E_n\big[\m\wv_r - v_r\big]\big\|_{\L_4}\,dr\,ds + \big\|\wideparen{R}_n(t)\big\|_{\L_2}\\[3pt]
&\mmmm + M_2\int_{t_n}^t\int_{t_n}^s \big\|v_r - v_{t_n}\big\|_{\L_4}\big\|v_s - v_{t_n}\big\|_{\L_4}\,dr\,ds + M\big\|\m\E_n\big[\m\wx_t - x_t\big]\big\|_{\L_2}\\[3pt]
&\mm \leq \frac{2}{9}M_2\m C_{\wideparen{v}}(4)\m C_{\wideparen{v}}(4)\m d\m (t-t_n)^3 + \frac{14}{25}M_2\m C_{\wideparen{v}}(4)\m C_K(4)\m d\m (h_n)^\frac{1}{2}(t-t_n)^\frac{5}{2}\\[3pt]
&\mmmm + \frac{1}{2}M_2\big(C_K(4)\big)^2\m d\m h_n(t-t_n)^2 + \frac{1}{3}M_2\m C_5(4)\m C_{|v|}(4)\m d\m (h_n)^\frac{1}{2} (t-t_n)^4 \\[3pt]
&\mmmm + \frac{1}{6}M_2\big(C_v(4)\big)^2\m d\m (t-t_n)^3 + \frac{1}{3}M C_5(p)\sqrt{d\m}\m (h_n)^\frac{1}{2} (t-t_n)^3\\[3pt]
&\mmmm + \Big(C_8^{(1)} + C_8^{(2)}\Big)d^\frac{3}{2}(t-t_n)^3,
\end{align*}
where we used estimates from Theorems \ref{thm:inside_estimate}, \ref{thm:inside_mean_estimate}, \ref{thm:global_bounds}, \ref{thm:growth_bounds} and \ref{thm:high_order_taylor} in the last line.
\end{proof}

\begin{theorem}\label{thm:third_local_estimate}
For $n\geq 0$ and $\lambda\in [0,\gamma]$, we have
\begin{align}
\Big\|\m\E_n\Big[\big(\lambda\m x_{t_{n+1}} + v_{t_{n+1}}\big) - \big(\lambda\m x_{t_{n+1}}^\prime + v_{t_{n+1}}^\prime\big)\Big]\Big\|_{\L_2} \leq C_{10}(\lambda)\m d^\frac{3}{2}\m (h_n)^4,\label{eq:third_local_estimate}
\end{align}
where the constant $C_{10}(\lambda)$ is given by
\begin{align*}
C_{10}(\lambda) := \frac{1}{4}\m u\m C_8  & + \frac{1}{3} u\m C_9\m d^{\m-\frac{1}{2}} + \frac{1}{24}\m\gamma uM\big(\gamma - \lambda \big)C_3(2)d^{\m -1}\m(h_n)^\frac{1}{2}\\[3pt]
& + u\big(\gamma - \lambda \big)\bigg(\frac{1}{20} C_8 + \frac{1}{12}C_9\, d^{\m-\frac{1}{2}}\bigg)h_n + \frac{1}{60}\m\gamma^2\big(\gamma - \lambda \big)C_5(2)\m d^{\m -1}\m (h_n)^\frac{3}{2}.
\end{align*}
\end{theorem}
\begin{proof}
Just as for Theorem \ref{thm:second_local_estimate}, we shall use the expansion (\ref{eq:key_expansion}) of $(\lambda x + v) - (\lambda x^\prime + v^\prime)$.
\begin{align*}
&\big(\lambda\m x_{t_{n+1}} + v_{t_{n+1}}\big) - \big(\lambda\m x_{t_{n+1}}^\prime + v_{t_{n+1}}^\prime\big)\\[3pt]
&\mm = \big(\gamma - \lambda\big)\Bigg(\gamma^2 \int_{t_n}^{t_{n+1}}\nm\int_{t_n}^t\,\int_{t_n}^s \big(v_r - \wv_r\big)\,dr\,ds\,dt + u \int_{t_n}^{t_{n+1}}\nm\int_{t_n}^t\big(\m\nabla f(x_s) - \nabla f(\wx_s)\big)\,ds\,dt\\
&\hspace{26.5mm} + \frac{1}{6}\gamma^2\sigma (h_n)^3\big(H_n + 6K_n\big) + \gamma u\int_{t_n}^{t_{n+1}}\nm\int_{t_n}^t\,\int_{t_n}^s \big(\m\nabla f(x_r) - \nabla f(\wx_r)\big)\,dr\,ds\,dt\Bigg)\\
&\mmmm + u\int_{t_n}^{t_{n+1}}\big(\m\nabla f(\wx_t) - \nabla f(x_t)\big)\,dt.
\end{align*}
Using Minkowski's inequality and the Lipschitz regularity of $\nabla f$, we have the $\L_2$ estimate:
\begin{align*}
&\Big\|\m\E_n\Big[\big(\lambda\m x_{t_{n+1}} + v_{t_{n+1}}\big) - \big(\lambda\m x_{t_{n+1}}^\prime + v_{t_{n+1}}^\prime\big)\Big]\Big\|_{\L_2}\\[3pt]
&\mm \leq \gamma^2\big(\gamma - \lambda \big)\int_{t_n}^{t_{n+1}}\nm\int_{t_n}^t\,\int_{t_n}^s \big\|\m\E_n\big[v_r - \wv_r\big]\big\|_{\L_2}\,dr\,ds\,dt\\[3pt]
&\mmmm + u\big(\gamma - \lambda \big)\int_{t_n}^{t_{n+1}}\nm\int_{t_n}^t\big\|\m\E_n\big[\nabla f(x_s) - \nabla f(\wx_s)\big]\big\|_{\L_2}\,ds\,dt\\[3pt]
&\mmmm + \gamma uM\big(\gamma - \lambda \big)\int_{t_n}^{t_{n+1}}\nm\int_{t_n}^t\,\int_{t_n}^s \big\|x_r - \wx_r\big\|_{\L_2}\,dr\,ds\,dt\\[3pt]
&\mmmm + u\int_{t_n}^{t_{n+1}}\big\|\m\E_n\big[\nabla f(\wx_t) - \nabla f(x_t)\big]\big\|_{\L_2}\,dt.
\end{align*}
Therefore by Theorem \ref{thm:fx_third_estimate} along with the local estimates (\ref{eq:inside_estimate_x}) and (\ref{eq:inside_mean_estimate_v}), it follows that
\begin{align*}
&\Big\|\m\E_n\Big[\big(\lambda\m x_{t_{n+1}} + v_{t_{n+1}}\big) - \big(\lambda\m x_{t_{n+1}}^\prime + v_{t_{n+1}}^\prime\big)\Big]\Big\|_{\L_2}\\[3pt]
&\mm \leq \frac{1}{4}u\m C_8\m d^\frac{3}{2}\m (h_n)^4 + \frac{1}{3}u\m C_9\m d\m (h_n)^4 + \frac{1}{24}\gamma uM\big(\gamma - \lambda \big)C_3(2)\sqrt{d\m}\m(h_n)^\frac{9}{2}\\[3pt]
&\mmmm + \frac{1}{20}u\big(\gamma - \lambda \big) C_8\,d^\frac{3}{2}\m (h_n)^5 + \frac{1}{12}u\big(\gamma - \lambda \big) C_9\, d\m (h_n)^5\\[3pt]
&\mmmm + \frac{1}{60}\gamma^2\big(\gamma - \lambda \big)C_5(2)\sqrt{d\m}\m (h_n)^\frac{11}{2},
\end{align*}
which gives the desired result.
\end{proof}
\begin{theorem}\label{thm:third_local_estimate_y} Let $\big\{y_n^\prime\big\}$ be the error process given by Theorem \ref{thm:error_propagation}. Then for $n\geq 0$,
\begin{align}
\big\|\m\E_n\big[y_{n+1}^\prime\big]\big\|_{\L_2} \leq C_{10}\m d^\frac{3}{2}\m (h_n)^4,\label{eq:third_local_estimate_y}
\end{align}
where the constant $C_{10}$ is defined as
\begin{align*}
C_{10} := \frac{1}{2}\m u\m C_8  & + \frac{2}{3} u\m C_9\m d^{\m-\frac{1}{2}} + \frac{1}{24}\m\gamma^2 uM C_3(2)d^{\m -1}\m(h_n)^\frac{1}{2}\\[3pt]
& + u\gamma\bigg(\frac{1}{20} C_8 + \frac{1}{12}C_9\, d^{\m-\frac{1}{2}}\bigg)h_n + \frac{1}{60}\m\gamma^3 C_5(2)\m d^{\m -1}\m (h_n)^\frac{3}{2}.
\end{align*}
\end{theorem}
\begin{proof}
The result follows immediately from Minkowski's inequality and Theorem \ref{thm:third_local_estimate}.
\end{proof}

Using this theorem, we arrive at our final global error estimate for the shifted ODE.
As stated before, we shall follow the same strategy of the proof as for Theorem 1 in \cite{RungeKuttaMCMC}.

\begin{theorem}[Global error estimate for the shifted ODE with Lipschitz Hessian and third derivative]\label{thm:const_step_high_estimate}
Suppose $\big\{(\widetilde{x}_n\m, \widetilde{v}_n)\big\}_{n\m \geq\m 0}$ was obtained using a constant step size $h > 0$.
Let $\big\{y_n\big\}_{n\m \geq\m 0}$ be the error process between $(\m\widetilde{x}, \widetilde{v}\m)$ and the true diffusion $(x, v)$ defined in Theorem \ref{thm:exp_contract}.
Then under the assumptions detailed at the start of this subsection, we have
\begin{align}\label{eq:const_step_high_estimate}
\big\|y_n\big\|_{\L_2}^2 \leq e^{-n\alpha h}\big\|y_0\big\|_{\L_2}^2 + \frac{1 - e^{-n\alpha h}}{1-e^{-\alpha h}}\,C_{11}\m d^{\m 3}\m h^7\m,
\end{align}
for $n\geq 0$, where the constant $C_{11}$ is defined as
\begin{align*}
C_{11} & := \frac{2}{\alpha}\big(C_{10}\big)^2 + \big(C_8\big)^2d^{\m - 1} + \frac{2}{\alpha}e^{2\alpha h}\big(C_1\big)^2\m h\m d^{\m - 1},
\end{align*}
and contraction rate $\alpha$ is given by
\begin{align*}
\alpha & = \frac{\big((\gamma - \lambda)^2 - uM\big)\vee \big(um - \lambda^2\big)}{\gamma - 2\lambda}\,.
\end{align*}
\end{theorem}
\begin{proof}
Applying Theorem \ref{thm:globalestimatetrick} to $X = y_{n+1}^\prime, Y = \widetilde{y}_{n+1}, Z = y_n$ and $\F = \F_n$ produces
\begin{align*}
\big\|y_{n+1}^\prime+\widetilde{y}_{n+1}\big\|_{\L_2}^2 & \leq \big\|y_{n+1}^\prime\big\|_{\L_2}^2 + 2\m\big\|y_{n+1}^\prime - \E_n\big[y_{n+1}^\prime\big]\big\|_{\L_2}\big\|\widetilde{y}_{n+1} - y_n\big\|_{\L_2}\\
&\mmm + \frac{c}{h}\big\|\m\E_n\big[y_{n+1}^\prime\big]\big\|_{\L_2}^2 + \Big(1 + \frac{h}{c}\Big)\big\|\widetilde{y}_{n+1}\big\|_{\L_2}^2.
\end{align*}
Note that $y_{n+1}^\prime + \widetilde{y}_{n+1} = y_{n+1}$ and by Theorem \ref{thm:cond_lp_estimate} along with the Tower law, we have
\begin{align*}
\big\|y_{n+1}^\prime - \E_n\big[y_{n+1}^\prime\big]\big\|_{\L_2}^2 & = \big\|y_{n+1}^\prime\big\|_{\L_2}^2 + \big\|\E_n\big[y_{n+1}^\prime\big]\big\|_{\L_2}^2 - 2\m\E\Big[\big\langle y_{n+1}^\prime, \E_n\big[ y_{n+1}^\prime\big]\big\rangle\Big]\\[3pt]
& = \big\|y_{n+1}^\prime\big\|_{\L_2}^2 + \big\|\E_n\big[y_{n+1}^\prime\big]\big\|_{\L_2}^2 - 2\m\E\Big[\E_n\Big[\big\langle y_{n+1}^\prime, \E_n\big[ y_{n+1}^\prime\big]\big\rangle\Big]\Big]\\[3pt]
& = \big\|y_{n+1}^\prime\big\|_{\L_2}^2 - \big\|\E_n\big[y_{n+1}^\prime\big]\big\|_{\L_2}^2\\[3pt]
& \leq \big\|y_{n+1}^\prime\big\|_{\L_2}^2.
\end{align*}
Therefore, it follows that
\begin{align*}
\big\|y_{n+1}\big\|_{\L_2}^2 & \leq \big\|y_{n+1}^\prime\big\|_{\L_2}^2 + 2\m\big\|y_{n+1}^\prime\big\|_{\L_2}\big\|\widetilde{y}_{n+1} - y_n\big\|_{\L_2} + \frac{c}{h}\big\|\m\E_n\big[y_{n+1}^\prime\big]\big\|_{\L_2}^2 + \Big(1 + \frac{h}{c}\Big)\big\|\widetilde{y}_{n+1}\big\|_{\L_2}^2.
\end{align*}
Applying previous results (Theorems \ref{thm:exp_contract}, \ref{thm:error_flow}, \ref{thm:second_local_estimate_y} and \ref{thm:third_local_estimate_y}) yields the local estimate:
\begin{align*}
\big\|y_{n+1}\big\|_{\L_2}^2 & \leq \bigg(\big(C_8\big)^2 d^{\m 2} + c\m\big(C_{10}\big)^2 d^{\m 3}\bigg)h^7 + 2\m C_1\m C_8\m d\m h^\frac{9}{2}\big\|y_n\big\|_{\L_2} + \Big(1 + \frac{h}{c}\Big)e^{-2\alpha h}\big\|y_n\big\|_{\L_2}^2.
\end{align*}
By the Cauchy-Schwarz inequality, we have
\begin{align*}
2\m C_1\m C_8\m d\m h^\frac{9}{2}\big\|y_n\big\|_{\L_2} & = 2\m\bigg(\sqrt{\widetilde{c}\,}\m C_1\m C_8\m d\m h^4\bigg)\Bigg(\frac{\sqrt{h}}{\sqrt{\widetilde{c}\,}}\big\|y_n\big\|_{\L_2}\Bigg)\\[3pt]
& \leq \widetilde{c}\m\big(C_1\m C_8\big)^2\m d^{\m 2}\m h^8 + \frac{h}{\widetilde{c}}\big\|y_n\big\|_{\L_2}^2,
\end{align*}
where $\widetilde{c} > 0$ is some constant to be determined. Hence
\begin{align*}
\big\|y_{n+1}\big\|_{\L_2}^2 & \leq \bigg(\Big(1 + \widetilde{c}\m\big(C_1\big)^2\m h\Big)\big(C_8\big)^2 d^{\m 2} + c\m\big(C_{10}\big)^2 d^{\m 3}\bigg)h^7 + \bigg(\Big(1 + \frac{h}{c}\Big)e^{-2\alpha h} + \frac{h}{\widetilde{c}}\bigg)\big\|y_n\big\|_{\L_2}^2.
\end{align*}
Setting $c = \frac{2}{\alpha}$ and $\widetilde{c} = \frac{2}{\alpha}e^{2\alpha h}$ gives
\begin{align*}
\big\|y_{n+1}\big\|_{\L_2}^2 & \leq \bigg(\Big(1 + \frac{2}{\alpha}e^{2\alpha h}\big(C_1\big)^2\m h\Big)\big(C_8\big)^2 d^{\m 2} + \frac{2}{\alpha}\big(C_{10}\big)^2 d^{\m 3}\bigg)h^7 + \Big(1 + \alpha h\Big)e^{-2\alpha h}\big\|y_n\big\|_{\L_2}^2\\[3pt]
& \leq \bigg(\Big(1 + \frac{2}{\alpha}e^{2\alpha h}\big(C_1\big)^2\m h\Big)\big(C_8\big)^2 d^{\m 2} + \frac{2}{\alpha}\big(C_{10}\big)^2 d^{\m 3}\bigg)h^7 + e^{-\alpha h}\big\|y_n\big\|_{\L_2}^2.
\end{align*}
Finally, we can apply the above inequality $n$ times to derive the required estimate for $y_n$.
\begin{align*}
\big\|y_n\big\|_{\L_2}^2 & \leq e^{-\alpha h}\big\|y_{n-1}\big\|_{\L_2}^2 + C_{11}\m d^{\m 3}\m h^7\\
&\hspace{2mm}\vdots\\[-10pt]
& \leq e^{-n\alpha h}\big\|y_0\big\|_{\L_2}^2 + \sum_{k=0}^{n-1}\Big(e^{-k\alpha h}C_{11}\m d^{\m 3}\m h^7\Big),
\end{align*}
and the result follows.
\end{proof}
Just as in the previous subsection, we can deduce the convergence rate for the ODE.

\begin{corollary}\label{cor:third_order}
Suppose that the underdamped Langevin diffusion $\big\{(x_t, v_t)\big\}_{t\geq 0}$ and the shifted ODE approximation $\big\{(\widetilde{x}_n\m, \widetilde{v}_n)\big\}_{n\m \geq\m 0}$ have the same initial velocity. Then for $n\geq 0$,
\begin{align}\label{eq:third_order}
\big\|y_n\big\|_{\L_2} \leq \sqrt{\lambda^2 + (\gamma - \lambda)^2}\,e^{-\frac{1}{2}n\alpha h}\big\|\widetilde{x}_0 - x_0\big\|_{\L_2} + \bigg(\Big(\m\frac{1}{\alpha} + h\Big)C_{11}\bigg)^\frac{1}{2} d^{\m\frac{3}{2}}\m h^3.
\end{align}
\end{corollary}
\begin{proof} Applying the inequalities $e^{-n\alpha h}\geq 0$ and $\alpha h \leq (1+\alpha h)(1-e^{-\alpha h})$ to (\ref{eq:const_step_high_estimate}) yields
\begin{align*}
\big\|y_n\big\|_{\L_2}^2 \leq e^{-n\alpha h}\big\|y_0\big\|_{\L_2}^2 + \Big(\m\frac{1}{\alpha} + h\Big)\m C_{11}\m d^{\m 3}\m h^6\m.
\end{align*}
Since we assume that $\widetilde{v}_0 = v_0$ (which is achievable in practice as $v_0 \sim \mathcal{N}(0, u\m I_d)$), we have
\begin{align*}
\big\|y_0\big\|_{\L_2}^2 & = \big\|\big(\lambda\m\widetilde{x}_0 + \widetilde{v}_0\big) - \big(\lambda\m x_0 + v_0\big)\big\|_{\L_2}^2 + \big\|\big((\gamma - \lambda)\m\widetilde{x}_0 + \widetilde{v}_0\big) - \big((\gamma - \lambda)\m x_0 + v_0\big)\big\|_{\L_2}^2\\[3pt]
& = \big\|\lambda\m(\widetilde{x}_0 - x_0)\big\|_{\L_2}^2 + \big\|(\gamma - \lambda)(\widetilde{x}_0 - x_0)\big\|_{\L_2}^2\\[3pt]
& = \Big(\lambda^2 + (\gamma - \lambda)^2\Big)\big\|\widetilde{x}_0 - x_0\big\|_{\L_2}^2.
\end{align*}
Substituting this into the above inequality gives
\begin{align*}
\big\|y_n\big\|_{\L_2}^2 & \leq \big(\lambda^2 + (\gamma - \lambda)^2\m\big)\m e^{-n\alpha h}\big\|\widetilde{x}_0 - x_0\big\|_{\L_2}^2 + \Big(\m\frac{1}{\alpha} + h\Big)C_{11}\m d^{\m 3}\m h^6\m\\[3pt]
& \leq \Bigg(\sqrt{\lambda^2 + (\gamma - \lambda)^2}\,e^{-\frac{1}{2}n\alpha h}\big\|\widetilde{x}_0 - x_0\big\|_{\L_2} + \bigg(\Big(\m\frac{1}{\alpha} + h\Big)C_{11}\bigg)^\frac{1}{2} d^{\m\frac{3}{2}}\m h^3\Bigg)^2,
\end{align*}
and the result follows.
\end{proof}

%
%
%
%

\section{Generating integrals of Brownian motion}\label{appen:integrals}

In this section, we shall detail how the stochastic integrals used by the numerical methods for underdamped Langevin dynamics can be generated. In particular, we discuss how one can ``combine'' integrals defined over intervals $[s,u]$ and $[u,t]$ to give integrals over $[s,t]$.
To begin, we will consider the integrals $\{W_n\m,\m H_n\m,\m K_n\}_{n\m\geq\m 0}$ used by the proposed methods.
\begin{definition}
Let $W$ be a $d$-dimensional Brownian motion. For $0\leq s\leq t$, we define
\begin{align*}
H_{s,t} & := \frac{1}{t-s}\int_s^t \Big(W_{s,r} - \frac{r - s}{t-s}\,W_{s,t}\Big)dr\m,\\[3pt]
K_{s,t} & := \frac{1}{(t-s)^2}\int_s^t \Big(W_{s,r} - \frac{r - s}{t-s}\,W_{s,t}\Big)\bigg(\frac{1}{2}(t-s) - (r-s)\bigg)dr\m,\\[3pt]
M_{s,t} & := \int_s^t W_{s,r}\,dr\m,\\[3pt]
N_{s,t} & := \int_s^t (r-s)W_{s,r}\,dr\m.
\end{align*}
\end{definition}
Recall that by Theorem \ref{thm:hk_st}, we have that $W_{s,t}\sim\mathcal{N}\big(0, h I_d\big)$, $H_{s,t}\sim\mathcal{N}\big(0, \frac{1}{12}h I_d\big)$, $K_{s,t}\sim\mathcal{N}\big(0, \frac{1}{720}h I_d\big)$ are independent. In our numerical experiment, we generate and ``combine'' the triples $\big(W_{s,u}\m,\m H_{s,u}\m,\m K_{s,u}\big)$ and $\big(W_{u,t}\m,\m H_{u,t}\m,\m K_{u,t}\big)$ to give $\big(W_{s,t}\m,\m H_{s,t}\m,\m K_{s,t}\big)$.

\begin{theorem}[Combining $(W, H, K)$ on neighbouring intervals]
Let $0\leq s\leq t$. Then
\begin{align*}
W_{s,t} & = W_{s,u} + W_{u,t}\m,\\[2pt]
M_{s,t} & = M_{s,u} + M_{u,t} + (t-u)W_{s,u}\m,\\
N_{s,t} & =  N_{s,u} + N_{u,t} + (u-s)M_{u,t} + \bigg(\frac{1}{2}(t-u)^2 + (t-u)(u-s)\bigg)W_{s,u}\m.
\end{align*}
for $u\in[s,t]$. In addition, $(W\m, H\m , K)_{s,t}$ can be mapped to $(W\m, M\m, N\m)_{s,t}$ and vice versa as
\begin{align*}
M_{s,t} & =  \frac{1}{2}(t-s)W_{s,t} + (t-s) H_{s,t}\m,\\[3pt]
N_{s,t} & = \frac{1}{3}(t-s)^2\m W_{s,t} + \frac{1}{2} (t-s)^2 H_{s,t} - (t-s)^2 K_{s,t}\m.
\end{align*}
\end{theorem}
\begin{proof}
The first identity is trivial as $W_{s,t} = W_t - W_s\m$, $W_{s,u} = W_u - W_s$ and $W_{u,t} = W_t - W_u\m$.
The identities for $M_{s,t}$ and $N_{s,t}$ can be shown by direct calculation.
\begin{align*}
M_{s,t} & = \int_s^u W_{s,r}\,dr + \int_u^t W_{s,r}\,dr\\[2pt]
& = \int_s^u W_{s,r}\,dr + \int_u^t W_{s,u}\,dr + \int_u^t W_{u,r}\,dr\\
& = M_{s,u} + (t-u)W_{s,u} + M_{u,t}\m.\\[6pt]
N_{s,t} & = \int_s^u (r-s)W_{s,r}\,dr + \int_u^t (r-s)W_{s,r}\,dr\\[2pt]
& = \int_s^u (r-s)W_{s,r}\,dr + \int_u^t (r-s)W_{s,u}\,dr + \int_u^t (u-s)W_{u,r}\,dr + \int_u^t (r-u)W_{u,r}\,dr\\
& = N_{s,u} + W_{s,u}\bigg(\frac{1}{2}(t-u)^2 + (t-u)(u-s)\bigg) + (u-s)M_{u,t} + N_{u,t}\m.
\end{align*}
It is also straightforward to express $M_{s,t}$ and $N_{s,t}$ in terms of $W_{s,t}$, $H_{s,t}$ and $K_{s,t}$ since
\begin{align*}
(t-s) H_{s,t} & =  \int_s^t W_{s,r}\, dr -  \int_s^t \frac{r - s}{t-s}\, W_{s,t}\, dr = M_{s,t} - \frac{1}{2}(t-s)W_{s,t}\m,\\[6pt]
(t-s)^2K_{s,t} & = \int_s^t \Big(W_{s,r} - \frac{r - s}{t-s}\,W_{s,t}\Big)\bigg(\frac{1}{2}(t-s) - (r-s)\bigg)dr\\[3pt]
& = \frac{1}{2}\int_s^t (t-s)W_{s,r}\,dr - \int_s^t (r-s)W_{s,r}\,dr\\[2pt]
&\mmm - \frac{1}{2}\int_s^t (r-s)W_{s,t}\,dr + \int_s^t \frac{(r - s)^2}{t-s}\,W_{s,t}\,dr\\[3pt]
& = \frac{1}{2}(t-s)M_{s,t} - N_{s,t} + \frac{1}{12}(t-s)^2 W_{s,t}\m.
\end{align*}
Substituting $M_{s,t} = (t-s) H_{s,t} + \frac{1}{2}(t-s)W_{s,t}$ into the above gives the final identity.
\end{proof}

We now consider the stochastic integrals used by the numerical methods discussed in Section \ref{sect:other_methods}. These integrals are either of the form $\int_s^t e^{-\gamma(t - \tau)}dW_\tau$ or $\int_s^t\int_s^\tau e^{-\gamma(\tau - r)}dW_r\, d\tau$.
\begin{theorem}[Distribution of integrals required to simulate physical Brownian motion]\label{thm:generate_int} Let $0\leq s\leq t$. Then
\begin{align*}
\begin{pmatrix}
\int_s^t e^{-\gamma(t - \tau)}dW_\tau\\
\int_s^t\int_s^\tau e^{-\gamma(\tau - r)}dW_r\, d\tau
\end{pmatrix} & \sim \mathcal{N}\left(\Bigg(\begin{matrix}0\\[-3pt] 0\end{matrix}\Bigg)\,,\begin{pmatrix} \frac{1-e^{-2\gamma h}}{2\gamma}\m I_d & \frac{(1 - e^{-\gamma h})^2}{2\gamma^2}\m I_d\\[3pt]
\frac{(1 - e^{-\gamma h})^2}{2\gamma^2}\m I_d &  \frac{4e^{-\gamma h} - e^{-2\gamma h} + 2\gamma h - 3}{2\gamma^3}\m I_d\end{pmatrix}\right),
\end{align*}
where $h := t -s$.
\end{theorem}
\begin{proof} Since each coordinate of $W$ is an independent Brownian motion, it suffices to show the result when $d = 1$. Since $\int_s^t e^{-\gamma(t - \tau)}dW_\tau$ and $\int_s^t\int_s^\tau e^{-\gamma(\tau - r)}dW_r\, d\tau$ can be obtained as linear functionals of the same Brownian motion, they follow a (joint) normal distribution.
Moreover, by the symmetry of Brownian motion, we see that they both have mean zero. Using It\^{o}'s isometry, the covariance matrix for the integrals is straightforward to compute.
\begin{align*}
\E\Bigg[\Bigg(\int_s^t e^{-\gamma(t - \tau)}dW_\tau\Bigg)^2\,\Bigg] & = \int_s^t e^{-2\gamma(t - \tau)}\m d\tau = \frac{1-e^{-2\gamma h}}{2\gamma}\,.\\[-26pt]
\end{align*}
\begin{align*}
\E\Bigg[\Bigg(\int_s^t\int_s^\tau e^{-\gamma(\tau - r)}dW_r\, d\tau\Bigg)^2\,\Bigg]
& =\int_s^t\int_s^t\E\Bigg[\int_s^{\tau_1} e^{-\gamma(\tau_1 - r)}dW_r\int_s^{\tau_2} e^{-\gamma(\tau_2 - r)}dW_r\Bigg]\m d\tau_2\, d\tau_1\\
& = 2\m\int_s^t\int_s^{\tau_1}\nnm\int_s^{\tau_2} e^{-\gamma(\tau_1 + \tau_2 - 2r)}\m dr \, d\tau_2\, d\tau_1\\
& = \frac{4e^{-\gamma h} - e^{-2\gamma h} + 2\gamma h - 3}{2\gamma^3}\,.\\[-26pt]
\end{align*}
\begin{align*}
\E\Bigg[\int_s^t e^{-\gamma(t - \tau)}dW_\tau\hspace{-0.25mm}\int_s^t\int_s^\tau e^{-\gamma(\tau - r)}dW_r\, d\tau\,\Bigg] & = \int_s^t\E\Bigg[\int_s^t e^{-\gamma(t - r)}dW_r\int_s^\tau e^{-\gamma(\tau - r)}dW_r\Bigg]\m d\tau\\
& = \int_s^t\int_s^\tau e^{-\gamma(t + \tau - 2r)}\m dr\, d\tau\\
& = \frac{(1 - e^{-\gamma h})^2}{2\gamma^2}\,.
\end{align*}
Using Fubini's theorem, we could have also written $\int_s^t\int_s^\tau e^{-\gamma(\tau - r)}dW_r\, d\tau = \int_s^t \frac{1 - e^{-\gamma(t - r)}}{\gamma}\,dW_r$ before applying It\^{o}'s isometry (which may have slightly simplified the calculations). 
\end{proof}

\begin{corollary} The stochastic integral $\int_s^t\int_s^\tau e^{-\gamma(\tau - r)}dW_r\, d\tau$ can be expressed as\label{cor:generate_int}
\begin{align*}
\int_s^t\int_s^\tau e^{-\gamma(\tau - r)}dW_r\, d\tau & = \frac{1-e^{-\gamma h}}{\gamma\m(1+e^{-\gamma h})}\int_s^t e^{-\gamma(t - \tau)}dW_\tau + X_{s,t}\m,
\end{align*}
where
\begin{align*}
X_{s,t} \sim \mathcal{N}\Bigg(0,\Bigg(\frac{4e^{-\gamma h} - e^{-2\gamma h} + 2\gamma h - 3}{2\gamma^3} - \frac{(1-e^{-\gamma h})^3}{2\gamma^3\m(1+e^{-\gamma h})}\Bigg)I_d\Bigg),
\end{align*}
is independent of $\int_s^t e^{-\gamma(t - \tau)}dW_\tau\sim \mathcal{N}\Big(0,\frac{1-e^{-2\gamma h}}{2\gamma}I_d\Big)$.
\end{corollary}

In our numerical experiment, both of these integrals will be generated over the intervals $[s,u]$, $[u,t]$ and $[s,t]$. Therefore, just as before, we require a method for ``combining'' them. 
\begin{theorem}[Combining integrals on neighbouring intervals]\label{thm:combine_int}
Let $0\leq s\leq u\leq t$. Then
\begin{align*}
\int_s^t e^{-\gamma(t-r)}\,dW_r & = e^{-\gamma(t-u)}\int_s^u e^{-\gamma(u-r)}\,dW_r + \int_u^t e^{-\gamma(t-r)}\,dW_r\m,\\[3pt]
\int_s^t\int_s^\tau e^{-\gamma(\tau - r)}dW_r\, d\tau & = \int_s^u\int_s^\tau e^{-\gamma(\tau - r)}dW_r\, d\tau + \int_u^t\int_u^\tau e^{-\gamma(\tau - r)}dW_r\, d\tau\\[1pt]
&\mmm + \bigg(\frac{1 - e^{-\gamma(t-u)}}{\gamma}\bigg)\int_s^u e^{-\gamma(u-r)}\,dW_r\m.
\end{align*}
\end{theorem}
\begin{proof} The first identity is trivial as $e^{-\gamma(t-r)} = e^{-\gamma(t-u)}\m e^{-\gamma(u-r)}$. For the second, we have
\begin{align*}
\int_s^t\int_s^\tau e^{-\gamma(\tau - r)}dW_r\, d\tau & = \int_s^u\int_s^\tau e^{-\gamma(\tau - r)}dW_r\, d\tau + \int_u^t\int_s^\tau e^{-\gamma(\tau - r)}dW_r\, d\tau\\[3pt]
& = \int_s^u\int_s^\tau e^{-\gamma(\tau - r)}dW_r\, d\tau + \int_u^t\int_u^\tau e^{-\gamma(\tau - r)}dW_r\, d\tau\\
&\mmm + \int_u^t\int_s^u e^{-\gamma(\tau - r)}dW_r\, d\tau.
\end{align*}
The result now follows as $\displaystyle\int_u^t\int_s^u e^{-\gamma(\tau - r)}dW_r\, d\tau = \int_u^t e^{-\gamma(\tau - u)}\,d\tau\int_s^u e^{-\gamma(u-r)}\,dW_r\m.$
\end{proof}

In the numerical experiment for the randomized midpoint method, it will also be necessary to generate uniform random variables within the intervals $[s,u]$, $[u,t]$ and $[s,t]$.
To simply the generation of subsequent stochastic integrals, we use the following procedure:\medbreak

{\textbf{Step 1.}} Generate independent randomized midpoints $x\sim U[s,u]$ and $y\sim U[u,t]$.\smallbreak

{\textbf{Step 2.}} Generate an independent Rademacher random variable $r$.\medbreak

{\textbf{Step 3.}} Let $z := \Bigg\{\begin{matrix}x\,\,\,\text{if}\,\,\, r = +1\\[-3pt] y\,\,\,\text{if}\,\,\, r = -1\end{matrix}\,$.\hspace{-0.25mm} (i.e.\hspace{-0.5mm} $z\sim U[s,t]$ is the randomized midpoint for $[s,t]$).\medbreak

{\textbf{Step 4.}} Generate the pair $\big(\int_{\sbt}^{\sbt} e^{-\gamma(t - \tau)}dW_\tau, \int_{\sbt}^{\sbt}\int_{\sbt}^\tau e^{-\gamma(\tau - r)}dW_r\, d\tau\big)$ over the intervals
$\big([s,x], [x,u], [s,u]\big)$, $\big([u,y], [y, t], [u,t]\big)$ and $\big([s,z], [z,t], [s,t]\big)$ using Theorems \ref{thm:generate_int} and \ref{thm:combine_int}.

\end{document}